\def\paperversion{2}
\ifnum\paperversion=1
\documentclass[opre,blindrev]{informs4}
\fi
\ifnum\paperversion=2
\documentclass[opre]{informs4}
\fi
\OneAndAHalfSpacedXI

\usepackage{silence}
\usepackage{endnotes}
\usepackage[toc,page]{appendix}

\usepackage{endnotes}
\let\footnote=\endnote

\usepackage{fix-cm}
\OneAndAHalfSpacedXII %
\usepackage{amsmath,amssymb,amsfonts}
\usepackage{amsmath}
\usepackage{bm}
\usepackage{algorithm}
\usepackage{algpseudocode}
\usepackage{tikz}
\usepackage{enumerate}
\usepackage{multirow}
\usepackage{subfigure}

\usepackage{natbib}
 \bibpunct[, ]{(}{)}{,}{a}{}{,}%
 \def\bibfont{\small}%

\usepackage{hyperref}
\usepackage{threeparttable}
\usepackage{xcolor}
\hypersetup{%
  colorlinks=true,%
  linkcolor={blue},%
  linkbordercolor=red,%
  citecolor={black}
}
\usepackage{cleveref}

\def\E{{\mathbb E}}

\def\Re{\mathbb{R}}

\def\hat{\widehat}

\def\O{{\mathcal O}}

\def\P{{\mathbb P}}

\def\Re{{\mathbb R}}
\def\S{{\mathcal S}}

\def\O{{\mathcal O}}

\newcommand{\exclude}[1]{}

\DeclareMathOperator{\conv}{conv}

\DeclareMathOperator{\Diag}{Diag}

\DeclareMathOperator{\tr}{tr}

\renewcommand*{\qed}{\hfill\ensuremath{\square}}
\newcommand*{\qedA}{\hfill\ensuremath{\diamond}}
\EquationsNumberedThrough    %

\TheoremsNumberedThrough     %
\ECRepeatTheorems  %
\EquationsNumberedThrough
\MANUSCRIPTNO{IJOC-0001-2024.00}

\usepackage{anyfontsize}
\usepackage{multirow}%

\newtheorem{innercustomgeneric}{\customgenericname}
\providecommand{\customgenericname}{}
\newcommand{\newcustomtheorem}[2]{%
  \newenvironment{#1}[1]
  {%
   \renewcommand\customgenericname{#2}%
   \renewcommand\theinnercustomgeneric{##1}%
   \innercustomgeneric
  }
  {\endinnercustomgeneric}
}
\newcustomtheorem{customthm}{Assumption}

\begin{document}

\RUNAUTHOR{Yongchun Li}

\RUNTITLE{Strong Formulations and Algorithms for Regularized A-optimal Design}

\TITLE{Strong Formulations and Algorithms for Regularized A-optimal Design}

\ARTICLEAUTHORS{
\AUTHOR{Yongchun Li}
\AFF{Department of Industrial and Systems Engineering, University of Tennessee, Knoxville, \EMAIL{ycli@utk.edu}}}

\ABSTRACT{We study the Regularized A-optimal Design (RAOD) problem, which selects a subset of $k$ experiments   to minimize the inverse of the Fisher information matrix, regularized with a scaled identity matrix.  RAOD has broad applications in Bayesian experimental design, sensor placement, and cold-start recommendation.  We prove its NP-hardness via a reduction from the independent set problem. By leveraging convex envelope techniques, 
we propose a new convex integer programming formulation for RAOD, 
whose continuous relaxation  dominates those of existing formulations. 
More importantly, we demonstrate that our continuous relaxation achieves bounded optimality gaps for all $k$, 
whereas previous relaxations may suffer from unbounded gaps.
This new formulation enables the development of an exact cutting-plane algorithm with superior efficiency, especially in high-dimensional and small-$k$ scenarios.
We also investigate  scalable forward and backward greedy algorithms for solving RAOD, each with provable performance guarantees for different $k$ ranges.
Finally, our numerical results on synthetic and real data
 demonstrate the efficacy of the proposed exact and approximation algorithms.
We further showcase the practical effectiveness of RAOD by applying it to a real-world user cold-start recommendation problem. 
}


\maketitle

\section{Introduction}\label{sec:Intro}
The A-optimal Design (AOD) problem arises in   statistical design for situations where, due to budget limitations, 
one can observe the outcomes for a small subset of experiments to optimize the parameter estimates. Specifically, AOD aims to minimize the trace of the inverse of the Fisher information matrix, resulting in minimizing the average variance of the parameter estimates.
AOD has found broad applications in fields such as engineering \citep{ilzarbe2008practical,mason2003statistical}, biology \citep{quinn2002experimental}, and  chemistry \citep{higgs1997experimental,leardi2009experimental}, when dealing with large-scale and costly experiments.
Linear regression models are one of the most widely used and studied models in this area (see, \citealt{jobson2012applied,nikolov2022proportional,winer1971statistical}).
This paper focuses on the  Regularized AOD (\ref{aed}) problem in the context of linear regression models,  which adds a scaled identity matrix to the Fisher information matrix, as shown below. This addition typically helps prevent over-fitting \citep{tantipongpipat2020lambda}, improves numerical stability when the information matrix is rank-deficient \citep{derezinski2017subsampling}, and encodes prior information \citep{bian2017guarantees}. 
\begin{align}\label{aed}
z_k:=\min_{S\subseteq [n]} \bigg\{\tr \bigg(\bigg(\sum_{i\in S} \bm a_i\bm a_i^{\top}+ \lambda \bm I_d\bigg)^{-1} \bigg): |S|=k \bigg\}, \tag{RAOD}
\end{align}
where $\{\bm a_i\}_{i\in [n]}\in \Re^d$ represents a collection of $n$ experiments parameterized by $d$-dimensional vectors with $d\le n$, $k\le n$ denotes the number of chosen experiments, $\lambda>0$ is a constant, $\bm I_d$ is a $d\times d$ identity matrix, and for a subset $S\subseteq [n]$, let $|S|$ denote its cardinality. Note that the selection size $k$ can be smaller or larger than the dimension $d$, depending on the specific problem setup.

The following outlines the two most relevant variants and an equivalent problem of \ref{aed}, which will be  used throughout to show the scope and limitations of existing research, facilitate our analysis, and highlight the general applicability of some of our results.
\begin{enumerate}[(i)]
\item \textit{AOD}. As previously mentioned, $\lambda$ is fixed at zero in AOD.
It is noted that the condition $k\ge d$ must hold for AOD to ensure that the objective matrix is invertible;
\item \textit{Bayesian AOD}. Since specific information is usually available prior to the  experimental selection process, Bayesian AOD, which leverages prior knowledge, can play an important role in improving design efficiency (see the excellent survey by \citealt{chaloner1995bayesian,rainforth2024modern}). Mathematically, it generalizes the scaled identity matrix in  \ref{aed} to an arbitrary positive definite matrix representing prior information.  Bayesian AOD is also referred to as the A-Fusion problem  in \cite{hendrych2023solving}.
When the prior is an isotropic Gaussian, Bayesian AOD reduces to \ref{aed} as a  special case (see, e.g., \citealt{bian2017guarantees,chamon2017approximate}); and
    \item  \textit{A-optimal Maximum Entropy Sampling Problem (AMESP)}. As  introduced in \cite{li2024best}, AMESP focuses on selecting the most informative principal submatrix from a given covariance matrix by minimizing the trace of its inverse. We are the first to demonstrate
    that \ref{aed} and AMESP are equivalent, with an additive additional constant of $(d-k)/\lambda$,  when the covariance matrix in AMESP is positive definite.
\end{enumerate}

Although \ref{aed} is a special case of Bayesian AOD, it remains broadly applicable. One key motivation comes from the user cold-start problem in recommendation systems. The problem arises when a new user joins a recommendation platform and has not rated any item, making recommendations challenging due to the lack of historical data \citep{lika2014facing}.

Traditional methods, such as content-based filtering and collaborative filtering, exploit relationships and similarities between users to make recommendations \citep{bobadilla2012collaborative,volkovs2017content}. Alternatively, experimental design approaches, including sequential strategies, tackle the problem at its root by explicitly asking new users to rate selected items upon joining  a recommendation platform (see, e.g.,  \citealt{anava2015budget,elahi2016survey,gope2017survey}).  These items are selected carefully to capture new users' preferences. However, users are often reluctant to rate many items due to cognitive costs. If users rate too few items, the Fisher information matrix becomes rank-deficient, leading to ill-posed matrix inverse and determinant problems in experimental design. 
To mitigate this, regularized experimental design has emerged as a practical  approach by adding a scaled identity matrix, as shown in prior research by \cite{barz2015nonlinear,anava2015budget,rubens2009output}.
In particular, \cite{chamon2017approximate} applied the \ref{aed} approach  and demonstrated its effectiveness in selecting movies for new users to rate.

\subsection{Related Work}
In this subsection, we review the relevant literature on \ref{aed} and its variants, with a focus on their exact algorithms, convex relaxations, and approximation algorithms.

\noindent\textit{Bayesian D-optimal Design (Bayesian DOD) \& Bayesian AOD.} 
In experimental designs,   the $D$- and A-optimality criteria are widely used for selecting experiments to achieve precise parameter estimation.
A-optimality minimizes the estimation variance directly, while D-optimality seeks to reduce the volume of the confidence ellipsoid of the parameter estimates by maximizing the determinant of the Fisher information matrix. Notably, given that the determinant is a multiplicative operator,  Bayesian DOD  is equivalent to regularized DOD \citep{li2024d}; however, this equivalence does not hold between Bayesian AOD and \ref{aed}.

A-optimality is known to be computationally more expensive than   D-optimality due to its need to compute the inverse matrix \citep{ahipacsaouglu2021branch}. 
As a result, algorithms and software for Bayesian DOD are more developed than Bayesian AOD (see \citealt{jones2021optimal}).
Unfortunately, these techniques cannot be directly applied to Bayesian AOD.  
Nevertheless, Bayesian AOD often outperforms  Bayesian DOD, with improved variance and bias properties (see, e.g., \citealt{jones2021optimal,stallrich2023d}).   Our numerical study on cold-start recommendation further shows the superior performance of \ref{aed} compared to regularized DOD.
 In addition,
 DOD, AOD, and Bayesian DOD   have been proven to be NP-hard, see \cite{welch1982algorithmic,nikolov2022proportional,li2024d}, respectively. In contrast, the complexity of Bayesian AOD remains an open question, though \cite{hendrych2023solving} conjectured that it was NP-hard.
The limited literature on Bayesian AOD motivates us to develop efficient algorithms for its special case, \ref{aed}. 
While this paper does not address the general Bayesian AOD problem, we demonstrate its NP-hardness for the first time.

\vspace{1em}
\noindent\textit{Exact Algorithms.} 
Experimental design problems can be formulated as mixed-integer nonlinear programs (MINLPs) by introducing binary variables to represent subset selection. In view of this, advancing exact solution methods for these MINLPs has been a longstanding research focus.
Given the complexity of the A-optimality function, \ref{aed} and its variants naturally reduce to mixed-integer semidefinite programs (MISDPs) \citep{duarte2023exact}.
To solve these MISDPs, the branch-and-bound (B\&B) method has been extensively studied in the literature (see, e.g., \citealt{ahipacsaouglu2021branch,hendrych2023solving,liang2024pnod}).
These works leveraged the standard continuous relaxation of \ref{aed} to provide a lower bound, referred to as \ref{aed-R1} throughout this paper,  which is also reviewed in detail below. 

 An alternative approach formulates \ref{aed} and its variants as mixed-integer second-order cone programs (MISOCPs). It is worth noting that unlike MISDP,  modern optimization solvers such as Gurobi and CPLEX are well-suited  for solving MISOCP.
\cite{sagnol2011computing} pioneered  this direction by applying Elfving's Theorem
 to demonstrate the second-order representability of the continuous relaxation of  AOD, and it is actually a consequence of Gauss–Markov Theorem. \cite{duarte2015finding} later extended the result to the Bayesian version. Building on these,
a seminar work by \cite{sagnol2015computing} developed exact MISOCP reformulations of AOD and  Bayesian AOD.

 \vspace{1em}
\noindent 
\textit{Convex Relaxation.} A natural way to get rid of the integer challenge  is relaxing the binary variables in the exact MISDP reformulation of \ref{aed} to be continuous, which leads to the convex relaxation \ref{aed-R1}. Much of the literature on the relaxation has focused on developing efficient convex optimization algorithms, see  \cite{ahipacsaouglu2015first,liang2024pnod}, and many references therein.
However, despite these computational advancements, we demonstrate that \ref{aed-R1}  may perform poorly when $k$ is within the range of $[d-1]$.
This limitation
motivates us to develop stronger relaxations, which are important for improving the efficiency of exact algorithms.

\vspace{1em}
\noindent \textit{Approximation Algorithms.}
\cite{avron2013faster} studied the classic AOD problem
and proposed  volume-sampling-based algorithms that achieved an approximation ratio of $(n-d+1)/(k-d+1)$ for $k\in [d, n]$. The result was later extended to \ref{aed} by \cite{derezinski2017subsampling} in the same  regime of $k$, who developed a faster $\O(nd^2)$-time sampling algorithm. 

Another approach to developing sampling algorithms is based on convex relaxations. These algorithms  start from solving a convex relaxation and then, based on the continuous solution,  design randomized sampling strategies to obtain a discrete solution with approximation guarantees,  see \citealt{nikolov2022proportional,wang2017computationally} for the classic AOD problem, \cite{tantipongpipat2020lambda,li2024best} for \ref{aed}, and \cite{allen2017near,derezinski2020bayesian} for Bayesian AOD, respectively. However, these convex relaxation-based sampling algorithms  of \ref{aed} often scale poorly with
the number of experiments $n$  (see, e.g., the fourth and fifth rows of \Cref{table:approx}). 
It is important to note that when applying the approximation algorithms developed for AMESP by \cite{li2024best} to \ref{aed}, an additional constant $(d-k)/\lambda$ must be added, as shown in \Cref{prop:amesp}. Hence, their approximation ratios remain valid when $k \le d$ after adding a nonnegative constant, since for any $c_1\ge c_2$, it is easy to verify that $\frac{c_1+(d-k)/\lambda}{c_2+(d-k)/\lambda}
\le \frac{c_1}{c_2}$. However, these ratios do not hold for $k \in [d+1, n]$, where $(d-k)/\lambda$ becomes negative.

Natural combinatorial algorithms, such as the greedy and local search algorithms,  have been widely used in experimental design (see, e.g.,
\citealt{krause2008near,madan2019combinatorial}).  The forward greedy algorithm has also been extended to Bayesian criteria in \cite{bian2017guarantees,chamon2017greedy}. 
However, A-optimality is not a supermodular function \citep{chamon2017approximate},  which means that the typical approximation ratio for forward greedy (i.e., $(1 - 1/e)$) cannot apply. Instead, these greedy methods often achieve data-dependent guarantees, where the worst-case performance can vary significantly with the experimental setup.  
By data-independence, we mean a theoretical guarantee that does not rely on the experiments $\{\bm a_i\}_{i\in [n]}$ or $\lambda$ but may depend on the parameters $n,d$, and $k$. We are unaware of any prior work on developing a backward greedy search method in the literature of \ref{aed}. This gap is addressed in our paper.
Notably, \cite{li2024best} proposed a local search algorithm for AMESP with strong empirical performance, but its theoretical guarantee is data-independent. 
The performance guarantees and complexities of the approximation algorithms reviewed for \ref{aed} are summarized in \Cref{table:approx}.

\subsection{Summary of the Organization and Contributions}
 Below we list the major contributions and an outline of the remaining  paper.

\begin{enumerate}[(i)]
    \item \textbf{Complexity analysis.} In \Cref{sec:np}, we prove the NP-hardness of \ref{aed} by establishing its equivalence to \ref{amesp} and reducing the independent set decision problem to \ref{amesp};

    \item \textbf{Optimality gaps of convex relaxations.}  \Cref{sec:Method} presents existing (mixed-)integer convex formulations of \ref{aed} and analyzes the optimality gaps of their convex relaxations, as summarized in \Cref{table:rel}.
    Notably, we show that the conventional relaxation \ref{aed-R1} can have unbounded gaps for $k \in [d-1]$, while the relaxation \ref{amesp-R} of \ref{amesp} may perform worse than the trivial zero lower bound for $k\in [d+1,n]$. Finally, in
 Subsection \ref{subsec:misocp}, we revisit the formulation \ref{aed_miscop} from an optimization perspective and prove that its continuous relaxation coincides with \ref{aed-R1};

     \item \textbf{New formulations.} 
Motivated by the weaknesses of existing formulations, we propose a novel convex integer programming formulation  for \ref{aed} in \Cref{subsec:augmesp}, whose continuous relaxation \ref{aed-R2} is tighter than those presented in \Cref{sec:Method}. More importantly, \ref{aed-R2} achieves  bounded optimality gaps for all $k\in [n]$,  which addresses a key drawback of \ref{aed-R1} and \ref{amesp-R}.
Building on our new formulation, we develop an exact cutting-plane algorithm to solve \ref{aed} to optimality;

\item \textbf{Greedy algorithms.}  In \Cref{sec:approx}, we first establish a data-independent approximation ratio for the forward greedy algorithm when $k\in [d-1]$. To complement this, we propose a backward greedy algorithm for \ref{aed} and derive its data-independent performance guarantee for $k\in [d, n]$. These results, summarized in \Cref{table:approx}, present a comprehensive analysis of our greedy algorithms across different ranges of $k$; and
    
    \item \textbf{Numerical study.} Finally, in \Cref{sec:num},  our numerical study on synthetic and real data demonstrates the effectiveness of our exact algorithm and the scalability and high-quality outputs of our approximation algorithms.
    In particular, when applied to the user cold-start problem in movie recommendation systems, \ref{aed} effectively identifies movie subsets that lead to more reliable recommendations for new users in most instances, outperforming its D-optimality counterpart.
\end{enumerate}

\begin{table}[htbp]
\caption{Summary of convex relaxations for \ref{aed}} 
\centering
\label{table:rel}
\setlength{\tabcolsep}{6pt}\renewcommand{\arraystretch}{1}
\begin{tabular}{c|r| r}
\hline
\multicolumn{1}{c|}{\textbf{Convex relaxation}}  & \multicolumn{1}{c|}{\textbf{Optimality gap (i.e., ${z}_k$/relaxation value) }} & \multicolumn{1}{c}{\textbf{Value of $k$}}\\
\hline
\multirow{3}{*}{\ref{aed-R1}}  & Can be arbitrarily large  & $k\in [d-1]$\\
\cline{2-3}
&  $\min\left\{d, n-d+1\right\}$ & $k=d$\\
\cline{2-3}
& $\frac{n-d+1}{k-d+1}$ & $k\in [d+1,n]$ \\
\hline
\multirow{3}{*}{\ref{amesp-R}} & $\min\left\{ \frac{d-1}{d-k}, n-k+1\right\}$ & $k\in [d-1]$ \\
\cline{2-3}
&  $\min\left\{d, n-d+1\right\}$ & $k=d$\\
\cline{2-3}
& Can be negative & $k\in [d+1, n]$\\
\hline
\multirow{3}{*}{\ref{aed-R2}} & $\min\left\{ \frac{d-1}{d-k}, n-k+1\right\}$ & $k\in [d-1]$ \\
\cline{2-3}
&  $\min\left\{d, n-d+1\right\}$ & $k=d$\\
\cline{2-3}
& $\frac{n-d+1}{k-d+1}$ & $k\in [d+1, n]$\\
\hline
\end{tabular}%
\end{table}

\begin{table}
\caption{Summary of approximation algorithms for \ref{aed}} 
\centering
\label{table:approx}
\begin{threeparttable}
\setlength{\tabcolsep}{2pt}\renewcommand{\arraystretch}{1}
\begin{tabular}{c|r|r}
\hline
\multicolumn{1}{c|}{\textbf{Approximation algorithm}} & \multicolumn{1}{c|}{\textbf{Time complexity}} & \multicolumn{1}{c}{\textbf{Approximation ratio}}\\
\hline
Sampling \citep{derezinski2017subsampling} & $\O((n^2+d^2)d)$ & $\frac{n-d+1}{k-d+1}$ if $k\in [d,n]$\\
\hline
Fast sampling \citep{derezinski2017subsampling} & $\O((n+d)d^2)$ & $\frac{n-d+1}{k-d+1}$ if $k\in [d,n]$\\
\hline
 Forward greedy \citep{chamon2017greedy}  & $\O(knd^2)$ &$-$\tnote{i} \\
\hline
Sampling  \citep{tantipongpipat2020lambda} & $\O(n^4dk^2\log(dk))$ & $-$  \\
\hline
Sampling \citep{li2024best} & $\O(n^4dk^2\log(dk))$ &  $\min\{k, n-k+1\}$ if $k\in [d]$\\
\hline
Local search \citep{li2024best} & $\O(k(n-k)d^5)$ & $-$\\
\hline
\hline
Forward greedy \Cref{algo:forward} & $\O(knd^2)$ & $\frac{d-1}{d-k}$ if $ k\in [d-1]$\\
\hline
Backward greedy \Cref{algo:backward} & $\O((n-k)nd^2)$ &  $\frac{n-d+1}{k-d+1}$ if $k\in [d,n]$ \\
\hline
\end{tabular}%
\begin{tablenotes}
    \item[i]``$-$": The approximation ratio depends on data $\{\bm a_i\}_{i\in [n]}$ and the constant $\lambda$
\end{tablenotes}
\end{threeparttable}
\end{table}

 \textit{Notations.} We let $\S_+^n,\S_{++}^n$ denote the set of all the $n\times n$ symmetric positive semidefinite and positive definite matrices, respectively.
 We let $\Re^n$, $\mathbb{R}_+^n$, and $\mathbb{R}_{++}^n$ denote the set of all the $n$-dimensional vectors, nonnegative vectors, and positive vectors, respectively.
  We use bold lower-case letters
(e.g., $\bm x$) and bold upper-case letters (e.g., $\bm X$) to denote vectors and matrices, respectively, and use
corresponding non-bold letters (e.g., $x_i$) to denote their components.
 Given two positive integers $s\le n$, we let $[n]=\{1,2,\cdots, n\}$, and let $[s,n]=\{s, s+1, \cdots, n\}$. We let $\bm I_n$ be the $n\times n$
identity matrix. 
For a vector $\bm x \in \Re^n$,  let $\Diag(\bm x)$ denote a diagonal matrix with diagonal entries being $\bm x$ and let $\|\bm x\|_2$ denote its two norm. If $\bm x \in \Re_+^n$ is nonnegative,  we define $\sqrt{\bm x}$ as the vector where each entry is the square root of the corresponding element in $\bm x$.
For a  matrix $\bm X \in \Re^{d\times n}$, 
we let $\|\bm X\|_F$ denote its Frobenius norm, let $\bm X_i \in \Re^d$ denote its $i$-th column vector for each $i\in [n]$, and for a subset $S\subseteq[n]$, let $\bm X_S$ denote the submatrix consisting of the columns of $\bm X$ indexed by $S$.  For a symmetric matrix $\bm X\in \S^n_+$, 
let $\tr(\bm X)$ denote its trace, and for a subset $S\subseteq[n]$, we let $\bm X_{S,S}$ denote a principal submatrix of $\bm X$ indexed by $S$.

\section{Complexity Analysis}\label{sec:np}
We begin by showing the equivalence between \ref{aed} and AMESP with an additional constant of ${(d-k)}/{\lambda}$.
\Cref{prop:amesp}  extends \citet[Theorems 1 and 2]{li2024d}, where they demonstrated the equivalence between Bayesian DOD and MESP. Our \Cref{prop:amesp}   further connects experimental design problems with maximum entropy-based optimization.

 For brevity,  we define the objective function of \ref{aed} as 
 $$f(S):=\tr \bigg(\bigg(\sum_{i\in S} \bm a_i\bm a_i^{\top}+ \lambda \bm I_d\bigg)^{-1} \bigg)$$ for any subset $S\subseteq [n]$ throughout.

\begin{proposition}\label{prop:amesp}
  \ref{aed} can be converted into the following:
   \begin{align}\label{amesp}
z_k=\min_{S\subseteq [n] } \left\{ \tr \left(\left(\bm C_{S,S} \right)^{-1}\right) + \frac{d-k}{\lambda}: |S|=k \right\}, \tag{AMESP}
\end{align}
where $\bm C \in \S_+^n := \bm A^{\top}\bm A +\lambda \bm I_n$, and $\bm A\in \Re^{d\times n}$ is a matrix whose columns  are vectors $\{\bm a_i\}_{i\in [n]}$.
\end{proposition}
\begin{proof}{Proof.}
    See Appendix~\ref{proof:amesp}. \qed
\end{proof}

It is well known that the independent set decision problem is NP-hard \citep[Chapter 34]{cormen2022introduction}. Next, we prove that the independent set decision problem can be reduced to \ref{amesp}, which, combined with \Cref{prop:amesp},  establishes the NP-hardness of \ref{aed}.

\begin{definition}[Independent sets and the independent set decision problem]
In an undirected graph $G$, an independent set $S$ is a set of vertices such that no two vertices in $S$ are adjacent. Given an integer $k$,  the independent set decision problem  is to determine  if the graph $G$ contains an independent set of size $k$.
\end{definition}

Let us introduce a known inequality based on the comparison between the harmonic mean and arithmetic mean, which is a key technical result used throughout the paper.
\begin{lemma}[\citealt{sedrakyan2018algebraic}] \label{ineq:amhm}
Suppose $\bm \sigma\in \Re_{++}^n$. Then, the following holds
 \[\sum_{i\in [n]} \frac{1}{\sigma_i} \ge \frac{n^2}{\sum_{i\in [n]}\sigma_i},\]
which becomes an equality if and only if all entries of $\bm \sigma$ are equal.
\end{lemma}

\begin{theorem}[NP-hardness] \label{them:hard}
\ref{aed} is NP-hard.
\end{theorem}
\begin{proof}{Proof.}
    See Appendix~\ref{proof:hard}. \qed
\end{proof}
\Cref{them:hard} confirms the  following conjecture from \cite{hendrych2023solving} regarding Bayesian AOD.
Since \ref{aed} is a special case of Bayesian AOD, the conjecture is clearly true.
\begin{conjecture} [\citealt{hendrych2023solving}]
The Bayesian AOD problem is  NP-hard.  
\end{conjecture}

\section{Optimality Gaps of Convex Relaxations}\label{sec:Method}
This section presents two widely-used (mixed-)integer convex formulations of \ref{aed} and derives the optimality gaps for their continuous relaxations.
Leveraging \Cref{prop:amesp}, we further analyze the quality of a convex relaxation  for \ref{amesp}.

\subsection{The Conventional Convex Integer Programming Formulation and its Convex Relaxation}\label{subsec:convention}
By introducing a binary characterization of the subset $S\subseteq [n]$, where $x_i=1$ if $i\in S$ and $x_i=0$ otherwise, we can readily obtain a convex integer nonlinear formulation of \ref{aed}
\begin{align}\label{aed1}
z_k=\min_{\bm x \in\{0,1\}^n } \bigg\{ \tr \bigg( \bigg(\sum_{i\in [n]} x_i \bm a_i\bm a_i^{\top}+ \lambda \bm I_d\bigg)^{-1} \bigg): \sum_{i\in [n]} x_i=k \bigg\}, 
\end{align}
and its continuous relaxation  provides a practical lower bound \citep{tantipongpipat2020lambda}:
\begin{align}\label{aed-R1}
z_k \ge {\nu}_k^{\text{C}}:=\min_{\bm x \in[0,1]^n}\bigg\{  \tr \bigg(\bigg(\sum_{i\in [n]} x_i \bm a_i\bm a_i^{\top}+ \lambda \bm I_d\bigg)^{-1} \bigg): \sum_{i\in [n]} x_i=k\bigg\}. \tag{RAOD-RI}
\end{align}
It is evident that ${\nu}_k^{\text{C}}$ is always positive.

The problem \eqref{aed1} can be recast as a MISDP by nature.
The conventional relaxation \ref{aed-R1} has been widely incorporated into the B\&B algorithms for solving \eqref{aed1} to optimality (e.g., \citealt{hendrych2023solving}). For B\&B algorithms, the quality of the relaxation bound used is crucial for pruning the search space. We are thus motivated to explore the quality of \ref{aed-R1}. 
Before that,  let us explore the properties of the objective function $f(S)$ in \ref{aed}. 
First, the function $f(S)$ decreases monotonically with respect to $S$, that is, $f(S\setminus\{i\})\ge f(S)$ for any $i\in S$. We also remark that
\begin{remark}\label{remark:mono}
The optimal value $z_k$ decreases as the selection size  $k$ increases. \qedA
\end{remark}
 
\Cref{lem:techniques}  establishes an upper bound for the smallest value of $f(S\setminus\{i\})$ over  $i\in S$. Our derivation leverages eigen-decomposition techniques and the sampling probability of removing an element from $S$, as introduced by \cite{derezinski2017subsampling}. In \cite{derezinski2017subsampling}, this probability led to a regularized volume sampling algorithm with a theoretical guarantee for \ref{aed} only when $k\in [d, n]$. A key difference from \citet[Lemma 5]{derezinski2017subsampling} is that  \Cref{lem:techniques} provides a general treatment that applies to any subset $S$, without restricting its size at least $d$. Based on \Cref{lem:techniques}, we derive the optimality gap of \ref{aed-R1} for $k\in[d,n]$, as shown in \Cref{cor:gap}, and analyze the performance guarantees of our approximation algorithms  in \Cref{sec:approx}.

\begin{lemma}\label{lem:techniques}
For a nonempty subset $S\subseteq [n]$ of size $s\ge 1$, we have that
\[ \min_{i\in S} f(S\setminus\{i\}) \le\begin{cases}
\left(1+\frac{1}{s\rho+d-s}\right)
     f(S), & \text{if } s\in [d-1];\\
\left(1+\frac{1}{d\rho+s-d}\right) f(S), &  \text{if } s\in [d,n].
\end{cases},
 \]
where $\rho\in(0,1):=\frac{\lambda}{\lambda+ \max\{1, s/d\} \max_{i\in [n]}\|\bm a_{i}\|_2^2 }$.
\end{lemma}
\begin{proof}{Proof.}
    See Appendix \ref{proof:lem_tech}. \qed
\end{proof}

\begin{theorem}[Optimality gap]
\label{cor:gap}
The following hold for the lower bound ${\nu}_k^{\text{C}}$ in \ref{aed-R1}:
\begin{enumerate}[(i)]
\item For any $k\in [d,n]$, we have
$$
1\le \frac{z_k}{{\nu}_k^{\text{C}}} \le \frac{n-d+1}{k-d+1};
$$
and
\item For $k\in [d-1]$, there is a case of \ref{aed-R1} where ${z_k}/{{\nu}_k^{\text{C}}}\to \infty$ as $\lambda\to 0$. 
\end{enumerate}  
\end{theorem}
\begin{proof}{Proof.}
    See Appendix \ref{proof:gap}. \qed
\end{proof}

 \Cref{cor:gap} shows that \ref{aed-R1} provides a lower bound with a bounded data-independent optimality gap for $k\in [d,n]$. It is evident that this gap decreases in $k$ but increases in $d$ and $n$. Note that the gap can be further tightened at $k=d$, as shown in \Cref{them:relgap}.
 However, for small values of $k\in [d-1]$, \ref{aed-R1} may lead to an infinite optimality gap.

\subsection{A Convex Relaxation of AMESP}\label{subsec:mesp}
Based on the equivalence between \ref{aed} and \ref{amesp}, this subsection presents a convex relaxation of \ref{amesp} proposed by \cite{li2024best}. We establish a novel optimality gap of \ref{amesp-R} for $k\in [d]$ using \Cref{ineq:amhm}. However, when $k\in [d+1,n]$, we demonstrate that \ref{amesp-R} may yield a negative optimal value due to the negativity of the constant $(d-k)/\lambda$.

\begin{definition}[\citealt{li2024best}]\label{def:phi}
	For a matrix $\bm Y\in \S_+^n$ with eigenvalues $\beta_1 \ge \cdots \ge \beta_n \ge 0$,  we define a function $\Phi: \S_+^n\to \Re$ as
	\[\Phi(\bm Y) := \phi(\bm \beta)=  \sum_{i\in [\iota]} \frac{1}{\beta_i} + \frac{(k-\iota)^2}{\sum_{i\in [\iota+1, n]} \beta_i},\]
	where $0\le \iota \le k-1$ is a unique integer such that  $\beta_{\iota}>\frac{1}{k-\iota}\sum_{i\in[\iota+1,n]} \beta_i\ge \beta_{\iota+1} $ with $\beta_0=\infty$.
\end{definition}
Note that $\Phi(\bm Y)$ is a spectral function that relies on only the eigenvalues of $\bm Y$. \cite{li2024best}
demonstrated its convexity and explicitly described its subgradient.

Recall that in \Cref{prop:amesp}, we define $\bm C = \bm A^{\top}\bm A+\lambda\bm I_n$.
Given that  $\bm C\in \S_{++}^n$ is  positive definite,
we compute its Cholesky factorization $\bm C=\bm H^{\top}\bm H $, where $\bm H \in \Re^{n\times n}$ and $\bm h_i\in \Re^n$ represents the $i$-th column vector of $\bm H$ for all $i\in [n]$. According to \citet[Theorem 9]{li2024best}, the following formulation based on the function $\Phi$ is a convex relaxation of 
\ref{amesp}.
\begin{align}\label{amesp-R}
z_k \ge {\nu}^{\text{M}}_k:=\min_{\bm x \in [0,1]^n } \bigg\{ \Phi \bigg(\sum_{i\in [n]} x_i \bm h_i\bm h_i^{\top} \bigg) + \frac{d-k}{\lambda}: \sum_{i\in [n]} x_i=k \bigg\}. \tag{AMESP-R}
\end{align}
It is important to note that  \ref{amesp-R} is exact when $\bm x$ is binary.

Next, we present our main results about \ref{amesp-R}.

\begin{theorem}[Optimality gap] \label{cor:amesp_gap}
The following hold for the lower bound ${\nu}^{\text{M}}_k$ in \ref{amesp-R}:
\begin{enumerate}[(i)]
\item For any $k\in [d]$, we have $$1\le \frac{z_k}{{\nu}^{\text{M}}_k} \le \begin{cases}
  \min\left\{\frac{d-1}{d-k},  n-k+1\right\} & \text{if } k\in [d-1];\\
  \min\left\{d, n-d+1\right\} & \text{if } k=d.
\end{cases};$$
and
\item If $k\ge d+1$, then there is a case of \ref{amesp-R} where ${\nu}^{\text{M}}_k<0$.
\end{enumerate}
\end{theorem}
\begin{proof}{Proof.}
    See Appendix~\ref{proof:amesp_gap}. \qed
\end{proof}
When $k\in [d-1]$, Part (i) of \Cref{cor:amesp_gap} guarantees an optimality gap for \ref{amesp-R} equal to the minimum of two bounds: the first bound $(d-1)/(d-k)$, is newly derived in this work, while the second $(n-k+1)$ is the prior result in \cite{li2024best}. When $k\le (n+1)/2$, it is easy to check that $n-k+1\ge k\ge (d-1)/(d-k)$, where the second inequality is from the proof of \Cref{cor:amesp_gap}.
While not directly comparable, our bound $(d-1)/(d-k)$ offers a compelling advantage due to its independence of $n$ and its scaling with $d$. Specifically, it becomes smaller as the dimension $d$ increases, which is a somewhat counterintuitive yet interesting property.
 This suggests that for $k\in [d-1]$, \ref{amesp-R} achieves improved theoretical guarantees in high-dimensional settings.

According to \Cref{cor:gap,cor:amesp_gap}, the two convex relaxations, \ref{aed-R1} and \ref{amesp-R}, differ significantly in their theoretical guarantees.
Overall, \ref{amesp-R} performs well for  
$k\in[d]$, while \ref{aed-R1} is effective for larger values of $k\in [d,n]$. 
In addition,
we would like to highlight that \ref{amesp-R} has two major limitations.
\begin{enumerate}[(i)]
    \item First, \ref{amesp-R} is formulated in a  higher-dimensional matrix space of size  $n\times n$, as opposed to the $d\times d$ space used in \ref{aed-R1}. This increases the computational cost of solving \ref{amesp-R}, making it less unfavorable for large-scale instances; and
    \item For $k\in [d+1,n]$, according to Part (ii) of \Cref{cor:amesp_gap},  the lower bound from \ref{amesp-R} may underperform even the trivial bound of zero. The numerical results in Subsection \ref{subsec:rel} further confirm this, showing that \ref{amesp-R} often yields negative lower bounds  and  performs much more  poorly compared to \ref{aed-R1} in this range of $k$.
\end{enumerate}

\subsection{A Mixed-integer Second-order Cone Programming Formulation}\label{subsec:misocp}
This subsection presents a MISOCP reformulation of \ref{aed}, which was first introduced by \cite{sagnol2015computing} based on the Gauss–Markov Theorem from statistical estimation theory \citep{pukelsheim2006optimal}.  We provide a different proof of this formulation through an optimization lens and further contribute by analyzing the optimality gap of its continuous relaxation.

Using the Woodbury matrix identity, we show that the objective function of \ref{aed-R1} is SOCP representable for any (not necessarily binary) vector $\bm x$, as summarized below.
\begin{lemma}\label{lem:socp}
The identity below holds for any vector $\bm x\in [0,1]^n$.
	\begin{align*}
\tr \bigg[\bigg(\sum_{i\in [n]} x_i \bm a_i\bm a_i^{\top}+ \lambda \bm I_d\bigg)^{-1} \bigg] =	\min_{\bm X \in \Re^{d\times n}} \frac{1}{\lambda} \left\|\bm X  \Diag\left(\sqrt{\bm x}\right) \bm A^{\top} - \bm I_{d}\right\|_F^2 + \|\bm X\|_F^2.
	\end{align*}
\end{lemma}

\begin{proof}{Proof.}
    See Appendix~\ref{proof:socp}. \qed
\end{proof}

Using the identity in \Cref{lem:socp} to replace the objective function of \eqref{aed1}, we arrive at an equivalent MISCOP formulation. 
 
\begin{proposition}\label{prop:misocp}
\ref{aed} is equivalent to a MISOCP problem:
\begin{align}\label{aed_miscop}
z_k=	\min_{\begin{subarray}{c}
		\bm x \in \{0,1\}^n, \\
        \bm X \in \Re^{d\times n},
	\bm \mu \in \Re^n
		\end{subarray}
}\left\{ \frac{1}{\lambda} \|\bm X \bm A^{\top} - \bm I_{d}\|_F^2 +\sum_{i\in [n]}
	\mu_i: {\|\bm X_{i}\|^2_2} \le \mu_i x_i,\forall i \in [n], \sum_{i\in [n]} x_i=k \right\}. \tag{MISOCP} 
\end{align}
In addition, the continuous relaxation of \ref{aed_miscop} matches that of \ref{aed}, that is,
\begin{align}\label{aed_scop}
{\nu}_k^{\text{C}}=	\min_{\begin{subarray}{c}
	\bm x \in [0,1]^n, \\
        \bm X \in \Re^{d\times n},
	\bm \mu \in \Re^n
	\end{subarray}}\left\{ \frac{1}{\lambda} \|\bm X \bm A^{\top} - \bm I_{d}\|_F^2 +\sum_{i\in [n]}
	\mu_i: {\|\bm X_{i}\|^2_2} \le \mu_i x_i,\forall i \in [n], \sum_{i\in [n]} x_i=k \right\}.
\end{align}
\end{proposition}
\begin{proof}{Proof.}
    See Appendix~\ref{proof:misocp}. \qed
\end{proof}

As
 \ref{aed_miscop} attains the same continuous relaxation value as \eqref{aed1}, we can directly apply the optimality gaps of \ref{aed-R1} from \Cref{cor:gap} to \eqref{aed_scop}, which are summarized below. However, analogous to \ref{aed-R1}, \eqref{aed_scop} may provide a weak lower bound for $k\in [d-1]$.

 \begin{corollary}\label{cor:misocpgap}
The following hold for the lower bound ${\nu}_k^{\text{C}}$ in \eqref{aed_scop}:
\begin{enumerate}[(i)]
\item For any $k\in [d,n]$, we have
$$
1\le \frac{z_k}{{\nu}_k^{\text{C}}} \le \frac{n-d+1}{k-d+1};
$$
and
\item For $k\in [d-1]$, there is a case of \ref{aed-R1} where ${z_k}/{{\nu}_k^{\text{C}}}\to \infty$ as $\lambda\to 0$. 
\end{enumerate}  
\end{corollary}

\section{A Novel Convex Integer Formulation and its Convex Relaxation}\label{subsec:augmesp}
The limitations of \ref{aed-R1} and \ref{amesp-R}  highlight the need for a stronger reformulation of \ref{aed}.
Motivated by this, this section introduces a novel convex integer program for \ref{aed}, and we demonstrate that its continuous relaxation is stronger than \ref{aed-R1} and \ref{amesp-R}.  We also develop a cutting-plane algorithm to  solve \ref{aed} based on this  formulation.

\subsection{A Novel Convex Integer Programming Formulation}
We observe that for a binary vector $\bm x$ with a cardinality of $k$, the matrix $\sum_{i\in [n]} x_i\bm a_i\bm a_i^{\top}$ in the objective of \eqref{aed1} has a rank of at most $\min\{k, d\}$, and becomes low-rank when $k<d$. This rank deficiency  results in an arbitrarily large gap of \ref{aed-R1} for $k<d$, as detailed in the proof of Part (ii) of \Cref{cor:gap}. Inspired by this observation, we introduce a function based on  the top $\min\{k, d\}$  eigenvalues to reformulate the objective function of \eqref{aed1}.  We define  $\tilde k := \min\{k, d\}$ throughout for brevity. 

\begin{definition}\label{def:gamma}
 For a matrix $\bm X\in \S_+^d$ with eigenvalues $\sigma_1\ge \cdots \ge \sigma_d\ge 0$, we define a function $\Gamma(\bm X;\lambda): \S_+^d \to \Re$ as
 $$\Gamma(\bm X;\lambda) =\sum_{i\in [\tilde k]} \frac{1}{\sigma_i+\lambda}.$$
\end{definition}

\begin{lemma}\label{lem:obj}
For a binary vector $\bm x\in \{0, 1\}^n$ with $\sum_{i\in [n]}x_i=k$, the identity below holds.
$$\Gamma \bigg(\sum_{i\in [n]} x_i \bm a_i\bm a_i^{\top};\lambda \bigg) + \frac{d-\tilde k}{\lambda} = \tr\bigg (\bigg(\sum_{i\in [n]} x_i \bm a_i\bm a_i^{\top} + \lambda\bm I_d \bigg)^{-1}\bigg).$$
\end{lemma}
\begin{proof}{Proof.}
    See Appendix~\ref{proof:lem_obj}. \qed
\end{proof}

Although the function $\Gamma$ equals the objective function of \eqref{aed1} based on \Cref{lem:obj}, it is not convex, as noted below. To address this,  we derive its convex envelope, denoted by $\conv \Gamma$, which provides the tightest convex underestimator of $\Gamma$ \citep{rockafellar1997convex}. While  \cite{li2024best} derived the convex envelope $\conv \Gamma$ for the special case $\lambda=0$ to obtain \ref{amesp-R}, our \Cref{prop:conv} below extends  this result to the setting with $\lambda>0$.
\begin{remark}\label{remark:nonconvex}
    The function $\Gamma(\bm X;\lambda): \S_+^d \to \Re$ is nonconvex. \qedA
\end{remark}
\begin{proof}{Proof.}
    See Appendix~\ref{proof:nonconvex}. \qed
\end{proof}

Next, we present a technical result of \citet[Lemma 14]{nikolov2015randomized}. We will use the integer $\eta$ throughout this section,  with the vector $\bm \sigma$ being specified in context. We show that $\eta$ remains unchanged under certain modifications to  $\bm \sigma$. This invariance result plays a key role in deriving the convex envelope $\conv \Gamma$ and its subgradient. 
\begin{lemma}[The integer $\eta$, \cite{nikolov2015randomized}]\label{lem:eta}
Let $\sigma_1\ge \cdots\ge \sigma_d\ge 0$ be nonnegative reals. Then, there exists a unique integer $\eta$, $0\le \eta \le \tilde k-1$,  such that $\sigma_{\eta}>\frac{1}{\tilde k-\eta}\sum_{i\in[\eta+1,d]} \sigma_i\ge \sigma_{\eta+1} $ with the convention $\sigma_0=\infty$. 
\end{lemma}

\begin{lemma}\label{lem:inv}
Let  $s\in [d, n]$, and let $\sigma_1\ge \cdots\ge \sigma_d\ge 0$ be nonnegative reals. Define a new vector $\bm \sigma^+\in \Re_+^s$ as 
\[\sigma^+_i = \sigma_i+\lambda, \forall i\in[\tilde k], \ \  \sigma^+_i =\sigma_i, \forall i\in[\tilde k+1, d],  \ \ \sigma^+_i=0, \forall i\in[d+1,s].\]
Then, the unique index $0\le \eta\le \tilde k-1$, defined based on $\bm \sigma$ in \Cref{lem:eta}, remains valid for $\bm \sigma^+$, i.e.,
$$\sigma^+_{\eta}>\frac{1}{\tilde k-\eta} \sum_{i\in[\eta+1,s]} \sigma^+_i    \ge \sigma^+_{\eta+1}.$$
\end{lemma}
\begin{proof}{Proof.}
    See Appendix~\ref{proof:inv}. \qed
\end{proof}

\begin{proposition}\label{prop:conv}
 For a matrix $\bm X\in \S_+^d$ with eigenvalues $\sigma_1\ge \cdots \ge \sigma_d\ge 0$, the following hold:
 \begin{enumerate}[(i)]
     \item The convex envelope of the function $\Gamma$, denoted by $\conv \Gamma$, is defined as 
	\[\conv \Gamma(\bm X; \lambda)=\sum_{i\in [\eta]} \frac{1}{\sigma_i+\lambda} + \frac{(\tilde k-\eta)^2}{\sum_{i\in [\eta+1, d]} \sigma_i + (\tilde k-\eta)\lambda} ,\]
where $0\le \eta \le \tilde k-1$ is a unique integer from \Cref{lem:eta}.
\item  Let $\bm X= \bm Q\Diag(\bm \sigma)\bm Q^{\top}$ denote the eigen-decomposition of $\bm X$, and let $r$ denote its rank. Then, $\bm Q\Diag\left(\bm \varsigma\right)\bm Q^{\top}$ is a subgradient of $\conv \Gamma(\bm X;\lambda)$ at $\bm X$, where
\begin{align*}    \varsigma_{i}=\frac{1}{\sigma_{i}+\lambda}, \forall i\in [\eta], \ \ \varsigma_{i}=\frac{\tilde k-\eta}{\sum_{j\in [\eta+1,d]}\sigma_j+(\tilde k-\eta)\lambda}, \forall i \in [\eta+1, r], \ \ \varsigma_{i}=\varsigma_{\eta+1}, \forall i\in [r+1, d].
\end{align*}
 \end{enumerate}
\end{proposition}

\begin{proof}{Proof.}
    See Appendix~\ref{proof:conv}. \qed
\end{proof}

The functions $\conv \Gamma$ and $\Phi$ from  \Cref{def:phi} share a similar form but differ notably in their domains and parameter settings.
\begin{enumerate}[(i)]
    \item Unlike $\Phi$, which is defined over $n\times n$ matrices and relies on the condition $k\le n$,  the function $\conv \Gamma$ builds on $d\times d$ matrices and  the integer $\tilde k\le d$;
    \item In addition, $\conv \Gamma$ includes a regularization parameter $\lambda>0$, which $\Phi$ does not have; and
    \item When using $\Phi$ to relax \ref{aed}, the resulting relaxation \ref{amesp-R} incorporates a constant term $(d-k)/\lambda$ in its objective function. If $d>k$, that term becomes negative, which may  result in a negative lower bound, as demonstrated in \Cref{cor:amesp_gap}. In contrast, the relaxation \ref{aed-R2}, which is based on $\conv \Gamma$, always  incorporates a nonnegative constant $(d-\tilde k)/\lambda$.  Hence,   $\conv \Gamma$ can prevent the relaxation from yielding negative lower bounds. More importantly, we later demonstrate that \ref{aed-R2}  dominates \ref{amesp-R}.
\end{enumerate}

We establish below a connection between the function $\conv \Gamma$ and the trace of the matrix inverse, which lays the foundation for our subsequent results.
\begin{lemma}\label{lem:gamma}
 If the rank of $\bm X$ is bounded by $\tilde k$, then we have
 $$\conv \Gamma (\bm X;\lambda)+ \frac{d-\tilde k}{\lambda}=\tr((\bm X+\lambda\bm I_d)^{-1}).$$
\end{lemma}
\begin{proof}{Proof.}
    See Appendix~\ref{proof:gamma}. \qed
\end{proof}

Using \Cref{lem:gamma}, we are now ready to introduce a novel convex integer program for \ref{aed}.
\begin{theorem}\label{them:novel}
\ref{aed} is equivalent to the following convex integer program:
\begin{align}\label{eq:novel}
z_k:=\min_{\bm x \in \{0,1\}^n } \bigg\{ \conv\Gamma \bigg(\sum_{i\in [n]} x_i \bm a_i\bm a_i^{\top}; \lambda \bigg) + \frac{d-\tilde k}{\lambda}: \sum_{i\in [n]} x_i=k \bigg\}.
\end{align}
\end{theorem}
\begin{proof}{Proof.}
Let $\bm X= \sum_{i\in [n]} x_i \bm a_i\bm a_i^{\top}$.  For any binary vector $\bm x\in \{0,1\}^n$ with $\sum_{i\in [n]}x_i=k$, the rank of $\bm X$ must be bounded by $\tilde k$.  According to \Cref{lem:gamma},  the objective functions of \eqref{aed1} and \eqref{eq:novel} must be equal in this context. 
\qed
\end{proof}

\subsection{A Novel Convex Relaxation}
By relaxing the binary variables $\bm x$ of \eqref{eq:novel}, we therefore obtain a new lower bound of \ref{aed}:
\begin{align}\label{aed-R2}
z_k\ge {\nu}^{\text{N}}_k:=\min_{\bm x \in [0,1]^n } \bigg\{ \conv \Gamma \bigg(\sum_{i\in [n]} x_i \bm a_i\bm a_i^{\top}; \lambda \bigg) + \frac{d-\tilde k}{\lambda}: \sum_{i\in [n]} x_i=k \bigg\} .\tag{RAOD-RII}
\end{align}

By leveraging Part (ii) of \Cref{prop:conv}, we can readily obtain the (sub)gradient of the objective function of \ref{aed-R2}, which enables us to use first-order methods, such as the Frank-Wolfe algorithm, to solve \ref{aed-R2} efficiently.
\begin{remark}\label{remark:subgrad}
    For any $\bm x\in [0,1]^n$ with $\sum_{i\in [n]} x_i=k$,
    define the matrix $\bm X(\bm x)= \sum_{i\in [n]} x_i \bm a_i\bm a_i^{\top}$. Suppose $\bm W(\bm x) \in \S_+^n$ is a subgradient  of $\conv \Gamma \left(\bm X; \lambda \right)$  at $\bm X(\bm x)$, as introduced in \Cref{prop:conv}. Then, 
    the  vector $\bm g(\bm x) \in \Re_+^n$, defined by  
    \[g_i(\bm x) := \bm a_i^{\top} \bm W(\bm x) \bm a_i, \ \ \forall i\in [n],\]
is a subgradient of $\conv \Gamma (\sum_{i\in [n]} x_i \bm a_i\bm a_i^{\top}; \lambda )$ with respect to $\bm x$.  \qedA
\end{remark}

In the following, we compare the three convex relaxations of \ref{aed} and analyze their relationships by leveraging \Cref{lem:inv,lem:gamma} and a technical result from \citet[Lemma 1]{li2024augmented}.
\begin{theorem}[Comparison of convex relaxations]\label{them:compare}
\ref{aed-R2} outperforms both \ref{aed-R1} and \ref{amesp-R}. More precisely, the optimal values of these formulations satisfy
\begin{enumerate}[(i)]
\item ${\nu}^{\text{N}}_k = {\nu}^{\text{C}}_k$ for any $k\in [d,n]$;
\item ${\nu}^{\text{N}}_k \ge {\nu}^{\text{C}}_k$ for any $k\in [d-1]$; and
\item ${\nu}^{\text{N}}_k \ge {\nu}^{\text{M}}_k$ for any $k\in [n]$.
\end{enumerate}
\end{theorem}

\begin{proof}{Proof.}
    See Appendix~\ref{proof:compare}. \qed
\end{proof}

Based on  \Cref{them:compare}, we note that  \ref{aed-R2}
offers several key advantages:
\begin{enumerate}[(i)]
    \item It provides a stronger lower bound than \ref{aed-R1} and \ref{amesp-R}, effectively combining their strengths. Thus, we can directly combine \Cref{cor:gap,cor:amesp_gap} to guarantee  theoretical gaps of \ref{aed-R2} for all $k\in [n]$, as summarized in \Cref{them:gap};
    \item For $k\in [d, n]$, Parts (i) and (iii) of \Cref{them:compare} imply  that \ref{aed-R1} outperforms \ref{amesp-R}. This allows us to apply \Cref{cor:amesp_gap} to refine the gap of \ref{aed-R1} at $k=d$ in \Cref{them:relgap}; and
\item It mitigates the computational limitations of \ref{amesp-R} by operating in a $d\times d$ matrix space, achieving similar  efficiency as \ref{aed-R1}.
\end{enumerate}

\begin{corollary}[Optimality gap] \label{them:gap}
\ref{aed-R2} admits the following optimality gap
\begin{align*}
&1\le \frac{z_k}{\nu_k^{\text{N}}} \le \min\left\{\frac{d-1}{d-k}, n-k+1\right\},  \ \ \forall k\in [d-1], \quad 1\le \frac{z_k}{\nu_k^{\text{N}}} \le \min\left\{d, n-d+1\right\}, \ \ \text{if } k=d,\\
& \text{and}  \quad 1\le \frac{z_k}{\nu_k^{\text{N}}} \le \frac{n-d+1}{k-d+1}, \ \ \forall k\in [d+1,n].
\end{align*}
\end{corollary}
To better understand the performance of \ref{aed-R2}, we illustrate its optimality gaps in \Cref{fig_theorygap}. We see that the gap reaches its maximum when $k$ approaches $d$ from either side and tends to decrease as $d$ increases for small values of $k$. 
\begin{figure}[h]
	\centering
\subfigure[$n=30$]{
\includegraphics[width=0.33\textwidth,height=0.25\textwidth]{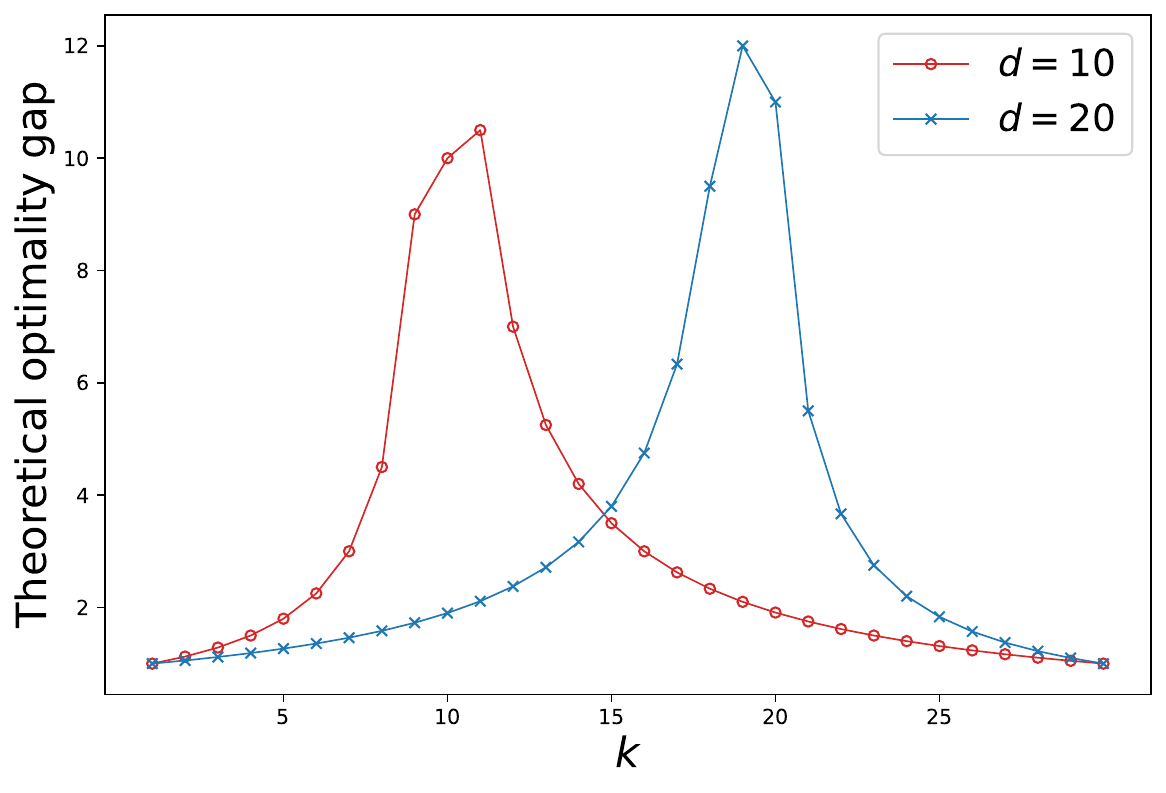}}
    \hspace{4em}
    \subfigure[$n=100$]{
\includegraphics[width=0.33\textwidth,height=0.25\textwidth]{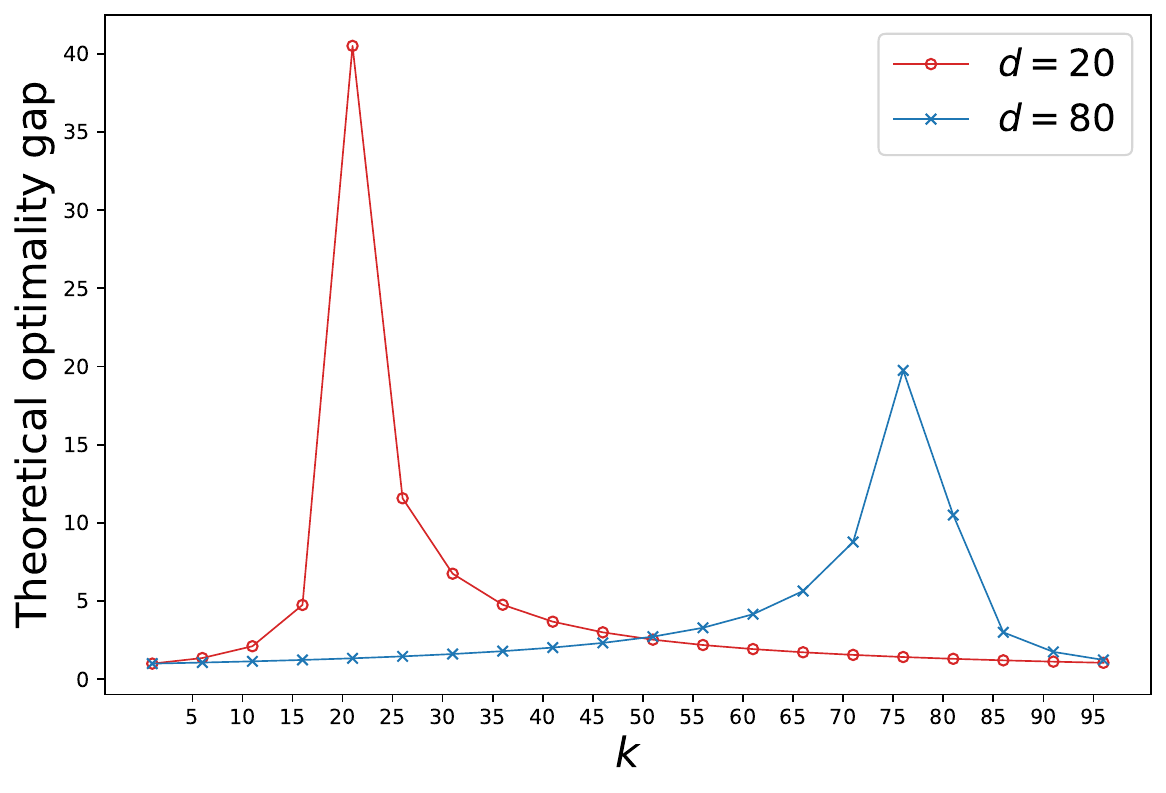}}
	\caption{The optimality gaps of \ref{aed-R2} in \Cref{them:gap} }\label{fig_theorygap}
\end{figure}

\begin{corollary}[Optimality gap] \label{them:relgap}
When $k=d$, \ref{aed-R1} admits the following optimality gap
\begin{align*}
1\le \frac{z_k}{\nu_k^{\text{C}}} \le \min\left\{d, n-d+1\right\}.
\end{align*}
\end{corollary}

\subsection{The Cutting-Plane Algorithm}\label{subsec:cut}
Solving \ref{aed} to optimality is computationally expensive. To mitigate this challenge, this subsection presents an exact algorithm designed to enhance existing B\&B methods, which typically use the classic lower bound from \ref{aed-R1} (see, e.g., \citealt{hendrych2023solving}). Instead, our approach leverages the formulation \eqref{eq:novel} with a stronger relaxation, \ref{aed-R2}. Motivated by the convexity of the function $\conv\Gamma$,  we can recast \eqref{eq:novel} as a mixed-integer linear programming (MILP) formulation \eqref{eq:novel-milp}. This reformulation is a well-established technique for mixed-integer nonlinear optimization problems (\citealt{duran1986outer,quesada1992lp}). 
\begin{equation}\label{eq:novel-milp}
\begin{aligned}
z_k=\min_{
        w\in \Re^+, \bm x\in \{0,1\}^n
  } \bigg\{ w:    \sum_{i\in [n]} x_i=k, w \ge &\conv\Gamma \bigg(\sum_{i\in [n]} y_i \bm a_i\bm a_i^{\top};\lambda\bigg) + \frac{d-\tilde k}{\lambda}\\
    &+ \bm g(\bm y)^{\top}\left(\bm x-\bm y\right), \ \ \forall \bm y\in \{0,1\}^n, \sum_{i\in [n]} y_i=k \bigg\},
    \end{aligned}
\end{equation}
where $\bm g(\bm y)\in \Re^n$ denotes a (sub)gradient of $\conv\Gamma (\sum_{i\in [n]} y_i \bm a_i\bm a_i^{\top};\lambda)$ at $\bm y$, as defined in \Cref{remark:subgrad}.

However, the problem \eqref{eq:novel-milp}  incorporates an exponentially large number of linear inequalities, which causes computational challenges. To  address this, we design a customized cutting plane algorithm for \eqref{eq:novel-milp} that iteratively adds linear inequalities.
The pseudo-code is provided in \Cref{algo:cut}.
This outer-approximation scheme, as introduced in the seminal work of \cite{duran1986outer}, was proven to converge optimally for solving MINLPs.

\begin{algorithm}[ht]
	\caption{The cutting-plane algorithm for \eqref{eq:novel-milp}} \label{algo:cut}
	\begin{algorithmic}
		\State \textbf{Input:} Vectors $\{\bm{a}_i \in \Re^d\}_{i \in [n]}$, a constant $\lambda>0$, an integer $k\in [n]$, and target optimality $\epsilon>0$
        \State Initialize $t=1$,  $\bm x^t\in \{0,1\}^n$ with $\sum_{i\in [n]}x^t_i=k$ and $\mbox{UB}=\conv\Gamma \left(\sum_{i\in [n]} x^{t}_i \bm a_i\bm a_i^{\top}; \lambda\right) + \frac{d-\tilde k}{\lambda}$
		\Repeat
        \State  Solve the following optimization problem, where $(w^{t+1}, \bm x^{t+1})$ denotes an optimal solution
        \begin{align*}
    \min_{\begin{subarray}{c}
        w\in \Re_+, \bm x\in \{0,1\}^n\\
        \sum_{i\in [n]} x_i=k
    \end{subarray}} \bigg\{ w: w \ge \conv\Gamma \bigg(\sum_{i\in [n]} x^{\ell}_i \bm a_i\bm a_i^{\top}; \lambda\bigg) + \frac{d-\tilde k}{\lambda}+ \bm g\left(\bm x^{\ell}\right)^{\top}\left(\bm x-\bm x^{\ell}\right),\forall \ell \in [t]   \bigg\}     
        \end{align*}
        \State Update $\mbox{UB}=\min\left\{ \mbox{UB},\conv\Gamma \left(\sum_{i\in [n]} x^{t+1}_i \bm a_i\bm a_i^{\top};\lambda\right) + \frac{d-\tilde k}{\lambda}\right\}$
        \State Update $t=t+1$
        \Until{$\left(w^t-\mbox{UB}\right)/\mbox{UB}\le \epsilon$}
		\State \textbf{Output:}  $\bm x^t$ and $w^t$%
	\end{algorithmic}
\end{algorithm}

At each iteration of \Cref{algo:cut}, a MILP problem over $(w, \bm x)$, called the \textit{master} problem, 
is solved. This subproblem of \eqref{eq:novel-milp} is based on a subset of 
linear inequalities. 
As the number of linear cuts increases, the solution values of $w^t$ form a non-decreasing sequence of lower bounds for \ref{aed}. \Cref{algo:cut} terminates when  the relative gap between this lower bound (i.e., $w^t$) and the best-known upper bound (i.e., $\mbox{UB}$) reaches a fixed tolerance $\epsilon$. We choose $\epsilon=10^{-4}$ in our experiments. According to \cite{bertsimas2020rejoinder}, \Cref{algo:cut} can be efficiently implemented in advanced MILP packages (e.g., CPLEX, Gurobi) with the use of \textit{lazy} constraints and \textit{callback} functions. However, a time limit is necessary to impose on \Cref{algo:cut}, especially when solving large-sized instances. 
Indeed, as often
in mixed-integer optimization, the cutting plane algorithm can quickly find a high-quality solution, but proving its optimality can take much longer. In our experiments, we set a time limit of one hour for \Cref{algo:cut}.

We employ  several strategies to accelerate the computation of \Cref{algo:cut} in practice. First, we enhance the algorithm  using a high-quality warm start obtained from our approximation algorithms in \Cref{sec:approx}, instead of randomly initializing a feasible solution.
Second, a variable-fixing technique, proposed by \cite{li2024augmented}, is applied to identify  some binary variables in \eqref{eq:novel} that take values of 1 or 0 at optimality. 
Fixing these variables can reduce the problem size of \ref{aed} while preserving the optimal solution. The effectiveness of this variable-fixing process depends on the strength of the convex relaxations used. Notably, 
 the use of \ref{aed-R2} allows for the fixing of more binary variables.
Finally, we solve \ref{aed-R2} before running  \Cref{algo:cut} to  obtain an initial lower bound. This helps reduce the root gap.

\section{The Forward and Backward Greedy Algorithms}\label{sec:approx}
In this section, we investigate the scalable forward and backward greedy search strategies and derive their approximation ratios tailored for two different ranges of $k$ in \ref{aed}, specifically for $k\in [d-1]$ and $k\in[d,n]$, respectively.

The forward greedy algorithm   begins with an empty set and then iteratively selects the best new data point to add to the set. This process continues until the set reaches a size of $k$. Using the Sherman–Morrison formula,  \cite{chamon2017greedy} designed an efficient implementation of forward greedy for solving \ref{aed}, as detailed in \Cref{algo:forward}.

\begin{algorithm}[ht]
	\caption{The forward greedy algorithm for \ref{aed} \citep{chamon2017approximate}} \label{algo:forward}
	\begin{algorithmic}
		\State \textbf{Input:} Vectors $\{\bm{a}_i\}_{i \in [n]}  \in \Re^d$, a constant $\lambda>0$, and an integer $k\in [n]$
        \State Initialize $S_F=\emptyset$ and $\bm \Lambda = \frac{1}{\lambda}\bm I_d$

		\For{$j\in [k]$}
        \State  Compute $i^*= \argmax_{i \in [n]\setminus S_F} \frac{\bm a_i^{\top} \bm \Lambda \bm \Lambda \bm a_i }{1 + \bm a_i^{\top}  \bm \Lambda \bm a_i}$ 
        \State Update $S_F = S_F \cup \{i^*\} $ and $\bm \Lambda=\bm \Lambda -  \frac{ \bm \Lambda  \bm a_{i^*} \bm a_{i^*}^{\top} \bm \Lambda }{1 + \bm a_{i^*}^{\top}  \bm \Lambda \bm a_{i^*}}$
        \EndFor
		\State \textbf{Output:} $S_F$%
	\end{algorithmic}
\end{algorithm}

The backward greedy algorithm, in contrast to the forward greedy approach, starts with a solution of selecting all data points and then removes an element that results in the smallest objective value at each iteration. The procedure repeats until only $k$ data points are left, which is presented in \Cref{algo:backward}. Specifically, at each iteration of \Cref{algo:backward}, we have that $\bm \Lambda = (\sum_{j\in S_B}  \bm a_j\bm a_j^{\top}+ \lambda \bm I_d)^{-1}$,  and the goal is to solve
\begin{align*}
 \min_{i\in S_B} f(S_B\setminus\{i\})  = \min_{i\in S_B} \tr\left(\bm \Lambda + \frac{\bm \Lambda\bm a_i\bm a_i^{\top} \bm \Lambda}{1-\bm a_i^{\top}\bm \Lambda \bm a_i} \right)= f(S_B) + \min_{i\in S_B}  \frac{\bm a_i^{\top} \bm \Lambda \bm \Lambda\bm a_i}{1-\bm a_i^{\top}\bm \Lambda \bm a_i},
\end{align*}
where the first equation results from the Sherman–Morrison formula.
Similarly, we update  the matrix $\bm \Lambda $ by the Sherman–Morrison formula.

\begin{algorithm}[ht]
	\caption{The backward greedy algorithm for \ref{aed}} \label{algo:backward}
	\begin{algorithmic}
		\State \textbf{Input:} Vectors $\{\bm{a}_i\}_{i \in [n]}  \in \Re^d$, a constant $\lambda>0$, and an integer $k\in [n]$
 \State Initialize $S_B=[n]$ and $\bm \Lambda = (\sum_{j\in S_B}  \bm a_j\bm a_j^{\top}+ \lambda \bm I_d)^{-1}$
		\For{$j\in [n-k]$}
        \State  Compute $i^*= \argmin_{i \in S_B} \frac{\bm a_i^{\top} \bm \Lambda \bm \Lambda \bm a_i }{1 - \bm a_i^{\top}  \bm \Lambda \bm a_i}$ 
        \State Update $S_B = S_B \setminus \{i^*\} $ and $\bm \Lambda=\bm \Lambda + \frac{ \bm \Lambda  \bm a_{i^*} \bm a_{i^*}^{\top} \bm \Lambda }{1 - \bm a_{i^*}^{\top}  \bm \Lambda \bm a_{i^*}}$
        \EndFor
		\State \textbf{Output:} $S_B$%
	\end{algorithmic}
\end{algorithm}

Using \Cref{lem:techniques}, we guarantee the theoretical performance for \Cref{algo:forward,algo:backward}.

\begin{theorem}[Approximation ratios and complexities]\label{them:approx}
For $k\in [d-1]$, \Cref{algo:forward} enjoys a-$\frac{d-1}{d-k}$ approximation ratio, that is,
\[f(S_F) \le \frac{d-1}{d-k} z_k. \]
For $k\in [d,n]$, \Cref{algo:backward} enjoys a-$\frac{n-d+1}{k-d+1}$ approximation ratio, that is,
\[f(S_B) \le \frac{n-d-1}{k-d+1} z_k. \]
In addition, the running time complexities of \Cref{algo:forward,algo:backward} are $\O(knd^2)$ and $\O((n-k)nd^2)$, respectively.
\end{theorem}

\begin{proof}{Proof.}
    See Appendix~\ref{proof:approx}. \qed
\end{proof}

We remark on \Cref{them:approx} that
\begin{enumerate}[(i)]
    \item The forward greedy \Cref{algo:forward} has been extensively studied in the Bayesian AOD literature  (see, \citealt{chamon2017approximate,bian2017guarantees}). 
\Cref{them:approx} offers the first-known data-independent approximation ratio for \Cref{algo:forward} when applied to \ref{aed} with $k\in [d-1]$;
    \item We develop the novel backward greedy \Cref{algo:backward} that delivers an approximation ratio when $k\in [d,n]$;
    \item 
As a classical intuition,  forward greedy is effective when selecting a small number of data points, while backward greedy becomes more efficient as the selection size grows.  \Cref{them:approx} rigorously validates this intuition, as it offers theoretical guarantees for \Cref{algo:forward,algo:backward} in the small- and large-$k$ ranges, respectively, as illustrated in \Cref{fig_approx}; and 
    \item Finally, we consider \Cref{algo:forward_backward} to select the  better solution between the outputs of \Cref{algo:forward,algo:backward}, which yields a high-quality solution to \ref{aed} for all values of $k$.
\end{enumerate}

\begin{figure}[h]
	\centering
\subfigure[$d=10$]{
\includegraphics[width=0.36\textwidth,height=0.25\textwidth]{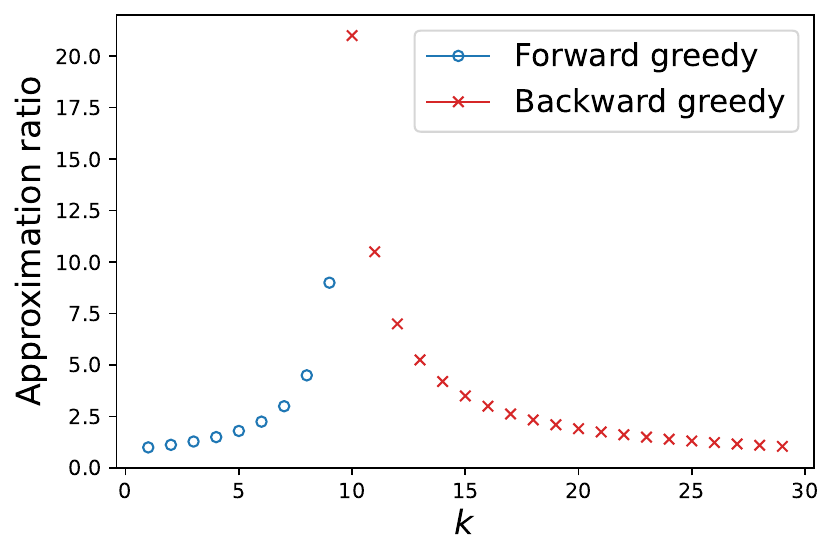}}
    \hspace{3em}
    \subfigure[$d=20$]{
\includegraphics[width=0.36\textwidth,height=0.25\textwidth]{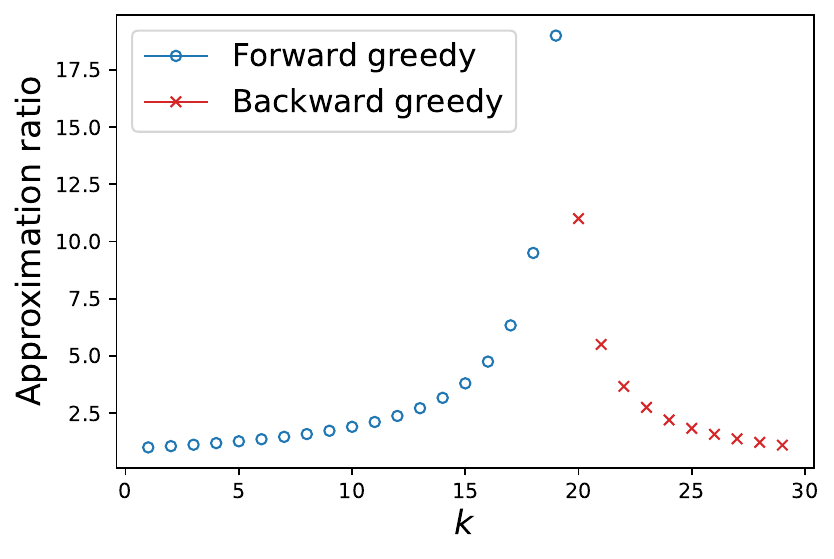}}
	\caption{The approximation ratios of \Cref{algo:forward,algo:backward} in \Cref{them:approx} with $n=30$ }\label{fig_approx}
\end{figure}

\begin{algorithm}[ht]
	\caption{The combination of the forward and backward greedy algorithms for \ref{aed}} \label{algo:forward_backward}
	\begin{algorithmic}
		\State \textbf{Input:} Vectors $\{\bm{a}_i\}_{i \in [n]}  \in \Re^d$, a constant $\lambda>0$, and an integer $k\in [n]$
 \State Denote by $S_F$ the output of \Cref{algo:forward} 
 \State Denote by $S_B$ the output of \Cref{algo:backward} 		
		\State \textbf{Output:} The better solution between $S_F$ and $S_B$%
	\end{algorithmic}
\end{algorithm}

\begin{corollary}\label{cor:approx}
\Cref{algo:forward_backward} has an approximation ratio of $\frac{d-1}{d-k}$ for $k\in [d-1]$,
and $\frac{n-d+1}{k-d+1}$ for $k\in [d,n]$.
In addition, its  time complexity  is $\O(n^2d^2)$.
\end{corollary}

\section{Numerical Results}\label{sec:num}
This section evaluates the numerical performance of our formulations and algorithms developed for solving \ref{aed}.
The computational results not only demonstrate the efficiency of our approaches but also provide empirical support for our theoretical findings. 
All the experiments enforce a 1-hour timeout, and are conducted in Python 3.6 with calls to Gurobi 9.5.2 on a PC with 10-core CPU, 16-core GPU, and 16GB of memory. We begin by introducing the datasets, both synthetic and real, used in our computational experiments.

\textbf{Synthetic data.} Given $(n,d)$, we randomly generate a data matrix $\bm A\in\Re^{d\times n}$ by a normal distribution $\mathcal{N}(0,1)$, and let $\bm a_i$ be the $i$-th column of $\bm A$ for each $i\in [n]$.

\textbf{UCI data \citep{blake1998uci}.} We further conduct experiments on five UCI datasets-  \textit{autos},  \textit{breastcancer}, \textit{sml}, \textit{gas}, and \textit{song}, with dimensions $(n,d)$ of $(159, 24)$, $(194, 33)$, $(200, 22)$, $(500, 128)$, and $(1000, 90)$, respectively.  We begin by removing all-zero columns from the data matrix. For datasets with a huge number of observations, we randomly sample a manageable subset of $n$ data points to ensure computational feasibility when testing exact methods. Finally, we apply min-max normalization to each feature (i.e., each row) of the resulting matrix $\bm A\in \Re^{d\times n}$, so that it falls within the range $[0,1]$.

\textbf{Movie rating data \citep{Dooms13crowdrec}.} The movie rating dataset is used to demonstrate an application of \ref{aed} to the new user cold-start problem in recommendation systems. The  dataset consists of ratings provided by 71,707 users for a total of 38,018 movies.  Ratings are shifted from the range [0, 10] to [1, 11] by adding one to each rating to avoid zero vectors.

\subsection{Evaluations of Convex Relaxations}\label{subsec:rel}
This subsection evaluates the three convex relaxations: \ref{aed-R1}, \ref{amesp-R}, and \ref{aed-R2}, by comparing their  optimality gaps and computation times with varying-scale instances. We use the Frank-Wolfe algorithm to solve them.  The results are summarized in \Cref{fig_synthetic_rel,fig_real_rel}. For each convex relaxation, we compute  {``Gap (\%)"} by $100\times(\text{upper bound}-\text{relaxation bound})/\text{upper bound}$, where the upper bound of \ref{aed} is obtained from \Cref{algo:forward_backward}.
Our results verify the effectiveness of \ref{aed-R2} and its dominance over \ref{aed-R1} and \ref{amesp-R}. 

We first test three synthetic datasets with $n=100$ and different values of $d,\lambda$, where for each dataset, we consider $k\in\{5, 10, \cdots, 95\}$. As \Cref{fig_synthetic_rel} below shows, \ref{aed-R2} delivers the tightest lower bound. 
It is worth noting that the gap curves for \ref{aed-R2} in \Cref{fig:100_20_rel_val,fig:100_80_rel_val,fig:100_20_10_rel_val} closely match the theoretical gaps shown in \Cref{fig_theorygap}, where the maximum occurs  near  $k=d$.
Besides, we would like to comment on the gaps of \ref{aed-R1} and \ref{amesp-R} in two ranges of $k$ based on \Cref{fig:100_20_rel_val,fig:100_80_rel_val,fig:100_20_10_rel_val}. We see that
for  $k< d$, \ref{amesp-R}  outperforms \ref{aed-R1}. 
When $k\ge d+1$, \ref{aed-R1} and \ref{aed-R2} are identical, and they dominate \ref{amesp-R}, as demonstrated in \Cref{them:compare}. 
In fact,  \ref{amesp-R} produces extremely large gaps for $k\in [d+1,n]$, often exceeding $100\%$,  which indicates that it returns a negative lower bound. This phenomenon aligns with our result in Part (ii) of \Cref{cor:amesp_gap}.
As seen from the differences between \Cref{fig:100_20_rel_val,fig:100_80_rel_val},  increasing $d$ tends to improve  \ref{amesp-R} and \ref{aed-R2} but worsen \ref{aed-R1}, especially for small-$k$ instances. 
Finally, we find that \ref{amesp-R} requires the largest computation time since it builds on the $n\times n$ matrix space.

In \Cref{fig_real_rel}, we test the relaxations of \ref{aed} on three larger-sized real datasets ($n=159, 200, 1000$) with $\lambda=1$. The values of 
$k$ are tested in increments of 5 for $n=159, 200$, and  larger increments of 50 for $n=1000$.
As before, \ref{aed-R2} wins on all test instances.
We observe in \Cref{fig:159_real_rel_val,fig:200_real_rel_val,fig:1000} that  as $k\to d$ from below, \ref{aed-R2} strictly outperforms \ref{aed-R1} and \ref{amesp-R} even if we combine them and take the minimum. 
It is still obvious that the gap pf \ref{aed-R2} tends to be larger  when $k$ approaches $d$ from either side. Interestingly, as we observe in \Cref{fig:159_real_rel_val,fig:200_real_rel_val},  \ref{aed-R1} and \ref{amesp-R} are no longer directly comparable for values of $k<d$. For the largest test case in \Cref{fig:1000_real_rel_time}, where $n=1000$, \ref{amesp-R} hits the time limit, highlighting its scalability limitations.
\begin{figure}[ht]
	\centering
\subfigure[Gap for $(n,d,\lambda)=(100,20, 1)$]{
\includegraphics[width=0.3\textwidth]{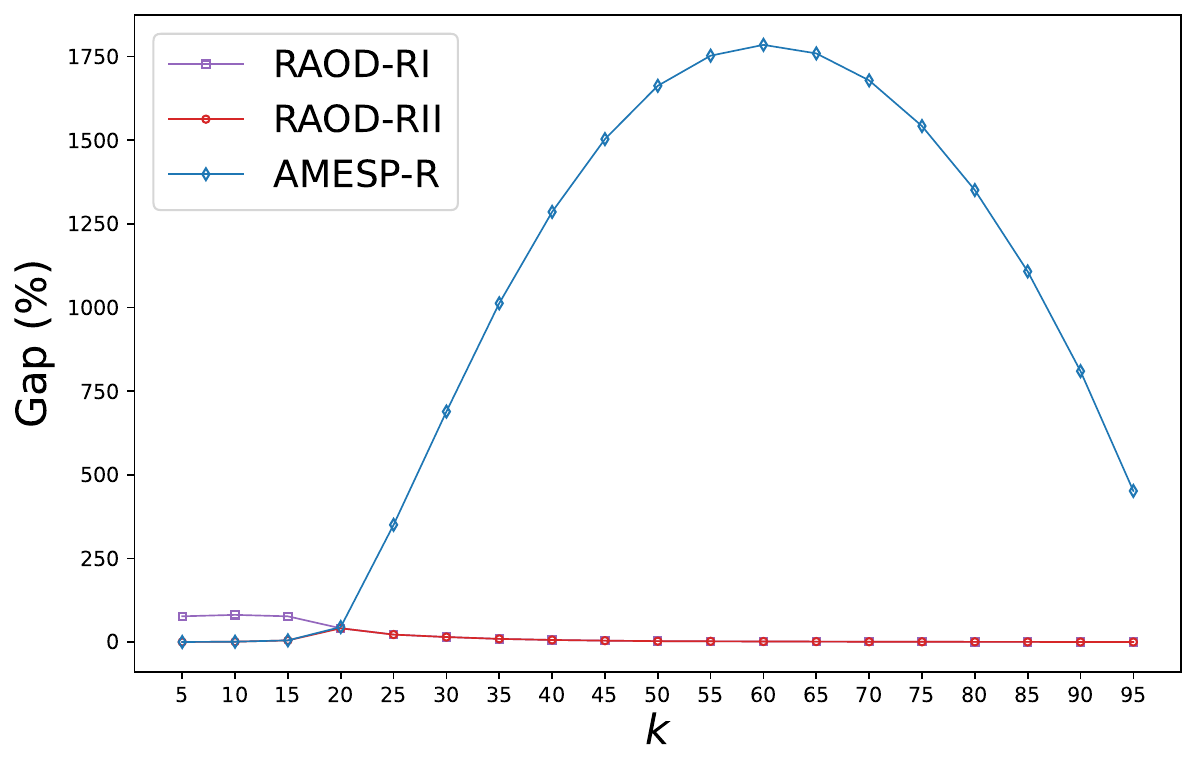}
\label{fig:100_20_rel_val}
	}
\subfigure[Gap for $(n,d,\lambda)=(100,80, 1)$]{
		\centering
\includegraphics[width=0.3\textwidth]{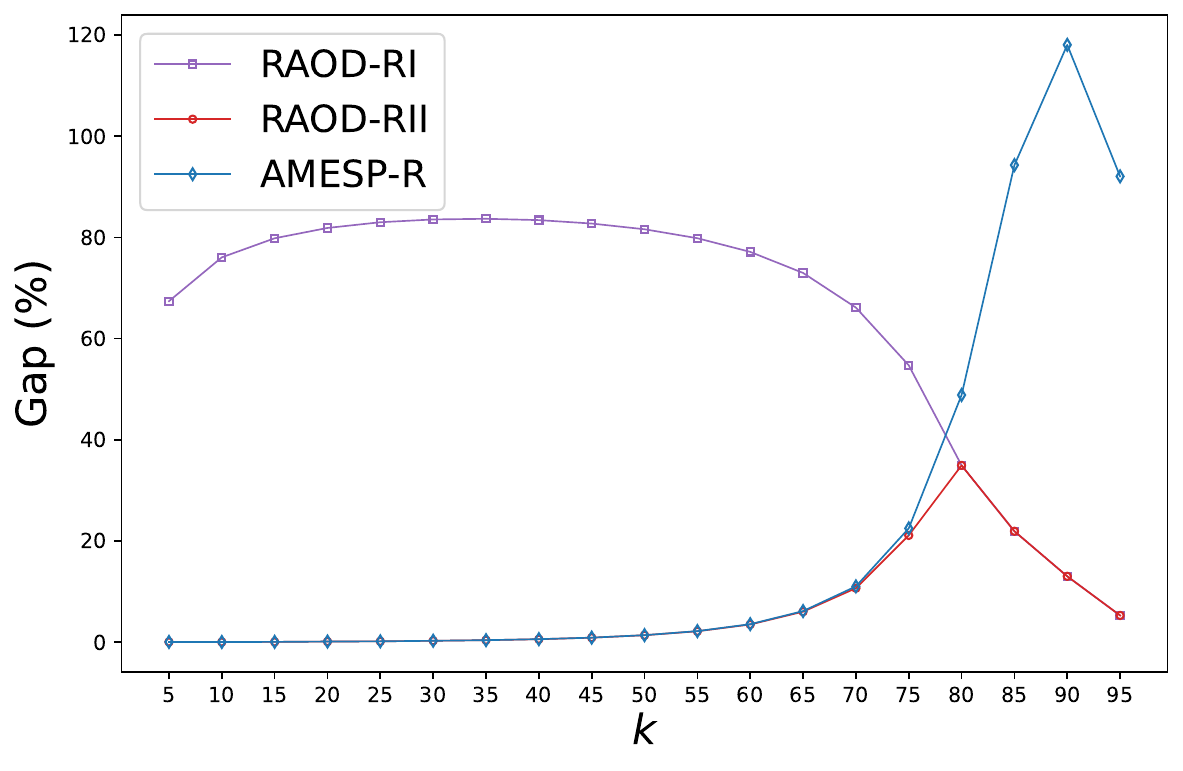}
\label{fig:100_80_rel_val}
	}
 	\subfigure[Gap for $(n,d,\lambda)=(100,20, 10)$]{
		\centering
\includegraphics[width=0.3\textwidth]{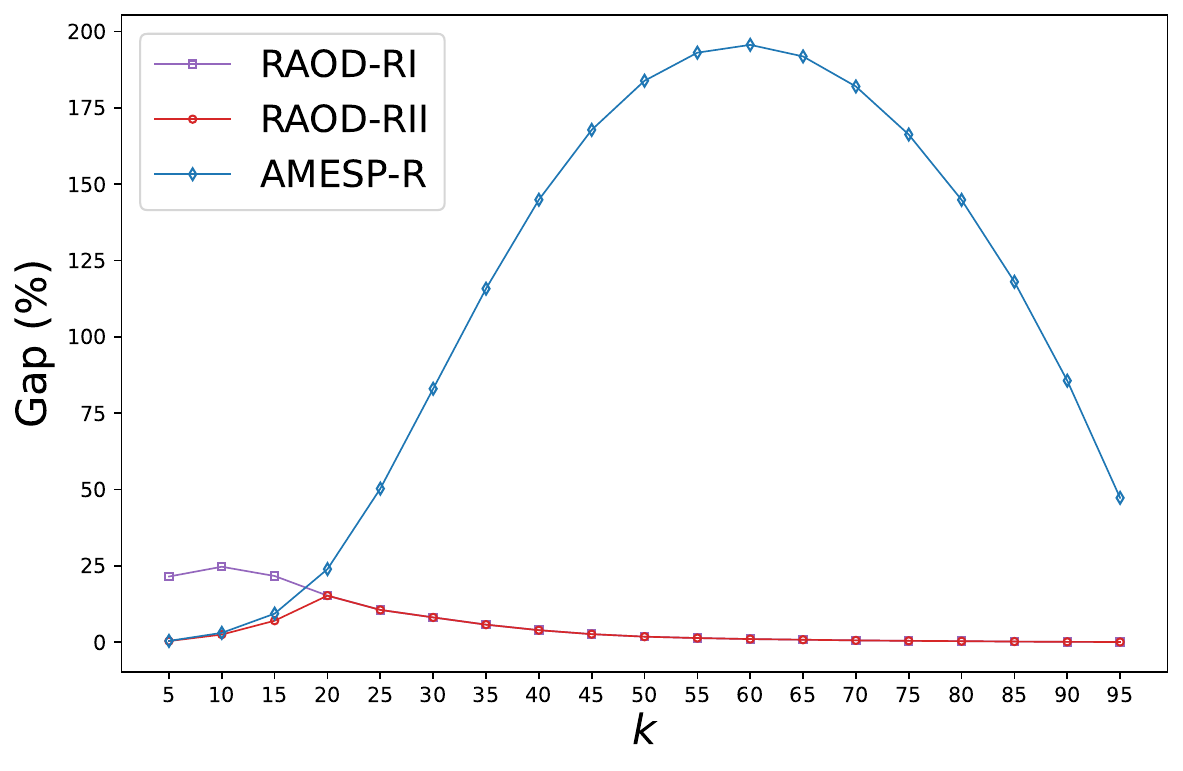}
\label{fig:100_20_10_rel_val}
	}
    \\
    \subfigure[Time for $(n,d,\lambda)=(100,20, 1)$]{
\includegraphics[width=0.3\textwidth]{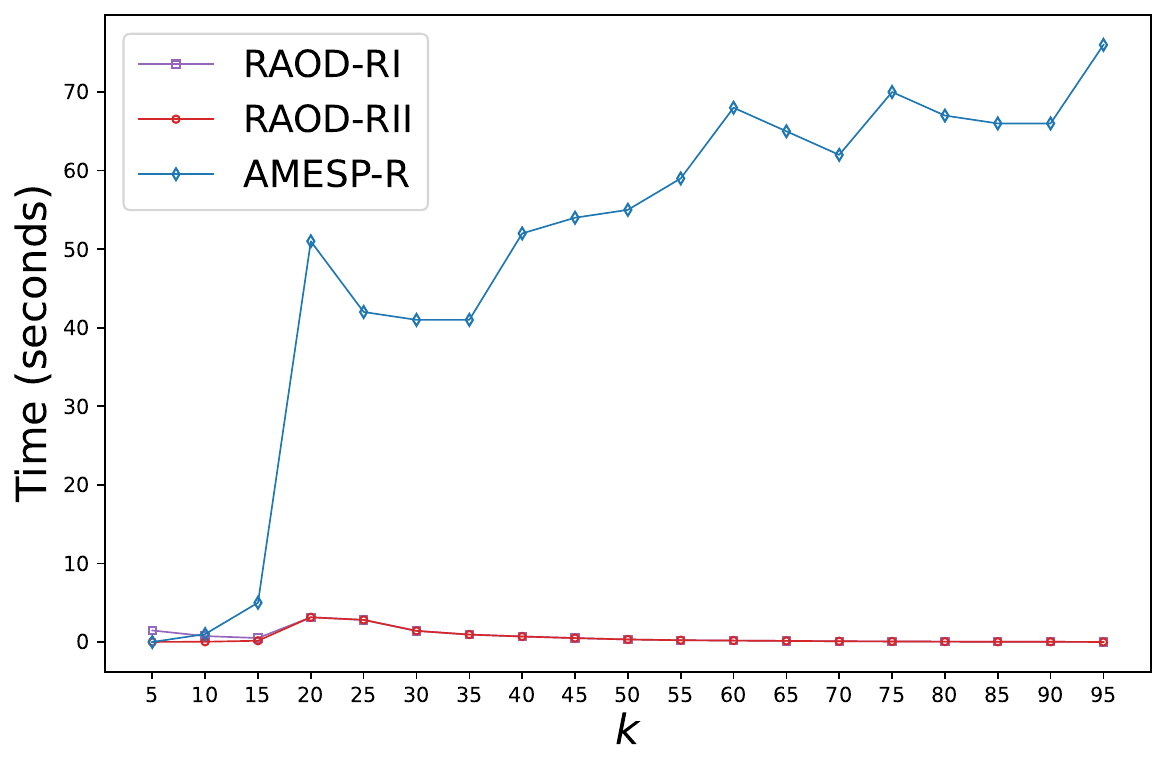}
\label{fig:100_20_rel_time}
	}
	\subfigure[Time for $(n,d,\lambda)=(100,80, 1)$]{
		\centering
\includegraphics[width=0.3\textwidth]{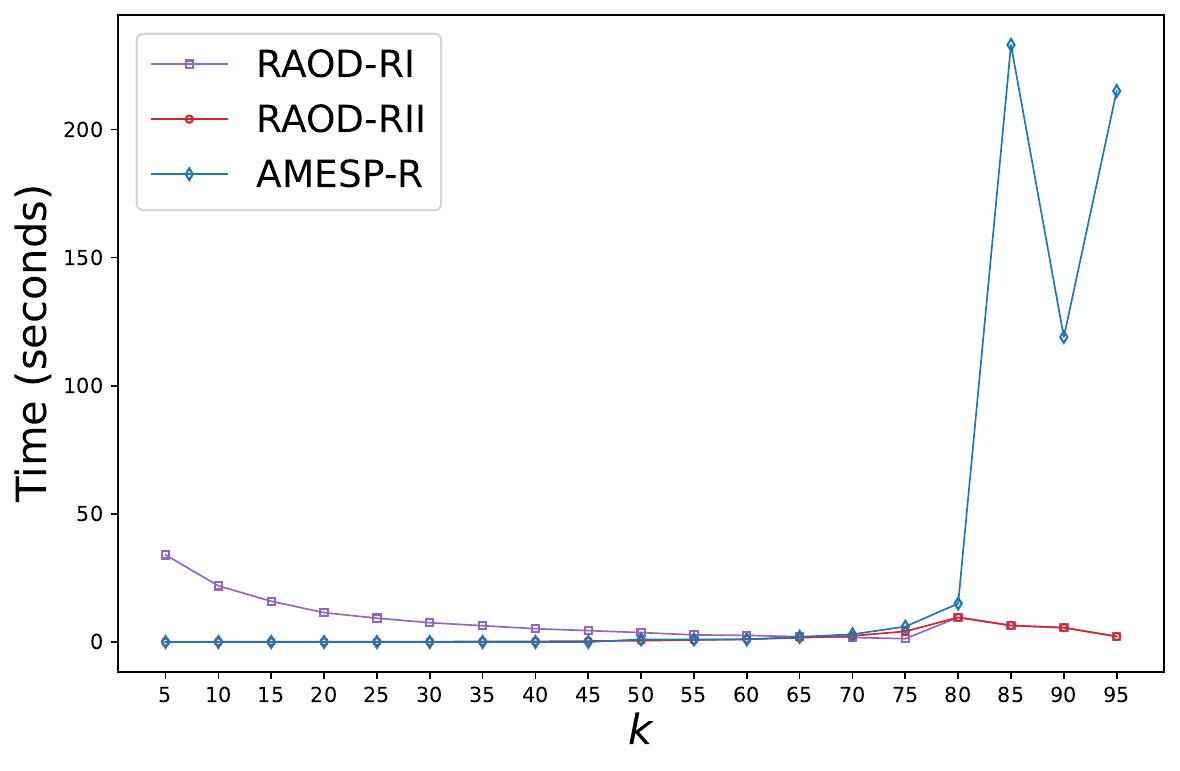}
\label{fig:100_80_rel_time}
	}
 	\subfigure[Time for $(n,d,\lambda)=(100,20, 10)$]{
		\centering
\includegraphics[width=0.3\textwidth]{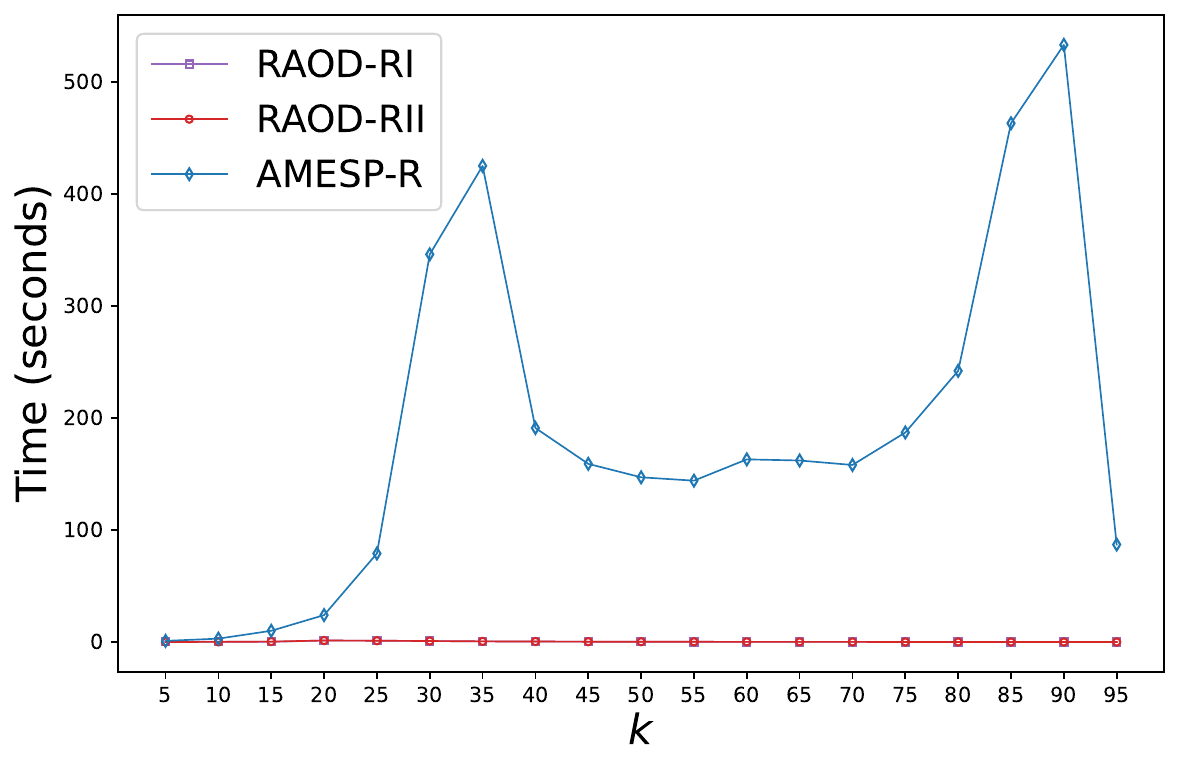}
\label{fig:100_20_10_rel_time}
	}
	\caption{Convex relaxations of \ref{aed}:  Gaps  and computation times on synthetic data  }\label{fig_synthetic_rel}
\end{figure}

\begin{figure}[ht]
	\centering
\subfigure[Gap for $(n,d,\lambda)=(159,24,1)$]{
\includegraphics[width=0.3\textwidth]{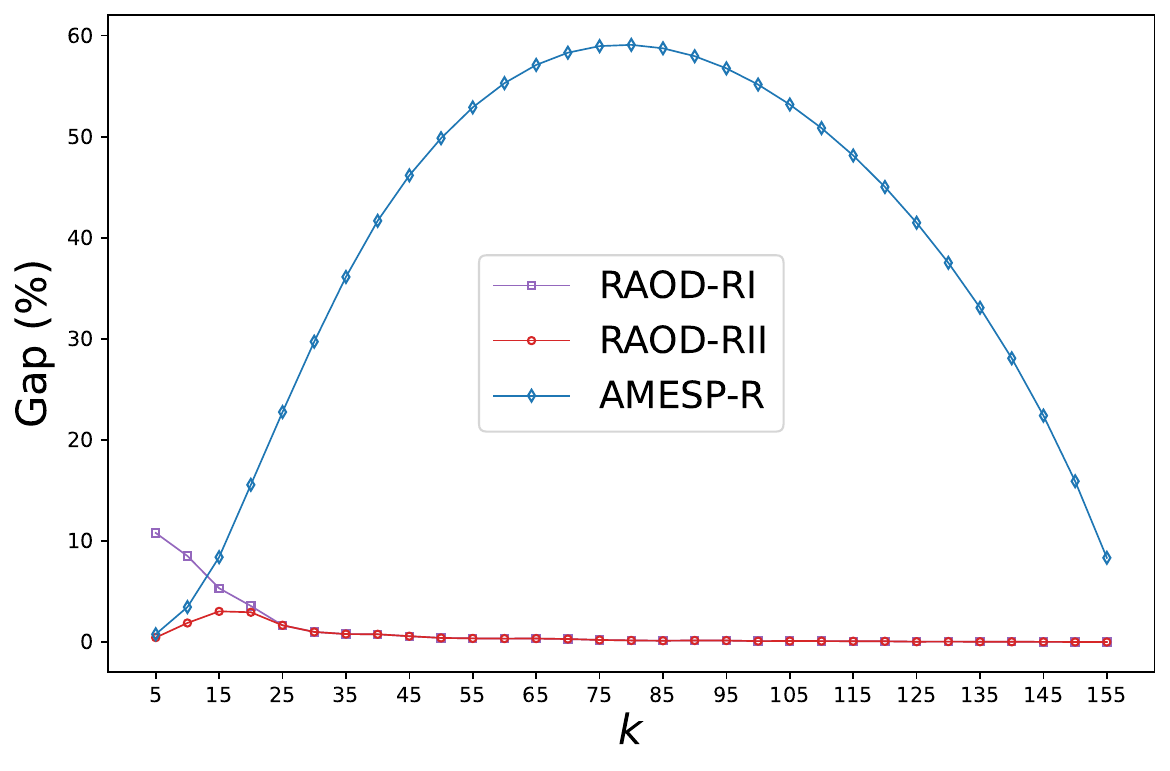}
\label{fig:159_real_rel_val}
	}
	\subfigure[Gap for $(n,d,\lambda)=(200,22,1)$]{
		\centering
\includegraphics[width=0.3\textwidth]{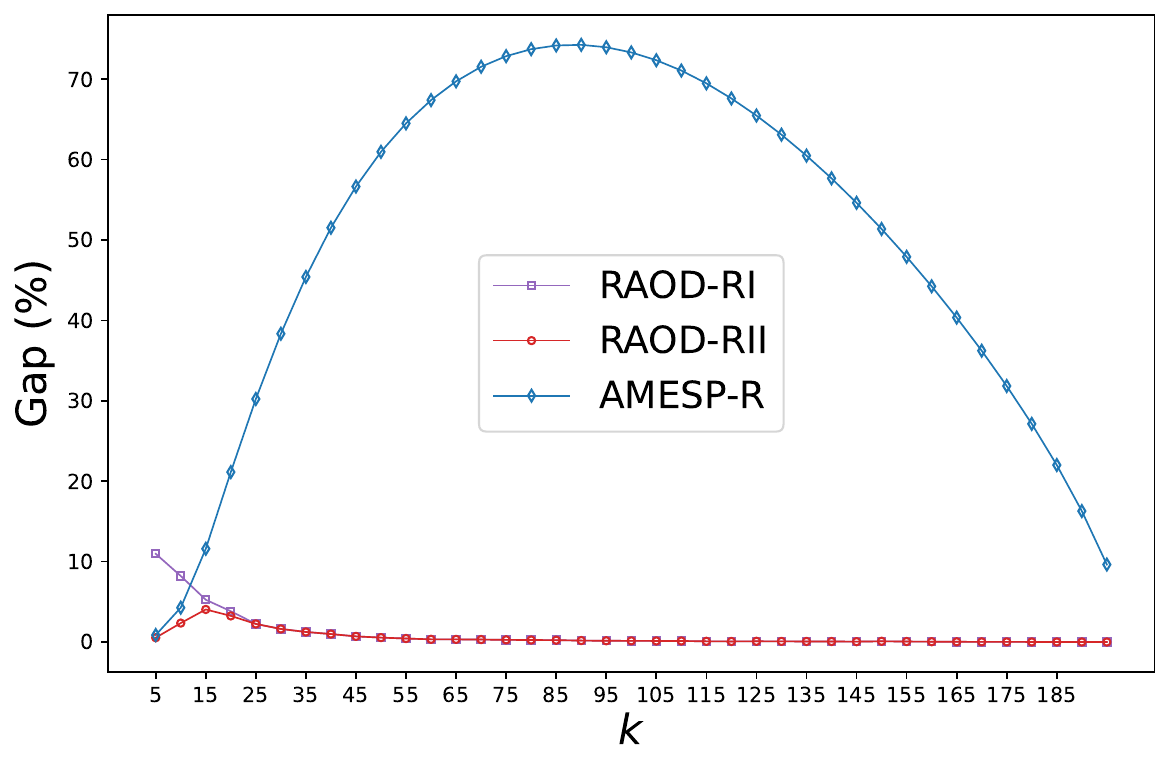}
\label{fig:200_real_rel_val}
	}
 	\subfigure[Gap for $(n,d,\lambda)=(1000,90,1)$]{
    \label{fig:1000}
		\centering
\includegraphics[width=0.3\textwidth]{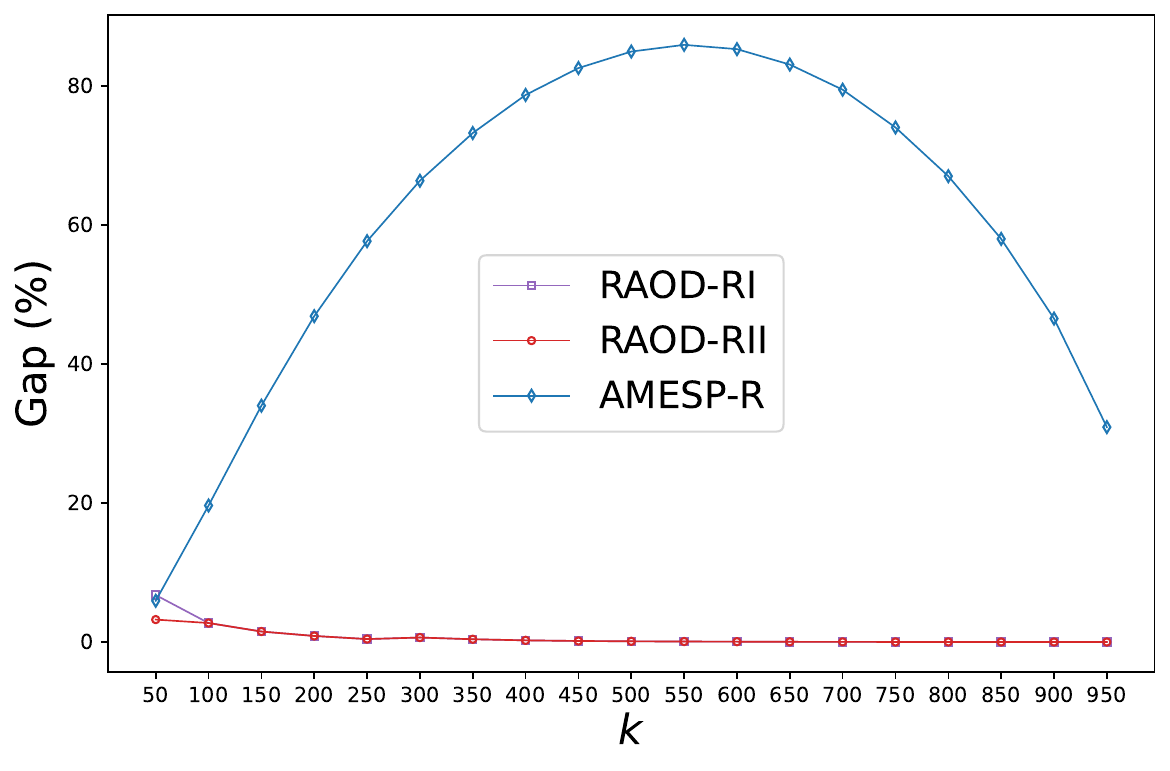}
	}
    \\
    \subfigure[Time for $(n,d,\lambda)=(159,24,1)$]{
\includegraphics[width=0.3\textwidth]{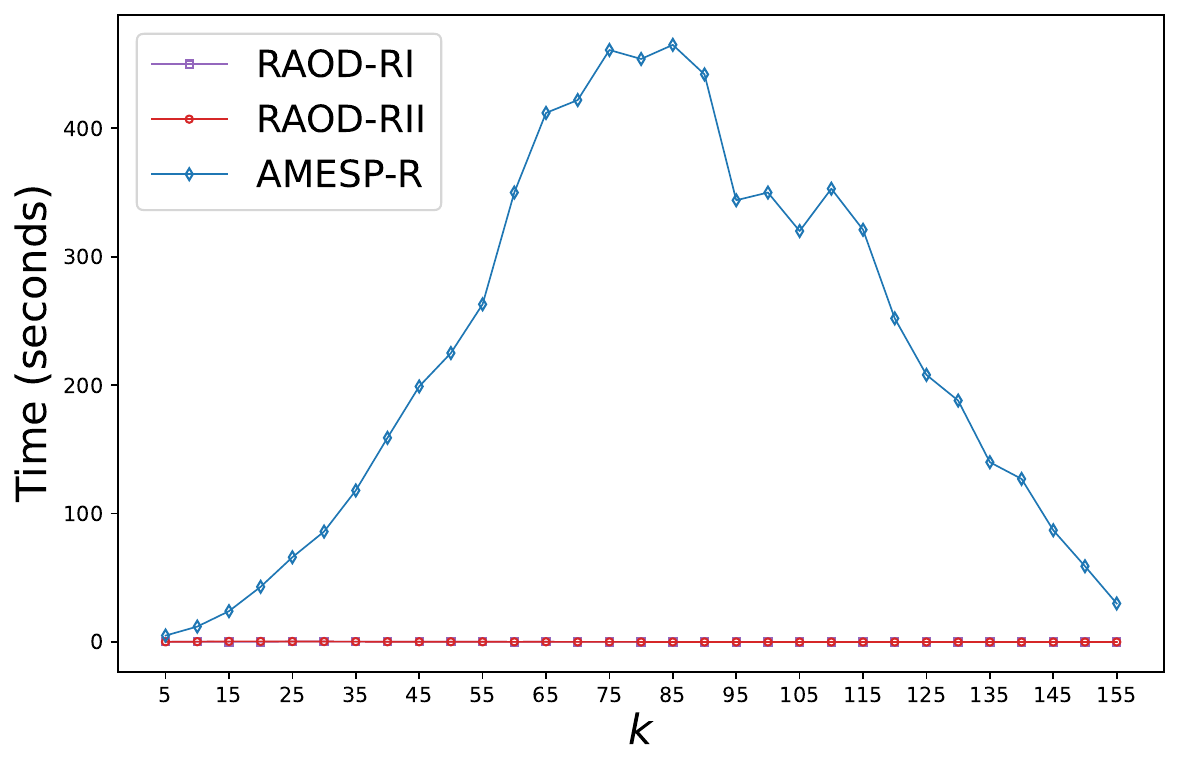}
	}
	\subfigure[Time for $(n,d,\lambda)=(200,22,1)$]{
		\centering
\includegraphics[width=0.3\textwidth]{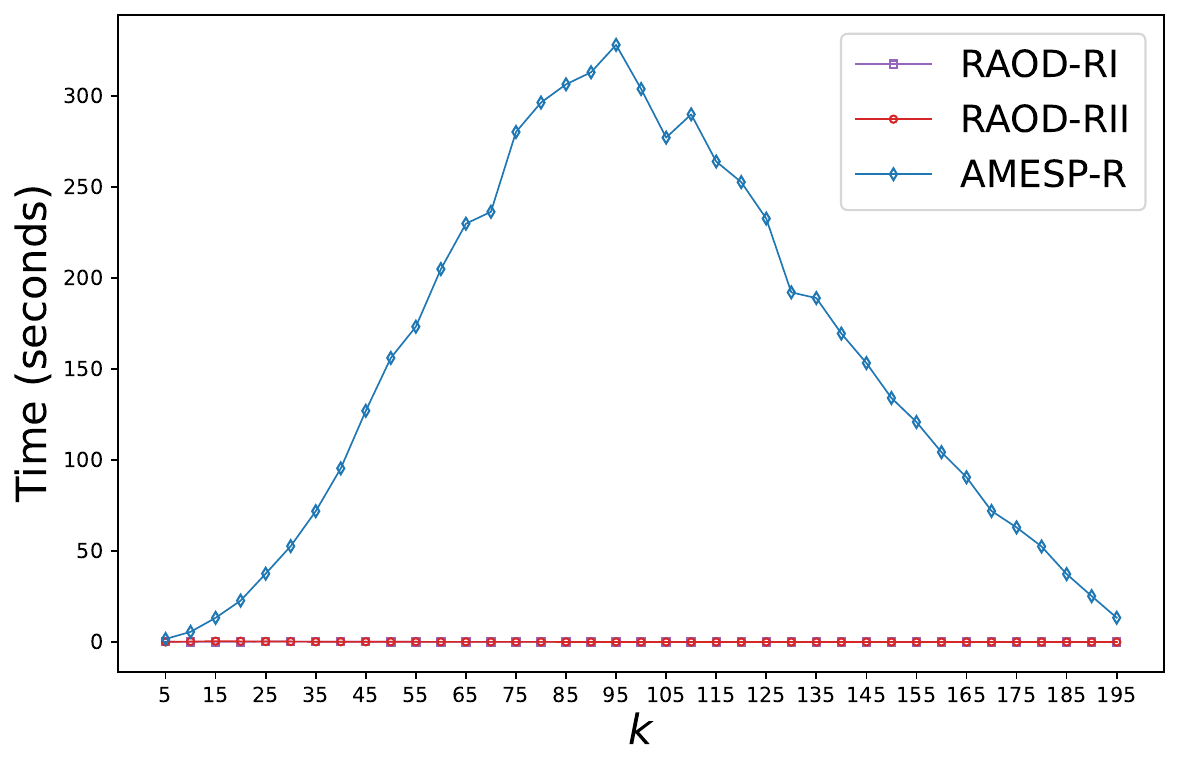}
	}
 	\subfigure[Time for $(n,d,\lambda)=(1000,90,1)$]{
		\centering
\includegraphics[width=0.3\textwidth]{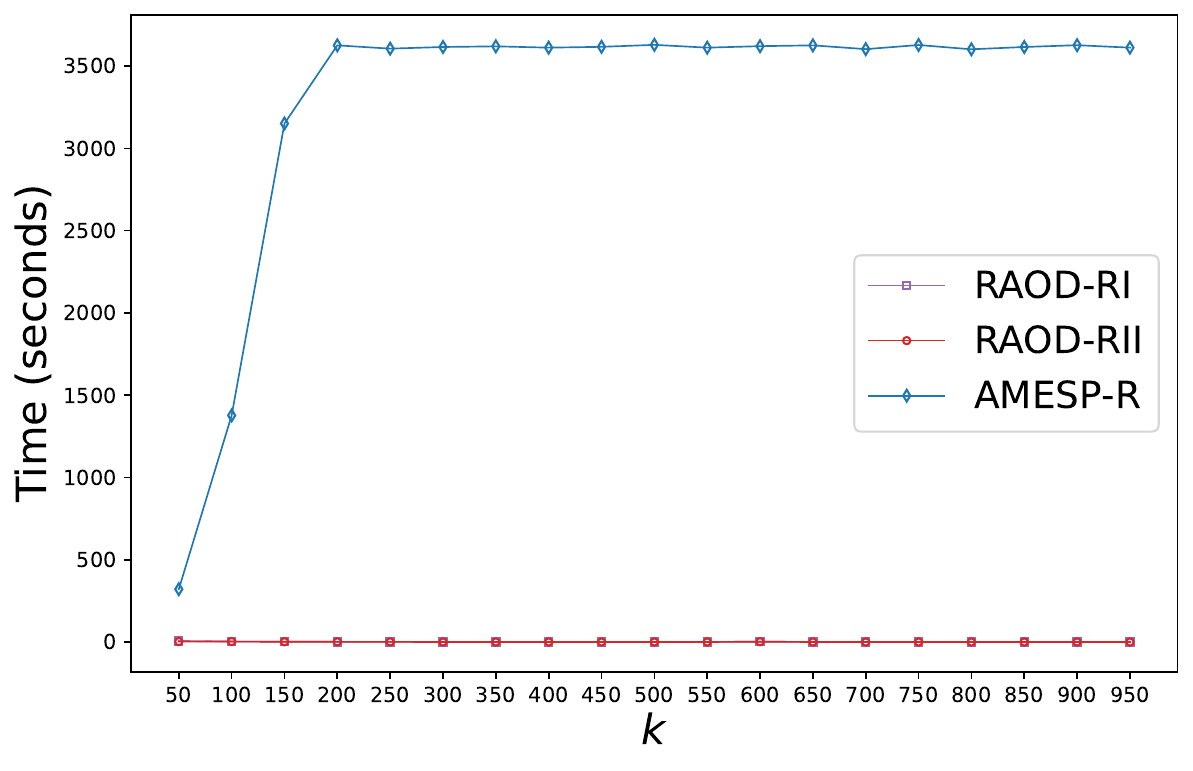}
\label{fig:1000_real_rel_time}
	}
	\caption{Convex relaxations of \ref{aed}:  Gaps  and computation times on real data}\label{fig_real_rel}
\end{figure}

\subsection{Evaluations of Approximation Algorithms}
In this subsection, we evaluate the performance of several approximation algorithms for solving \ref{aed} on the same data instances. Specifically, we compare our \Cref{algo:forward,algo:backward} with  the following methods: ``LocalSearch," referring to the local search algorithm proposed by \cite{li2024best} for \ref{amesp}; ``RegVol," the regularized volume sampling algorithm introduced by \cite{derezinski2017subsampling}; and its more efficient implementation ``FastRegVol." Local search starts from a random solution, and sampling algorithms are run 20 times per instance, with the best result reported.
Notably, we omit comparisons with the randomized sampling algorithms from \cite{li2024best, tantipongpipat2020lambda}, as they were computationally intensive and failed to finish within one hour on these instances.

The approximation gaps and computation times for the evaluated algorithms are summarized in \Cref{fig_synthetic_approx,fig_real_approx}. We track the best lower bound $\nu_k^{\text{N}}$ found from the
previous subsection, and  compute  {``Gap (\%)"} by $100\times(\text{upper bound}-\nu_k^{\text{N}})/\text{upper bound}$, where the upper bound is the objective value returned by each approximation method. We observe that combinatorial algorithms outperform the two sampling algorithms, with significant improvements  when $k>d$ and $d$ is relatively small compared to $n$. The comparison of \Cref{fig:20time,fig:80time} shows that the runtime of FastRegVol increases more rapidly as $d$ increases, compared to the other algorithms.   Across all algorithms, the largest approximation gaps tend to occur near $k=d$. This may result from using a lower bound $\nu_k^{\text{N}}$ rather than the true optimal value in computing the gap, as \ref{aed-R2} deviates most from the optimum around $k=d$ (see \Cref{fig_synthetic_rel,fig_real_rel}). In general, the three combinatorial algorithms are not directly comparable, but the local search algorithm outperforms our greedy \Cref{algo:forward,algo:backward} in more cases. However, for the largest instance with $n=1000$ in \Cref{fig:1000_approx_gap,fig:1000_approx_time}, it takes considerably longer but achieves nearly identical gaps to our algorithms. 
In summary, our \Cref{algo:forward,algo:backward} strike the best balance between scalability and solution quality among all the algorithms. As seen in \Cref{fig:200_approx}, \Cref{algo:forward} yields a smaller gap  when $k$ is small, while \Cref{algo:backward} becomes more effective when $k\ge d$. This trend aligns with their theoretical guarantees discussed following  \Cref{them:approx}.

\begin{figure}[ht]
	\centering
\subfigure[Gap for $(n,d,\lambda)=(100,20, 1)$]{
\includegraphics[width=0.3\textwidth]{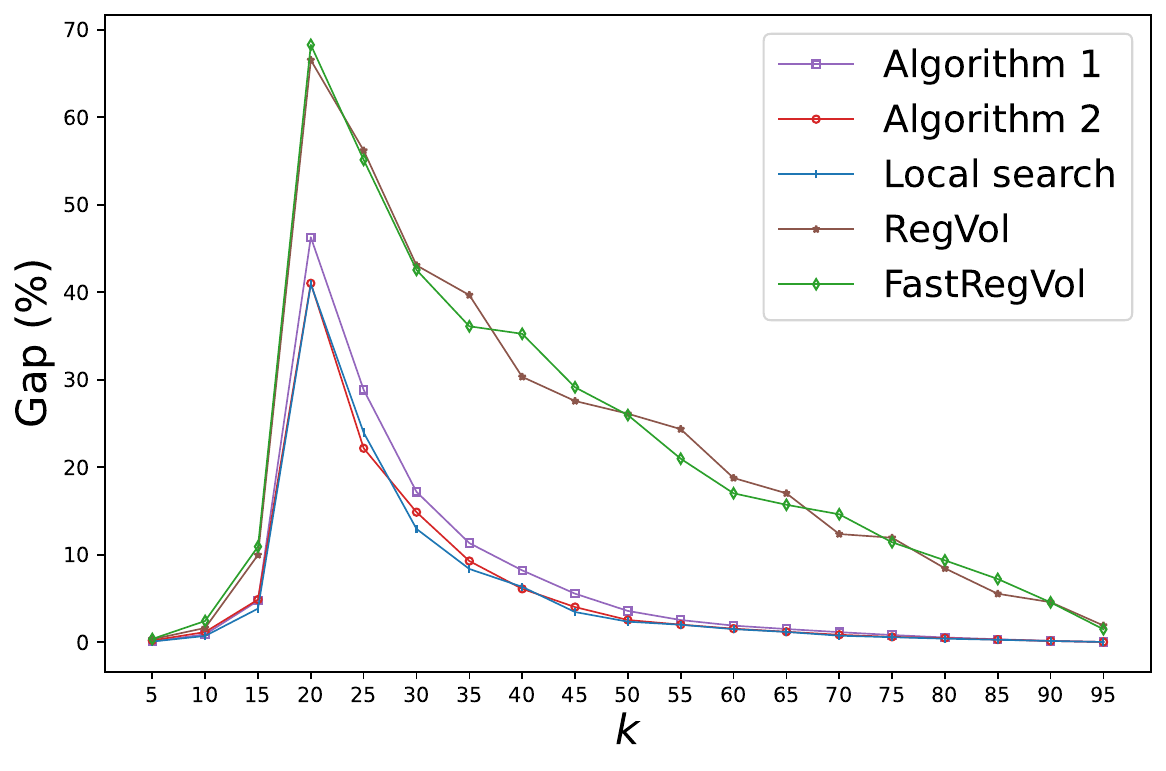}
	}
	\subfigure[Gap  for $(n,d,\lambda)=(100,80, 1)$]{
		\centering
\includegraphics[width=0.3\textwidth]{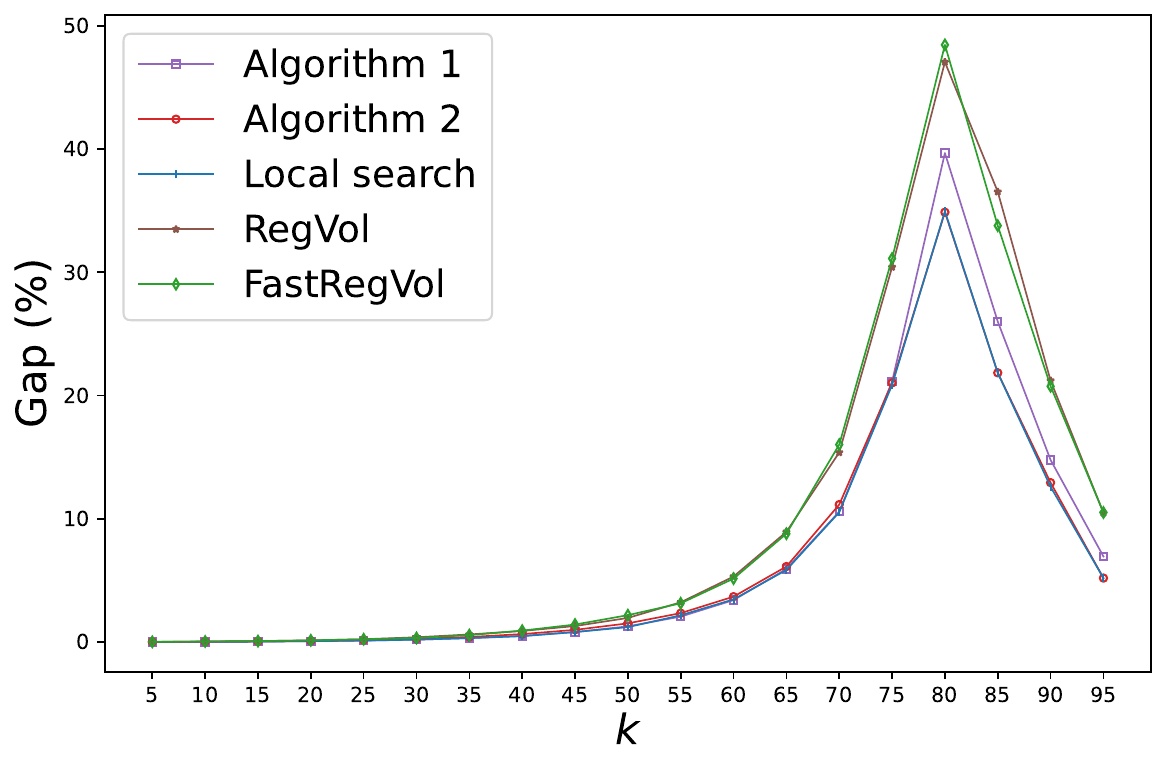}
	}
 	\subfigure[Gap  for $(n,d,\lambda)=(100,20, 10)$]{
		\centering
\includegraphics[width=0.3\textwidth]{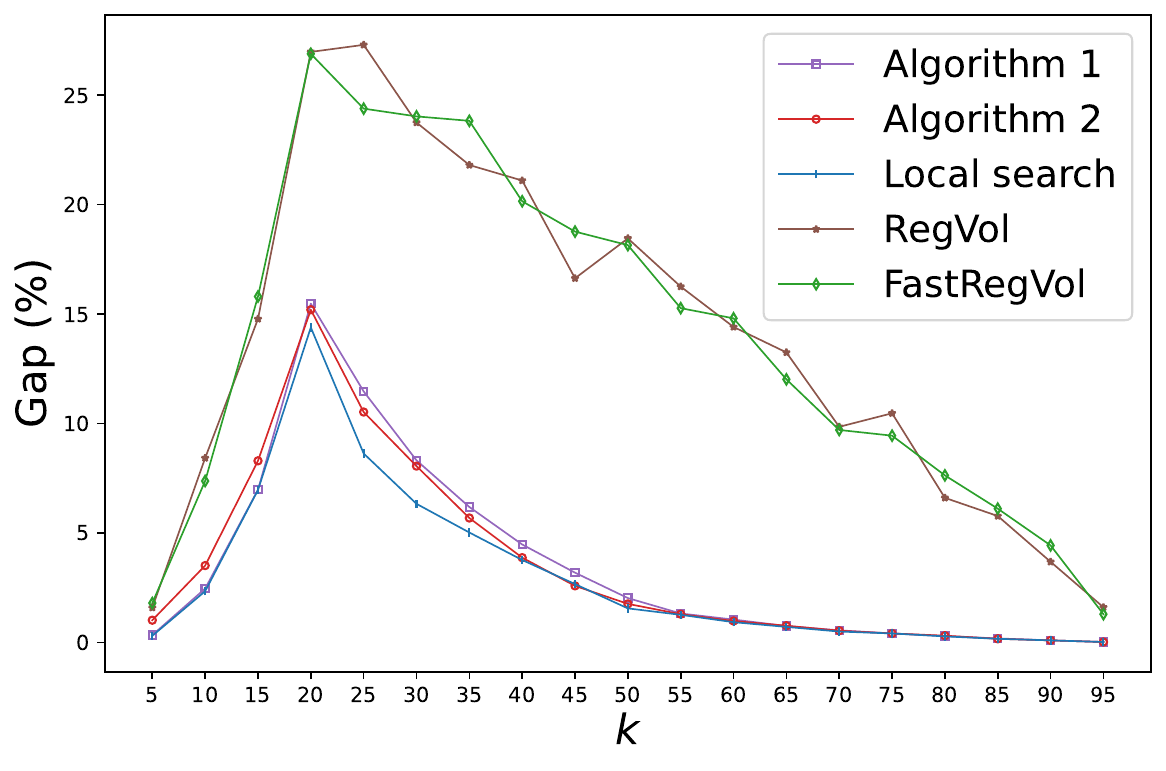}
	}
    \\
    \subfigure[Time for $(n,d,\lambda)=(100,20, 1)$]{\label{fig:20time}
\includegraphics[width=0.3\textwidth]{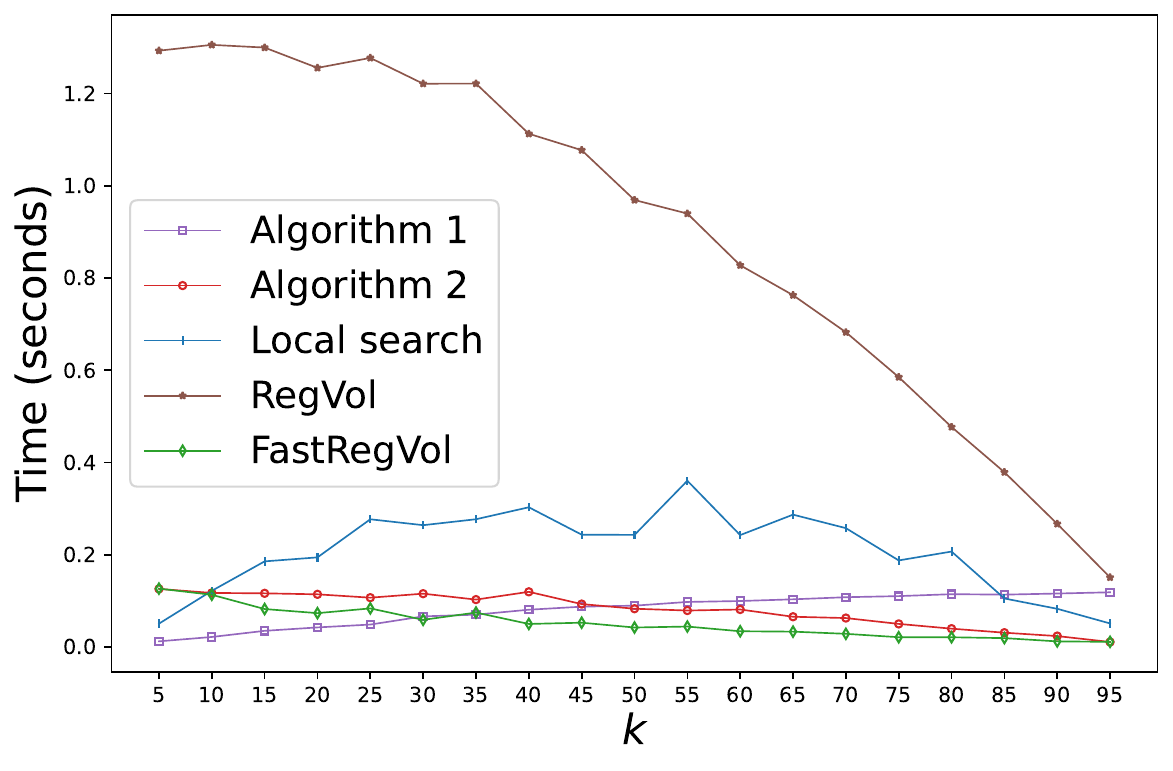}
	}
	\subfigure[Time for $(n,d,\lambda)=(100,80, 1)$]{\label{fig:80time}
		\centering
\includegraphics[width=0.3\textwidth]{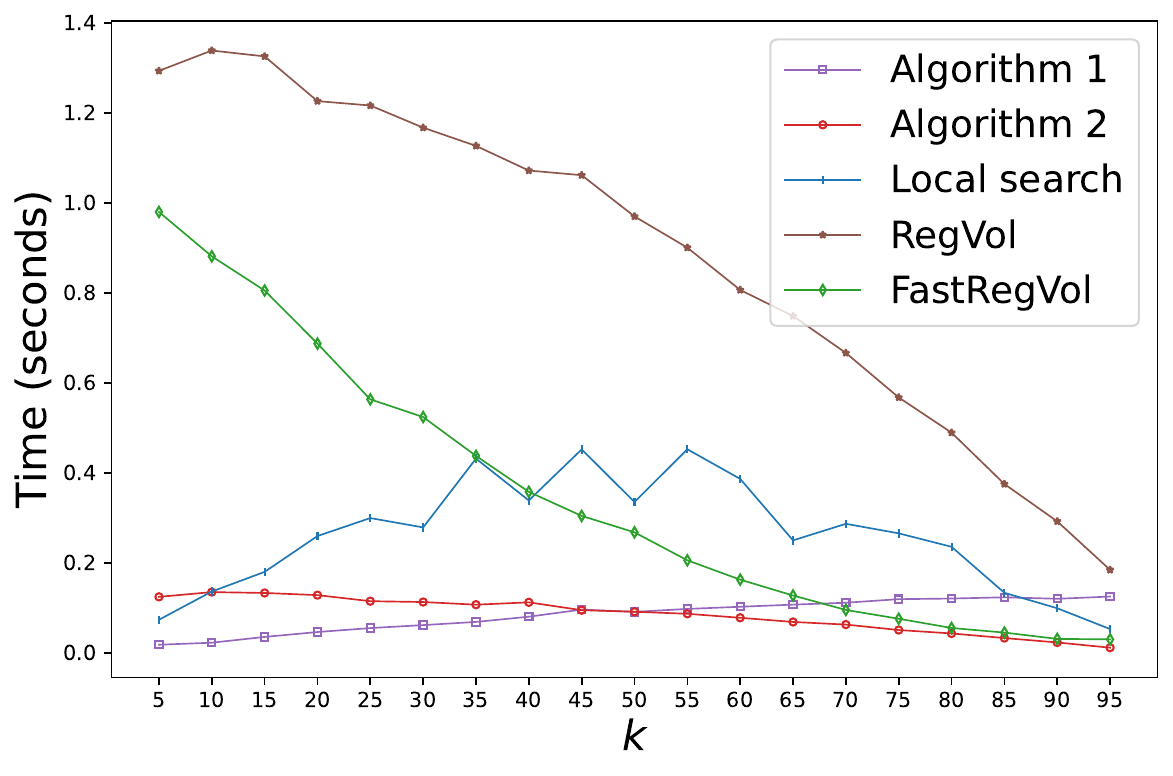}
	}
 	\subfigure[Time for $(n,d,\lambda)=(100,20, 10)$]{
		\centering
\includegraphics[width=0.3\textwidth]{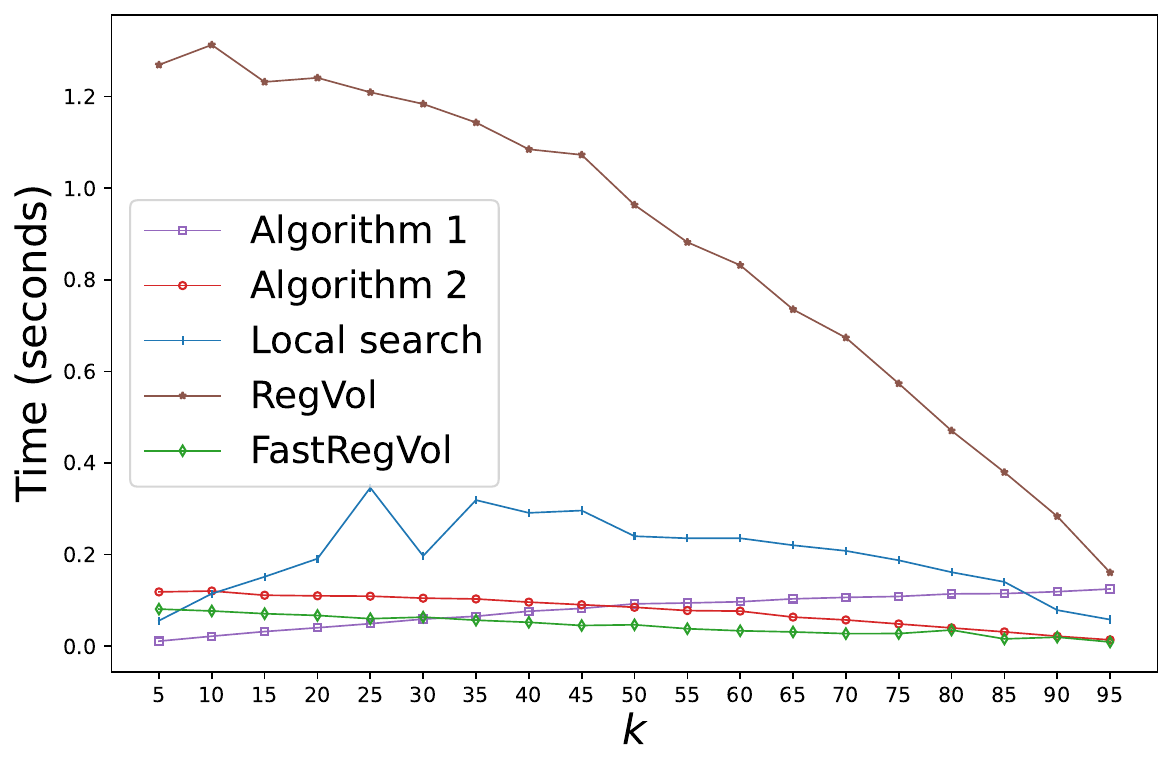}
	}
\caption{Approximation algorithms of \ref{aed}: Gaps  and computation times on synthetic data }\label{fig_synthetic_approx}
\end{figure}

\begin{figure}[ht]
	\centering
\subfigure[Gap for $(n,d,\lambda)=(159,24,1)$]{
\includegraphics[width=0.3\textwidth]{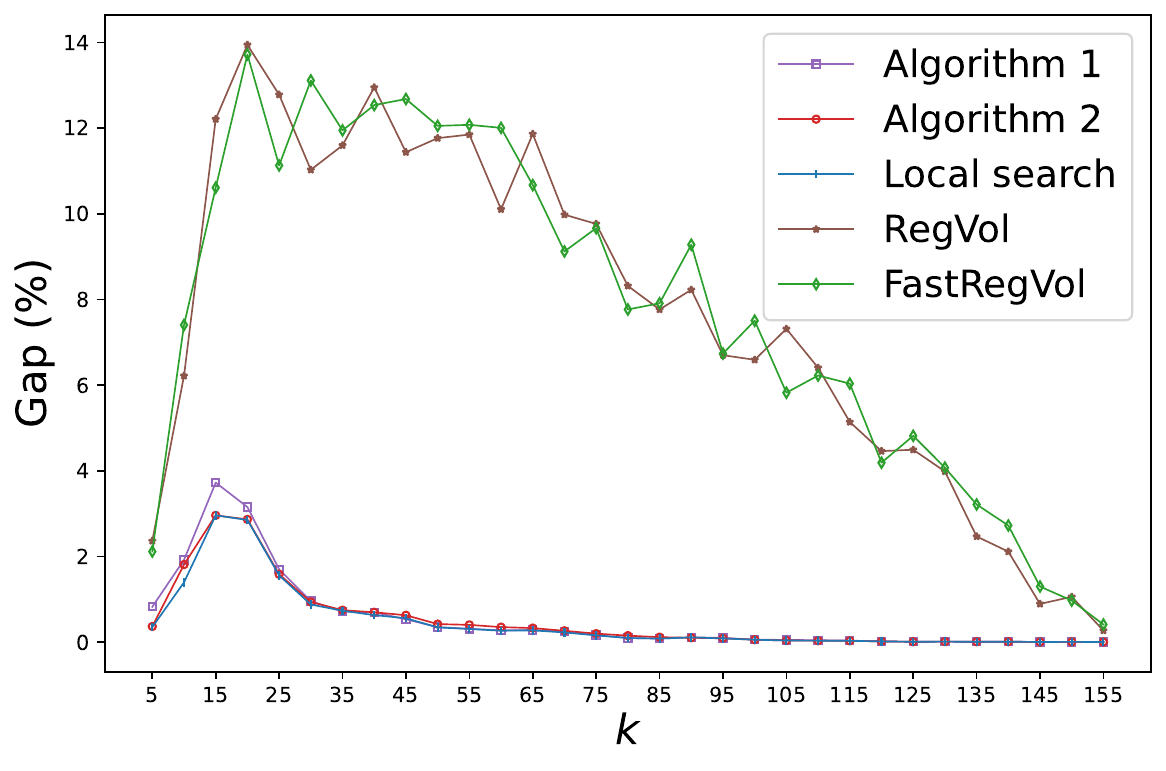}
	}
	\subfigure[Gap  for $(n,d,\lambda)=(200,22,1)$]{
    \label{fig:200_approx}
		\centering
\includegraphics[width=0.3\textwidth]{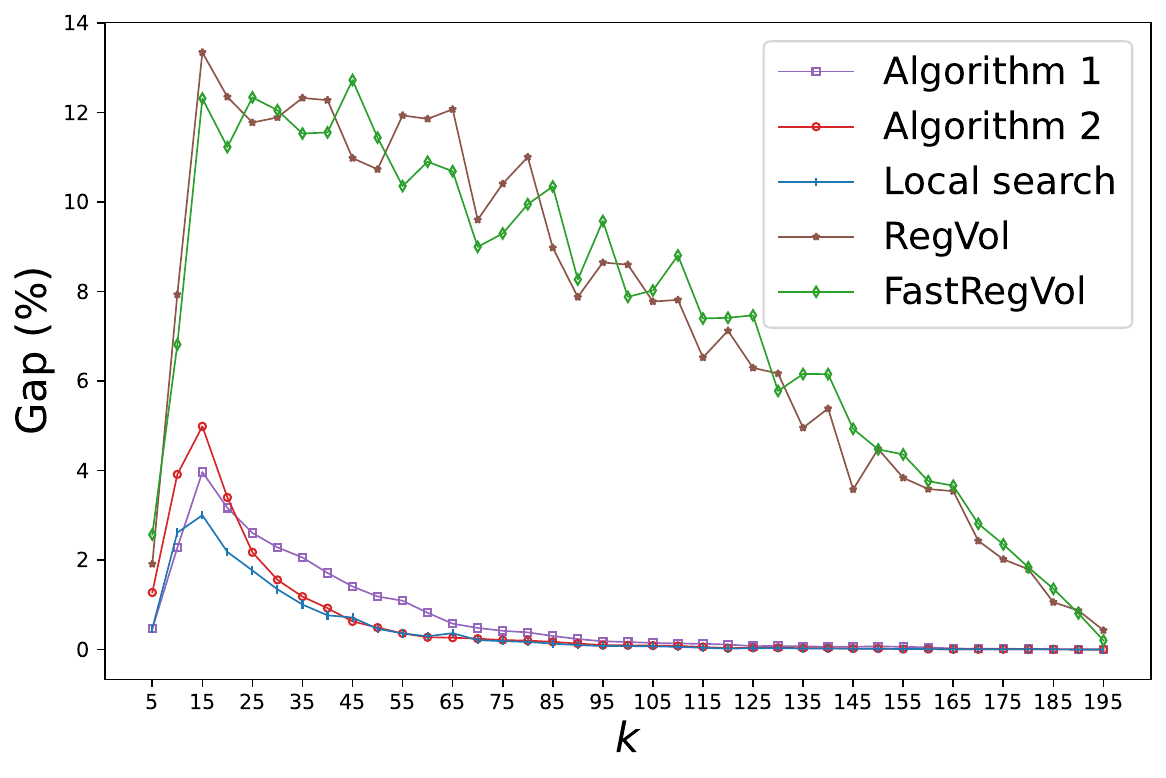}
	}
 	\subfigure[Gap  for $(n,d,\lambda)=(1000,90,1)$]{
    \label{fig:1000_approx_gap}
		\centering
\includegraphics[width=0.3\textwidth]{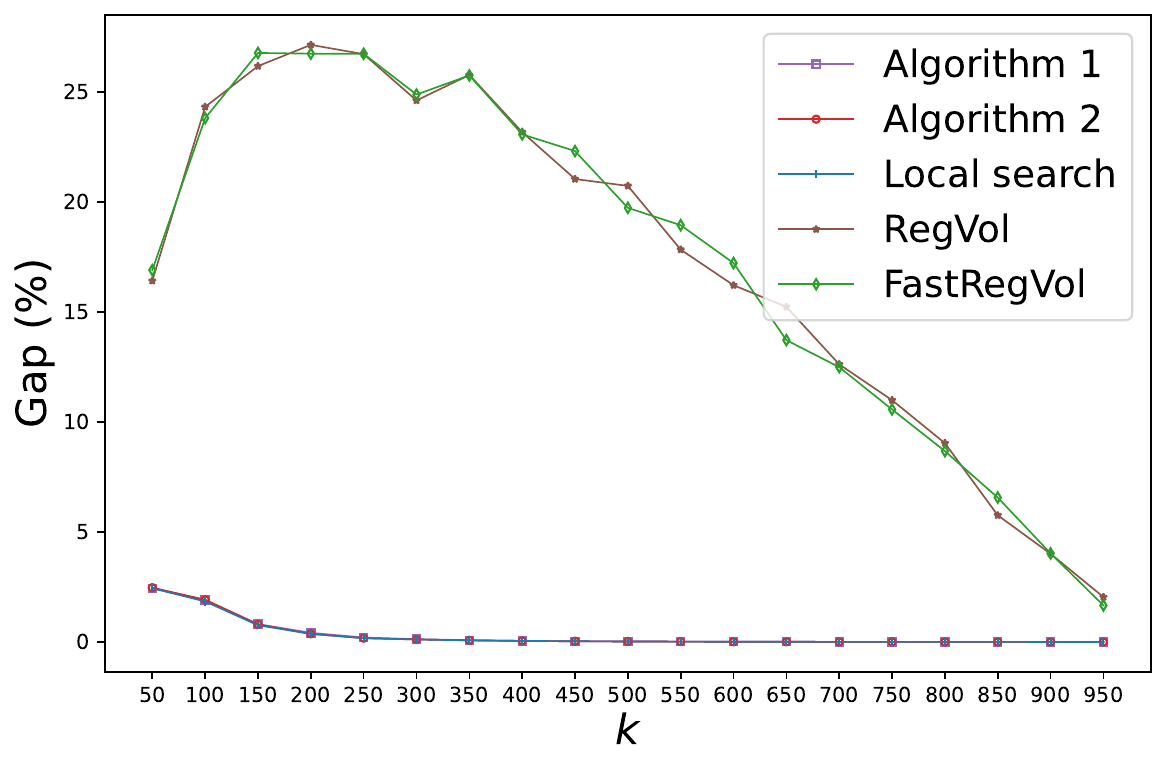}
	}
    \\
    \subfigure[Time for $(n,d,\lambda)=(159,24,1)$]{
\includegraphics[width=0.3\textwidth]{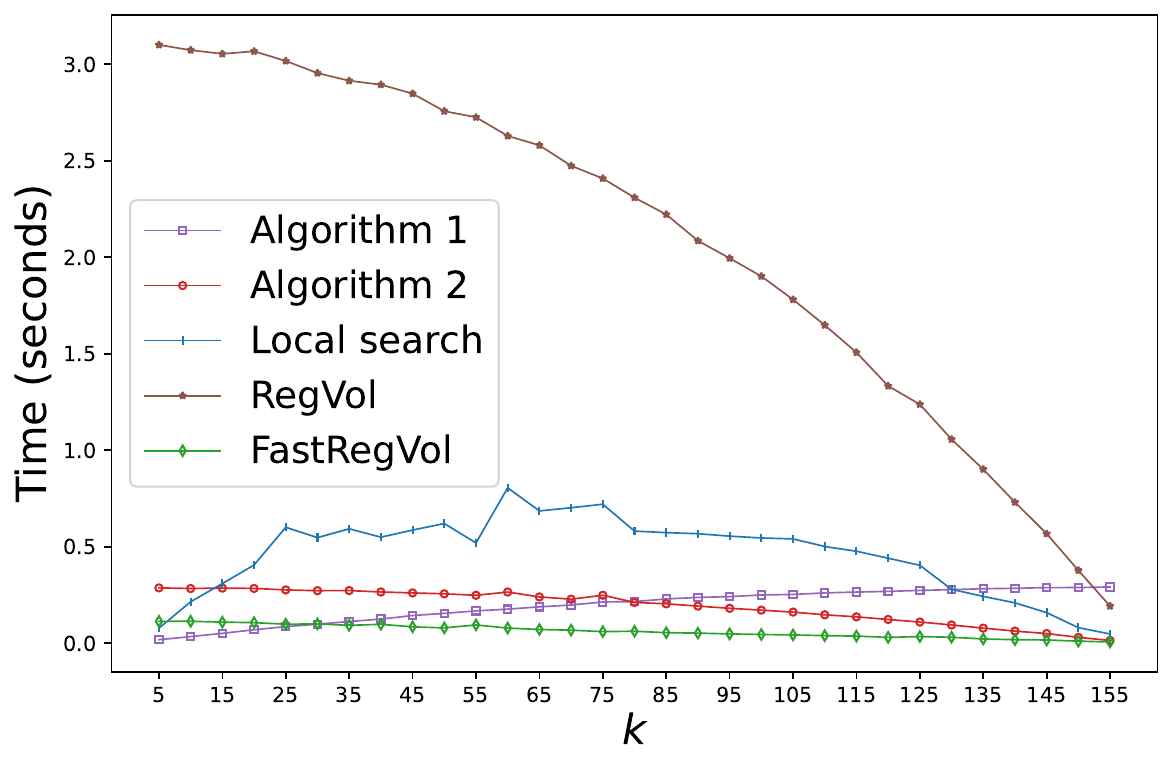}
	}
	\subfigure[Time for $(n,d,\lambda)=(200,22,1)$]{
		\centering
\includegraphics[width=0.3\textwidth]{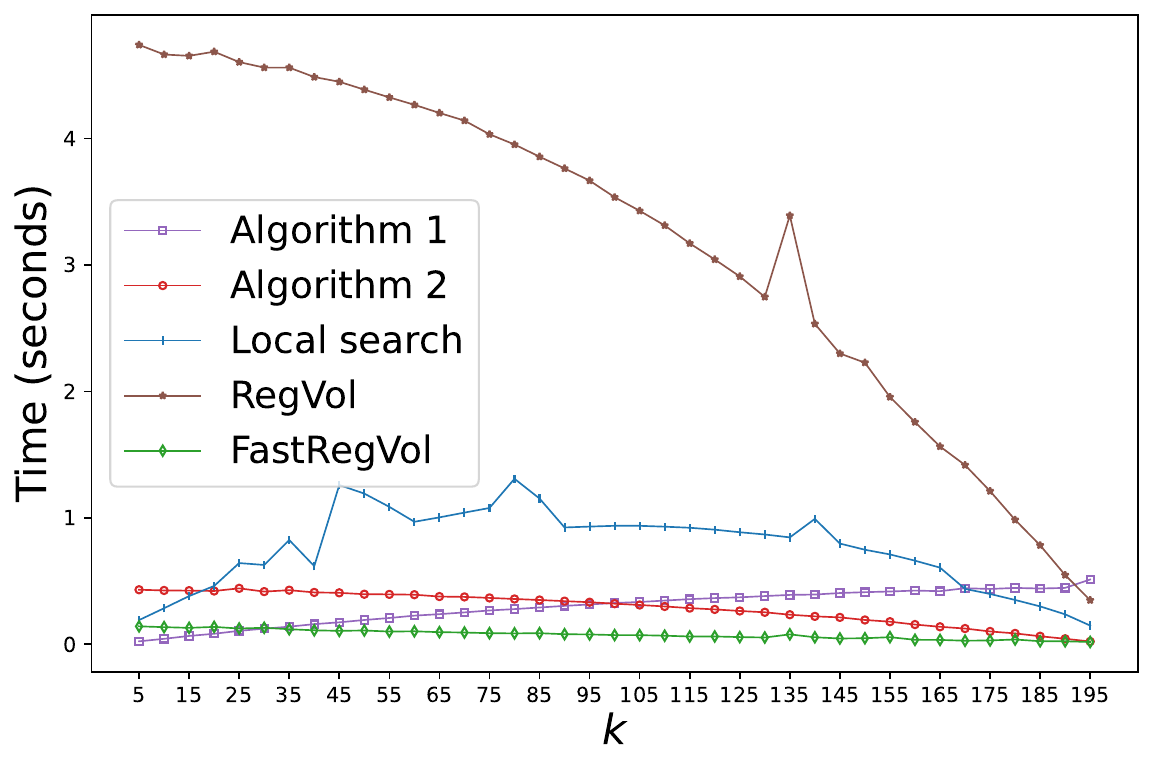}
	}
 	\subfigure[Time for $(n,d,\lambda)=(1000,90,1)$]{
    \label{fig:1000_approx_time}
		\centering
\includegraphics[width=0.3\textwidth]{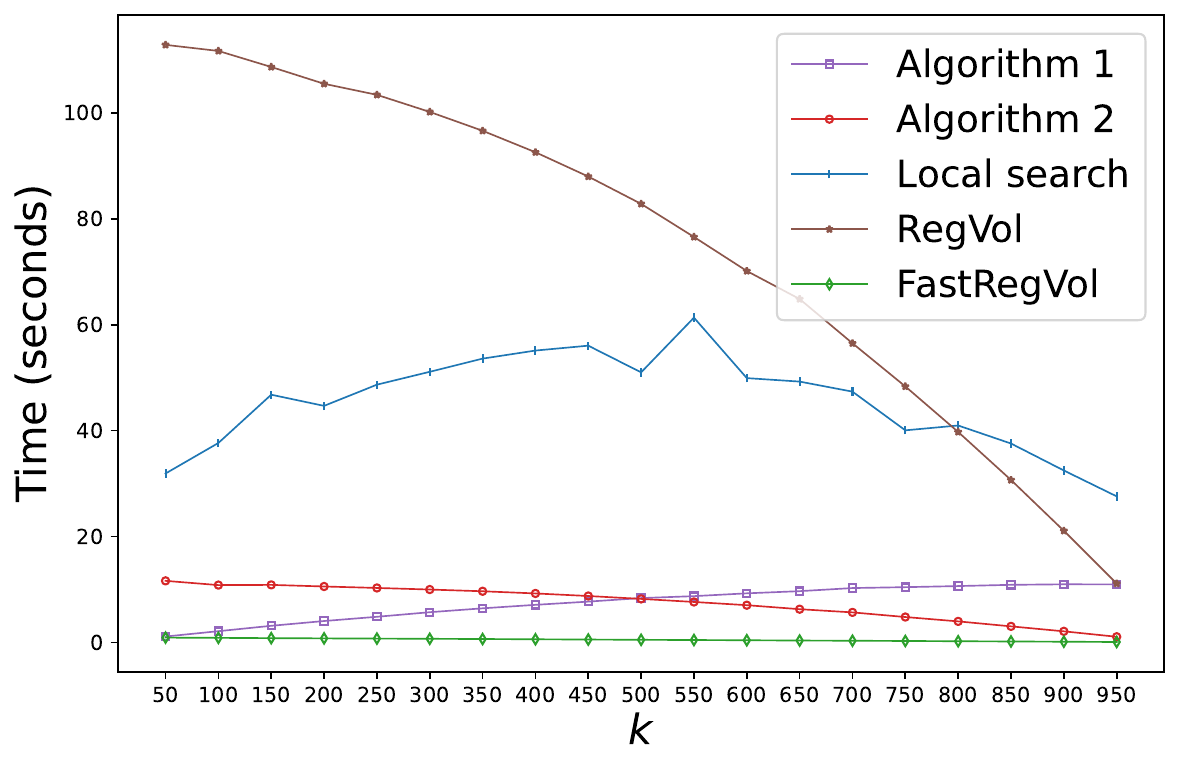}
	}
	\caption{Approximation Algorithms of \ref{aed}: Gaps  and computation times on real data }\label{fig_real_approx}
\end{figure}

\subsection{Performance of Exact Algorithms}\label{subsec:exact}
We now compare the three (mixed-)integer convex formulations introduced in \Cref{sec:Method,subsec:augmesp} using both synthetic and real datasets.
To solve our proposed formulation \eqref{eq:novel}, we develop the cutting-plane \Cref{algo:cut}, which we implement using Gurobi's \textit{lazy} constraints.  The same cutting-plane scheme  can be readily extended to \eqref{aed1}, since its objective function is also convex; we refer to this variant as ``Cutting plane via \eqref{aed1}." For fairness, both algorithms employ the same computational enhancements and termination criterion, as detailed in Subsection \ref{subsec:cut}. 
 In addition, we use Gurobi to  solve the formulation \ref{aed_miscop} directly. Although \ref{amesp-R} offers an exact formulation when $\bm x$ is binary, we do not include it in our comparison due to its high computational cost and the fact that \ref{aed-R2} theoretically dominates \ref{amesp-R} (see \Cref{them:compare}, Part (iii)).

 We restrict our experiments to the case $k < d$. From a theoretical perspective,  the objective functions of \eqref{aed1} and \eqref{eq:novel} are equivalent when $k \geq d$, as indicated by \Cref{lem:gamma},  making the comparison trivial. From a practical perspective, this choice is also well-motivated.
In high-dimensional application problems such as user cold-start recommendation systems, constraints such as limited budgets or user cognitive costs naturally enforce small values of $k$ \citep{chamon2017approximate}.

For each instance with the fixed pair $(n,d)$,  we vary the parameters $\lambda\in\{1, 10\}$ and $k\in\{5,10\}$, thus creating four testing cases.
The numerical results are presented in \Cref{table:exact_synthetic,table:exact_real}, where  ``{MIPGap(\%)}" represents  the relative optimality gap in percentage at which Gurobi terminates. A MIPGap of $0.00$ indicates a certified optimal solution.  Among the three exact algorithms, the cutting-plane algorithm based on \eqref{aed1} attains the largest MIPGaps in most instances and only occasionally reaches optimality when $\lambda=10$, as seen in \Cref{table:exact_real}. \ref{aed_miscop} also struggles with large MIPGaps and timeouts. In particular,  for instances with $n\ge 500$, \ref{aed_miscop} may even exhaust available memory, marked by ``--" in \Cref{table:exact_real}. In contrast, our proposed \Cref{algo:cut} achieves significantly smaller MIPGaps and solves most cases to optimality within 1 hour,  especially for $k=5$.
These results show the superiority of \Cref{algo:cut} in both computational efficiency and MIPGap reduction for the tested instances, which can be primarily attributed to the strength of the continuous relaxation of \eqref{eq:novel}. For small $k<d$, its relaxation \ref{aed-R2} is strictly tighter than \ref{aed-R1}, the continuous relaxation of \eqref{aed1} and \ref{aed_miscop}.

Next, we would like to comment on the effect of parameters $d$, $k$, and $\lambda$. What is interesting in \Cref{table:exact_synthetic} is that, for fixed $n, k,\lambda$, the MIPGap for \Cref{algo:cut} decreases as the experimental dimension $d$ increases. For instance,  \Cref{algo:cut} solves the case with $(n,d,k,\lambda) = (100, 80, 10, 1)$ in just $7$ seconds, whereas the lower-dimensional case $(100,20,10,1)$ takes longer than $1$ hour. A likely explanation is that, according to \Cref{them:gap} and the discussion following \Cref{cor:amesp_gap},  \ref{aed-R2} admits an optimality gap of $(d-1)/(d-k)$  when  $k\le (n+1)/2$. Since this theoretical gap is decreasing with $d$ (for fixed $k$), \ref{aed-R2} becomes tighter in higher dimensions, thus enhancing the performance of \Cref{algo:cut}. This observation highlights the advantage of our \Cref{algo:cut}  in complex, high-dimensional problems. As expected, instances with $k=10$ are harder to solve than those with $k=5$.
We also observe that ``Cutting plane via \eqref{aed1}" and \ref{aed_miscop} benefit from higher $\lambda$ but still fail to close the gap in many instances. 
 In contrast,  the effect of $\lambda$ on the performance of \Cref{algo:cut} is instance-dependent.
 \Cref{table:exact_synthetic} shows that the case with $\lambda=10$ poses larger MIPGaps for \Cref{algo:cut} than the case with $\lambda=1$, whereas in \Cref{table:exact_real}, the opposite holds.

Finally, we evaluate the approximation quality of \Cref{algo:forward_backward} by comparing its outputs with the best lower bounds returned by \Cref{algo:cut}. As shown in the ``Gap(\%)" column of \Cref{table:exact_synthetic,table:exact_real}, this gap remains within $2.45\%$ and is computed in under $11$ seconds, which demonstrates both the scalability and near-optimality of \Cref{algo:forward_backward}. Moreover, in cases where this gap is smaller than the MIPGap of \Cref{algo:cut}, it indicates that \Cref{algo:cut} failed to find a better feasible solution within the time limit than the one produced by \Cref{algo:forward_backward}.

\begin{table}[htbp]
\caption{Exact algorithms for \ref{aed} on synthetic data} 
\centering
\label{table:exact_synthetic}
\setlength{\tabcolsep}{1.5pt}\renewcommand{\arraystretch}{0.9}
\begin{tabular}{c|c|r|r|r| r| r| r|r|r|r}
\hline
\multicolumn{1}{c|}{{Dataset}}&\multicolumn{2}{c|}{{Params}}  & \multicolumn{2}{c|}{{Cutting plane via \eqref{aed1}}} & \multicolumn{2}{c|}{{\ref{aed_miscop}}} & \multicolumn{2}{c|}{\Cref{algo:cut} via \eqref{eq:novel}} & \multicolumn{2}{c}{{\Cref{algo:forward_backward}}}\\
\hline
 \multicolumn{1}{c|}{($n$, $d$)}  & \multicolumn{1}{c|}{$\lambda$}  & \multicolumn{1}{c|}{$k$} & 
\multicolumn{1}{c|}{MIPGap(\%)}& 
\multicolumn{1}{c|}{Time}&  \multicolumn{1}{c|}{MIPGap(\%)}& 
\multicolumn{1}{c|}{Time} &
\multicolumn{1}{c|}{MIPGap(\%)}& 
\multicolumn{1}{c|}{Time} & \multicolumn{1}{c|}{Gap(\%)} & \multicolumn{1}{c}{Time}\\
\hline
\multirow{4}{*}{(100, 20)} & 1 & 5 & 77.92& 3600 & 76.08&3600 & 0.00 & 589 & 0.02 &1\\
&1 & 10 & 82.54 & 3600 &80.14 &3600 & 0.82 & 3600 & 0.88 & 1\\
  & 10 & 5& 21.17 & 3600 & 18.15&3600 & 0.00 & 6 & 0.04 & 1\\
&10 & 10&  27.70 & 3600 &22.63 &3600 & 2.65 & 3600 & 2.45 & 1\\
\hline
\multirow{4}{*}{(100, 50)} & 1 & 5&  75.43 & 3600 &73.18 &3600 & 0.00 &1 & 0.00 &1\\
&1 & 10&  83.57 & 3600 &81.80 &3600 & 0.04 & 3600 & 0.03 &1\\
  & 10 & 5& 25.80 & 3600 &22.25 &3600 & 0.00 & 1 & 0.00 &1\\
&10 & 10 & 36.68 & 3600 &31.40 &3600 & 0.25 & 3600 & 0.23 &1\\
\hline
\multirow{4}{*}{(100, 80)} & 1 & 5& 69.19 & 3600 &66.73 &3600 &0.00 & 1 & 0.00 &1\\
&1 & 10 & 78.01 & 3600 &75.20 &3600 & 0.00 & 7 & 0.00 &1\\
  & 10 & 5& 25.78 & 3600 &23.60 &3600 & 0.00 & 1&0.00 &1\\
&10 & 10 &36.04 & 3600 &32.57 &3600 & 0.07 & 3600 &0.06 &1\\
\hline
\multirow{4}{*}{(150, 20)} & 1 & 5&  79.95 & 3600 &78.40 &3600 & 0.00 & 8 & 0.00 &1\\
&1 & 10 & 84.63 & 3600&82.14 &3600 & 0.47 & 3600 & 0.59 &1\\
  & 10 & 5& 23.41 & 3600 &21.88 &3600 & 0.00 & 2 & 0.00 &1\\
&10 & 10 & 30.58 & 3600 & 24.48&3600 & 1.79 & 3600 & 1.65 &1\\
\hline
\multirow{4}{*}{(150, 50)} & 1 & 5& 78.67 & 3600 & 76.91&3600 & 0.00 & 1 & 0.00 &1\\
&1 & 10&  86.36 & 3600 & 85.03&3600 & 0.02 & 3600 & 0.02 &1\\
  & 10 & 5& 28.07 & 3600 &25.79 &3600 & 0.00 & 1 & 0.00 &1\\
&10 & 10& 39.70 & 3600 &35.77 &3600 & 0.11 & 3600 & 0.11 &1\\
\hline
\multirow{4}{*}{(150, 80)} & 1 & 5& 75.18 & 3600 &73.05 &3600 & 0.00&1 & 0.00 &1\\
&1 & 10& 83.70 & 3600 & 82.13&3600 & 0.00 & 1 & 0.00 &1\\
  & 10 & 5& 28.69 & 3600 & 25.63&3600 & 0.00 & 1 &0.00 & 1\\
&10 & 10& 40.13 & 3600 &36.23 &3600 & 0.04 & 3600 &0.04 &1\\
\hline
\end{tabular}
\end{table}

\begin{table}[htbp]
\caption{Exact algorithms for \ref{aed} on real data} 
\centering
\label{table:exact_real}
\setlength{\tabcolsep}{1.5pt}\renewcommand{\arraystretch}{0.9}
\begin{tabular}{c|c| r|r|r| r| r| r|r|r|r}
\hline
\multicolumn{1}{c|}{{Dataset}}&\multicolumn{2}{c|}{{Params}}  & \multicolumn{2}{c|}{{Cutting plane via \eqref{aed1}}} & \multicolumn{2}{c|}{{\ref{aed_miscop}}} & \multicolumn{2}{c|}{\Cref{algo:cut} via \eqref{eq:novel}} & \multicolumn{2}{c}{{\Cref{algo:forward_backward}}}\\
\hline
 \multicolumn{1}{c|}{($n$, $d$)}  & \multicolumn{1}{c|}{$\lambda$}  &\multicolumn{1}{c|}{$k$} &
\multicolumn{1}{c|}{MIPGap(\%)}& 
\multicolumn{1}{c|}{time}&
\multicolumn{1}{c|}{MIPGap(\%)}& 
\multicolumn{1}{c|}{time} & 
\multicolumn{1}{c|}{MIPGap(\%)}& 
\multicolumn{1}{c|}{time} & \multicolumn{1}{c|}{Gap(\%)} & \multicolumn{1}{c}{time}\\
\hline
\multirow{4}{*}{(159, 24)} & 1 & 5 &  13.10& 3600& 8.49& 3600& 0.00 & 3548 & 0.02 & 1\\
&1 & 10&  16.14& 3600& 6.86& 3600& 1.64 & 3600 &1.80 & 1\\
  & 10 & 5& 0.00&19 & 0.27& 3600& 0.00 & 4 & 0.04 & 1 \\
&10 & 10&  0.12& 3600& 0.27& 3600& 0.06 & 3600 & 0.06 & 1 \\
\hline
\multirow{4}{*}{(194, 33)} & 1 & 5& 6.65& 3600& 5.93& 3600 &0.00 & 119 & 0.12 & 1 \\
&1 & 10&  9.64& 3600& 3.92& 3600 & 1.15 & 3600 & 1.23 & 1\\
  & 10 & 5& 0.00 & 2& 0.01& 3600 &0.00 & 1 & 0.00 & 1\\
&10 & 10&  0.00& 28& 0.11& 3600 & 0.00 & 3 & 0.00 & 1\\
\hline
\multirow{4}{*}{(200, 22)} & 1 & 5&  16.08& 3600& 0.46& 3600  &0.00 & 416 & 0.08 & 1\\
&1 & 10&  19.70& 3600& 2.27& 3600 & 2.43 & 3600 & 2.27 & 1\\
  & 10 & 5& 0.00& 47& 0.10& 3600 &  0.00 & 1 & 0.14 & 1 \\
&10 & 10&  0.24& 3600& 0.24& 3600 &  0.08 & 3600 & 0.11 & 1 \\
\hline
\multirow{4}{*}{(500, 128)} & 1 & 5&   2.90& 3600& --& -- & 0.00 & 48 & 0.02 & 3\\
&1 & 10&  1.91& 3600& --& -- &0.00 & 441 & 0.01 & 3\\
  & 10 & 5&  0.00& 30& --& -- &0.00 & 1 & 0.00 & 3\\
&10 & 10&  0.00& 26& --& -- &0.00 & 1 & 0.00 & 3\\
\hline
\multirow{4}{*}{(1000, 90)} & 1 & 5&  9.15& 3600 & --& -- & 0.00 & 192 & 0.00 & 11\\
&1 & 10&  20.05& 3600& --& -- &0.10 & 3600 & 0.10 & 11\\
  & 10 & 5& 0.00 &2181 & --& -- & 0.00 & 4 & 0.01 & 11 \\
&10 & 10& 0.37 & 3600&--& -- &0.00 & 26 & 0.00 & 11\\
\hline
\end{tabular}
\end{table}

\subsection{Application to User Cold-Start Recommendation}
This subsection presents an application of \ref{aed} to the user cold-start recommendation problem introduced in \Cref{sec:Intro}. Recall that the system  struggles with predicting preferences for new users without prior ratings. A common solution is to select a small, representative set of items for new users to rate. 
The criterion for selection usually relies on a
statistical measure aimed at minimizing the model mean squared error, most notably A-optimality and D-optimality
(see, e.g., \citealt{chamon2017approximate,rubens2009output,zhao2015theoretical}). These correspond to \ref{aed} and regularized DOD optimization problems. In the following, we compare the performance of \ref{aed}, regularized DOD, and a random selection baseline in addressing the cold-start problem. We apply the widely-used forward greedy algorithm to solve both \ref{aed} and regularized DOD for a fair comparison \citep{krause2008near}.

To understand a new user's movie preferences, we leverage the intuition that their tastes are related to how existing (training) users have rated the same movies \citep{deldjoo2019movie}.
 For each movie $i\in [n]$, the  vector $\bm a_i\in \Re^d$ contains its ratings from $d$ training users. As done in \cite{chamon2017approximate}, we assume that the new user’s ratings can be expressed as a linear combination of the past rating data, modeled as $\bm A^{\top} \bm \theta$,  where $\bm \theta \in \Re^d$ captures how similar the new user is to each training user.  Since $\bm \theta$ is unknown, we select a subset $\hat S\subseteq[n]$ of $k$ movies for the new user to rate, aiming to provide maximal information for estimating $\bm \theta$.
With the estimate $\hat{\bm \theta}$, we  predict the new user's ratings for the remaining $(n-k)$ movies as $\bm a_i^{\top}\hat{\bm \theta}$ for all $i\in [n]\setminus \hat S$. Finally, we evaluate the quality of the selected subset $\hat S$ by  computing the mean squared error (MSE) between predicted and actual ratings. A lower MSE indicates that the selected movies more effectively capture the new user's preferences.

To facilitate testing, we truncate the movie rating data from  \cite{Dooms13crowdrec} by randomly selecting a subset of ratings.
Specifically, for a fixed number of movies, training users, and new users, we repeat the truncation process  15 times with different random seeds. 
\Cref{table:movie_50,table:movie_100} compare the average MSE across all 20 and 50 new users using the three selection methods. To be specific, the average MSE is computed by averaging per-seed MSEs of all new users and then averaging these values across 15 random seeds.
We refer to the MSE values  based on our \ref{aed} method, regularized DOD, and random selection as ``\mbox{A-MSE}", ``{D-MSE}", and ``\text{R-MSE}", respectively. For each dataset, we vary the regularization parameter $\lambda\in \{0.5, 1, 1.5\}$. The random selection method is clearly unaffected by the choice of $\lambda$.
Across all experimental settings,  \ref{aed} consistently yields the lowest or comparable MSE. In contrast, while the D-optimality-based method occasionally outperforms \ref{aed}  in \Cref{table:movie_50}, it tends to generate extremely high MSE in larger-scale settings with 50 new users, as shown in \Cref{table:movie_100} where $n=700$. The random selection method performs the worst overall, particularly in \Cref{table:movie_100}. These results highlight the robustness and effectiveness of the \ref{aed}-based selection for cold-start recommendations, particularly as the number of new users grows.

\begin{table}[htbp]
\centering
\begin{minipage}{0.49\textwidth}
\caption{The average MSE of 20 new users} 
\centering
\label{table:movie_50}
\setlength{\tabcolsep}{2pt}\renewcommand{\arraystretch}{0.95}
\begin{tabular}{c| r | r| r |r }
\hline
\multicolumn{1}{c|}{$(n,d,k)$}  & \multicolumn{1}{c|}{$\lambda$} & \multicolumn{1}{c|}{A-MSE}  & \multicolumn{1}{c|}{D-MSE}  & \multicolumn{1}{c}{R-MSE} \\
\hline
\multirow{3}{*}{(500, 100, 20)} & 0.5 & \textbf{0.0346} & 0.2430 & 0.0460 \\
&1 & \textbf{0.0346} & 0.2430 & 0.0460\\
&1.5 & \textbf{0.0346} & 0.2435 & 0.0460 \\
\hline
\multirow{3}{*}{(600, 100, 20)} & 0.5&  \textbf{0.0368} & 0.0517 & 0.0463 \\
&1 & \textbf{0.0368} & 0.1864 & 0.0463 \\
&1.5 & \textbf{0.0386} & 0.1950 & 0.0463\\
\hline
\multirow{3}{*}{(700, 100, 20)} & 0.5& 0.0364 & \textbf{0.0360} & 0.0450\\
&1 & 0.0361 & \textbf{0.0357} & 0.0450\\
&1.5 & 0.0353 & \textbf{0.0351} & 0.0450\\
\hline
\multirow{3}{*}{(800, 100, 20)} & 0.5& \textbf{0.0366} & 0.0382 & 0.0466 \\
&1 & \textbf{0.0367} & 0.0383 & 0.0466\\
&1.5 & \textbf{0.0378} & 0.0454 & 0.0466\\
\hline
\multirow{3}{*}{(900, 100, 20)} & 0.5& \textbf{0.0383}	&0.0385	&0.0483 \\
&1 & 0.0390	& \textbf{0.0384}	&0.0483\\
&1.5 & 0.0391 & \textbf{0.0382} &0.0483\\
\hline
\end{tabular}
\end{minipage}%
\hfill
\begin{minipage}{0.5\textwidth}
 \caption{The average MSE of 50 new users} 
\centering
\label{table:movie_100}
\setlength{\tabcolsep}{2pt}\renewcommand{\arraystretch}{0.95}
\begin{tabular}{c| r | r| r |r }
\hline
\multicolumn{1}{c|}{$(n,d,k)$}  & \multicolumn{1}{c|}{$\lambda$} & \multicolumn{1}{c|}{A-MSE}  & \multicolumn{1}{c|}{D-MSE}  & \multicolumn{1}{c}{R-MSE} \\
\hline
\multirow{3}{*}{(500, 150, 50)} & 0.5 & \textbf{0.0323} & 0.0355 & 1.65e+26 \\
&1 & \textbf{0.0343} & 0.0355 & 1.65e+26\\
&1.5 & \textbf{0.0343} & 0.0381 & 1.65e+26\\
\hline
\multirow{3}{*}{(600, 150, 50)} & 0.5& \textbf{0.0388} & 0.0578 & 2.35e+26\\
&1 & \textbf{0.0398} & 0.0737 & 2.35e+26\\
&1.5 & \textbf{0.0392} & 0.0752 & 2.35e+26\\
\hline
\multirow{3}{*}{(700, 150, 50)} & 0.5& \textbf{0.0354} & 1.72e+21 & 2.00e+26 \\
&1 & \textbf{0.0354} & 1.72e+21 & 2.00e+26\\
&1.5 & \textbf{0.0356} & 1.72e+21 & 2.00e+26\\
\hline
\multirow{3}{*}{(800, 150, 50)} & 0.5& \textbf{0.0366} & 0.0993 & 3.54e+26 \\
&1 & \textbf{0.0367} & 0.1179 & 3.54e+26 \\
&1.5 & \textbf{0.0367} & 0.1182 & 3.54e+26\\
\hline
\multirow{3}{*}{(900, 150, 50)} & 0.5& 0.0361	& \textbf{0.0355}	&1.23e+26 \\
&1 & \textbf{0.0361}	&0.0425	&1.23e+26 \\
&1.5 & \textbf{0.0361} &0.0480 &1.23e+26\\
\hline
\end{tabular}
\end{minipage}
\end{table}

\section{Conclusion}
We study the regularized A-optimal design (RAOD) problem  and prove its NP-hardness for the first time.  A key insight from our results is that the performance of both formulations and algorithms for RAOD hinges on whether the selection size $k$ exceeds the data dimension $d$. Specifically, we demonstrate that the two existing relaxations perform well in only one of the two ranges ($k \le d-1$ or $k \ge d$), but poorly in the other. To address this, we propose a novel convex integer formulation that yields a stronger relaxation with provable performance guarantees for all  $k$. This formulation significantly accelerates the exact cutting-plane algorithm in small-$k$ and high-dimensional settings. We further investigate the complementary forward and backward greedy algorithms, tailored respectively to the ranges $k\le d-1$ and $k\ge d$.  Our numerical experiments demonstrate the effectiveness and efficiency of the proposed algorithms.
Our work can be seen as a non-sequential experimental design scheme, where the data points are selected at once.  
A possible future direction is to explore adaptive strategies that select points sequentially based on prior selections.

\newpage

  \bibliographystyle{informs2014}
\bibliography{reference.bib}

\begin{thebibliography}{60}
\providecommand{\natexlab}[1]{#1}
\providecommand{\url}[1]{\texttt{#1}}
\providecommand{\urlprefix}{URL }

\bibitem[{Ahipa\c{s}ao\u{g}lu(2021)}]{ahipacsaouglu2021branch}
Ahipa\c{s}ao\u{g}lu SD (2021) A branch-and-bound algorithm for the exact
  optimal experimental design problem. \emph{Statistics and Computing}
  31(5):65.

\bibitem[{Ahipa{\c{s}}ao{\u{g}}lu(2015)}]{ahipacsaouglu2015first}
Ahipa{\c{s}}ao{\u{g}}lu SD (2015) A first-order algorithm for the a-optimal
  experimental design problem: a mathematical programming approach.
  \emph{Statistics and Computing} 25:1113--1127.

\bibitem[{Allen-Zhu et~al.(2017)Allen-Zhu, Li, Singh, \protect\BIBand{}
  Wang}]{allen2017near}
Allen-Zhu Z, Li Y, Singh A, Wang Y (2017) Near-optimal design of experiments
  via regret minimization. \emph{International Conference on Machine Learning},
  126--135 (PMLR).

\bibitem[{Anava et~al.(2015)Anava, Golan, Golbandi, Karnin, Lempel, Rokhlenko,
  \protect\BIBand{} Somekh}]{anava2015budget}
Anava O, Golan S, Golbandi N, Karnin Z, Lempel R, Rokhlenko O, Somekh O (2015)
  Budget-constrained item cold-start handling in collaborative filtering
  recommenders via optimal design. \emph{Proceedings of the 24th international
  conference on world wide web}, 45--54.

\bibitem[{Avron \protect\BIBand{} Boutsidis(2013)}]{avron2013faster}
Avron H, Boutsidis C (2013) Faster subset selection for matrices and
  applications. \emph{SIAM Journal on Matrix Analysis and Applications}
  34(4):1464--1499.

\bibitem[{Barz et~al.(2015)Barz, K{\"o}rkel, Wozny et~al.}]{barz2015nonlinear}
Barz T, K{\"o}rkel S, Wozny G, et~al. (2015) Nonlinear ill-posed problem
  analysis in model-based parameter estimation and experimental design.
  \emph{Computers \& Chemical Engineering} 77:24--42.

\bibitem[{Bertsimas et~al.(2020)Bertsimas, Pauphilet, \protect\BIBand{}
  Van~Parys}]{bertsimas2020rejoinder}
Bertsimas D, Pauphilet J, Van~Parys B (2020) Rejoinder: Sparse regression:
  Scalable algorithms and empirical performance. \emph{Statistical Science}
  35(4):623--624.

\bibitem[{Bian et~al.(2017)Bian, Buhmann, Krause, \protect\BIBand{}
  Tschiatschek}]{bian2017guarantees}
Bian AA, Buhmann JM, Krause A, Tschiatschek S (2017) Guarantees for greedy
  maximization of non-submodular functions with applications.
  \emph{International conference on machine learning}, 498--507 (PMLR).

\bibitem[{Blake(1998)}]{blake1998uci}
Blake CL (1998) Uci repository of machine learning databases. \emph{http://www.
  ics. uci. edu/\~{} mlearn/MLRepository. html} .

\bibitem[{Bobadilla et~al.(2012)Bobadilla, Ortega, Hernando, \protect\BIBand{}
  Bernal}]{bobadilla2012collaborative}
Bobadilla J, Ortega F, Hernando A, Bernal J (2012) A collaborative filtering
  approach to mitigate the new user cold start problem. \emph{Knowledge-based
  systems} 26:225--238.

\bibitem[{Chaloner \protect\BIBand{} Verdinelli(1995)}]{chaloner1995bayesian}
Chaloner K, Verdinelli I (1995) Bayesian experimental design: A review.
  \emph{Statistical science} 273--304.

\bibitem[{Chamon \protect\BIBand{}
  Ribeiro(2017{\natexlab{a}})}]{chamon2017approximate}
Chamon L, Ribeiro A (2017{\natexlab{a}}) Approximate supermodularity bounds for
  experimental design. \emph{Advances in Neural Information Processing Systems}
  30.

\bibitem[{Chamon \protect\BIBand{}
  Ribeiro(2017{\natexlab{b}})}]{chamon2017greedy}
Chamon LF, Ribeiro A (2017{\natexlab{b}}) Greedy sampling of graph signals.
  \emph{IEEE Transactions on Signal Processing} 66(1):34--47.

\bibitem[{Constantine(1983)}]{constantine1983schur}
Constantine GM (1983) Schur convex functions on the spectra of graphs.
  \emph{Discrete Mathematics} 45(2-3):181--188.

\bibitem[{Cormen et~al.(2022)Cormen, Leiserson, Rivest, \protect\BIBand{}
  Stein}]{cormen2022introduction}
Cormen TH, Leiserson CE, Rivest RL, Stein C (2022) \emph{Introduction to
  algorithms} (MIT press).

\bibitem[{Deldjoo et~al.(2019)Deldjoo, Dacrema, Constantin, Eghbal-Zadeh,
  Cereda, Schedl, Ionescu, \protect\BIBand{} Cremonesi}]{deldjoo2019movie}
Deldjoo Y, Dacrema MF, Constantin MG, Eghbal-Zadeh H, Cereda S, Schedl M,
  Ionescu B, Cremonesi P (2019) Movie genome: alleviating new item cold start
  in movie recommendation. \emph{User Modeling and User-Adapted Interaction}
  29:291--343.

\bibitem[{Derezinski et~al.(2020)Derezinski, Liang, \protect\BIBand{}
  Mahoney}]{derezinski2020bayesian}
Derezinski M, Liang F, Mahoney M (2020) Bayesian experimental design using
  regularized determinantal point processes. \emph{International Conference on
  Artificial Intelligence and Statistics}, 3197--3207 (PMLR).

\bibitem[{Derezi{\'n}ski \protect\BIBand{}
  Warmuth(2017)}]{derezinski2017subsampling}
Derezi{\'n}ski M, Warmuth MK (2017) Subsampling for ridge regression via
  regularized volume sampling. \emph{arXiv preprint arXiv:1710.05110} .

\bibitem[{Dooms et~al.(2013)Dooms, De~Pessemier, \protect\BIBand{}
  Martens}]{Dooms13crowdrec}
Dooms S, De~Pessemier T, Martens L (2013) Movietweetings: a movie rating
  dataset collected from twitter. \emph{Workshop on Crowdsourcing and Human
  Computation for Recommender Systems, CrowdRec at RecSys 2013}.

\bibitem[{Drusvyatskiy \protect\BIBand{}
  Kempton(2015)}]{drusvyatskiy2015variational}
Drusvyatskiy D, Kempton C (2015) Variational analysis of spectral functions
  simplified. \emph{arXiv preprint arXiv:1506.05170} .

\bibitem[{Duarte(2023)}]{duarte2023exact}
Duarte BP (2023) Exact optimal designs of experiments for factorial models via
  mixed-integer semidefinite programming. \emph{Mathematics} 11(4):854.

\bibitem[{Duarte \protect\BIBand{} Wong(2015)}]{duarte2015finding}
Duarte BP, Wong WK (2015) Finding bayesian optimal designs for nonlinear
  models: a semidefinite programming-based approach. \emph{International
  Statistical Review} 83(2):239--262.

\bibitem[{Duran \protect\BIBand{} Grossmann(1986)}]{duran1986outer}
Duran MA, Grossmann IE (1986) An outer-approximation algorithm for a class of
  mixed-integer nonlinear programs. \emph{Mathematical programming}
  36:307--339.

\bibitem[{Elahi et~al.(2016)Elahi, Ricci, \protect\BIBand{}
  Rubens}]{elahi2016survey}
Elahi M, Ricci F, Rubens N (2016) A survey of active learning in collaborative
  filtering recommender systems. \emph{Computer Science Review} 20:29--50.

\bibitem[{Gope \protect\BIBand{} Jain(2017)}]{gope2017survey}
Gope J, Jain SK (2017) A survey on solving cold start problem in recommender
  systems. \emph{2017 International Conference on Computing, Communication and
  Automation (ICCCA)}, 133--138 (IEEE).

\bibitem[{Hendrych et~al.(2023)Hendrych, Besan{\c{c}}on, \protect\BIBand{}
  Pokutta}]{hendrych2023solving}
Hendrych D, Besan{\c{c}}on M, Pokutta S (2023) Solving the optimal experiment
  design problem with mixed-integer convex methods. \emph{arXiv preprint
  arXiv:2312.11200} .

\bibitem[{Higgs et~al.(1997)Higgs, Bemis, Watson, \protect\BIBand{}
  Wikel}]{higgs1997experimental}
Higgs RE, Bemis KG, Watson IA, Wikel JH (1997) Experimental designs for
  selecting molecules from large chemical databases. \emph{Journal of chemical
  information and computer sciences} 37(5):861--870.

\bibitem[{Ilzarbe et~al.(2008)Ilzarbe, {\'A}lvarez, Viles, \protect\BIBand{}
  Tanco}]{ilzarbe2008practical}
Ilzarbe L, {\'A}lvarez MJ, Viles E, Tanco M (2008) Practical applications of
  design of experiments in the field of engineering: a bibliographical review.
  \emph{Quality and Reliability Engineering International} 24(4):417--428.

\bibitem[{Jobson(2012)}]{jobson2012applied}
Jobson JD (2012) \emph{Applied multivariate data analysis: regression and
  experimental design} (Springer Science \& Business Media).

\bibitem[{Jones et~al.(2021)Jones, Allen-Moyer, \protect\BIBand{}
  Goos}]{jones2021optimal}
Jones B, Allen-Moyer K, Goos P (2021) A-optimal versus d-optimal design of
  screening experiments. \emph{Journal of Quality Technology} 53(4):369--382.

\bibitem[{Kim et~al.(2022)Kim, Tawarmalani, \protect\BIBand{}
  Richard}]{kim2022convexification}
Kim J, Tawarmalani M, Richard JPP (2022) Convexification of
  permutation-invariant sets and an application to sparse principal component
  analysis. \emph{Mathematics of Operations Research} 47(4):2547--2584.

\bibitem[{Krause et~al.(2008)Krause, Singh, \protect\BIBand{}
  Guestrin}]{krause2008near}
Krause A, Singh A, Guestrin C (2008) Near-optimal sensor placements in gaussian
  processes: Theory, efficient algorithms and empirical studies. \emph{Journal
  of Machine Learning Research} 9(2).

\bibitem[{Leardi(2009)}]{leardi2009experimental}
Leardi R (2009) Experimental design in chemistry: A tutorial. \emph{Analytica
  chimica acta} 652(1-2):161--172.

\bibitem[{Lewis(1995)}]{lewis1995convex}
Lewis AS (1995) The convex analysis of unitarily invariant matrix functions.
  \emph{Journal of Convex Analysis} 2(1):173--183.

\bibitem[{Li(2024)}]{li2024augmented}
Li Y (2024) The augmented factorization bound for maximum-entropy sampling.
  \emph{arXiv preprint arXiv:2410.10078} .

\bibitem[{Li et~al.(2024)Li, Fampa, Lee, Qiu, Xie, \protect\BIBand{}
  Yao}]{li2024d}
Li Y, Fampa M, Lee J, Qiu F, Xie W, Yao R (2024) D-optimal data fusion: Exact
  and approximation algorithms. \emph{INFORMS Journal on Computing}
  36(1):97--120.

\bibitem[{Li \protect\BIBand{} Xie(2024)}]{li2024best}
Li Y, Xie W (2024) Best principal submatrix selection for the maximum entropy
  sampling problem: scalable algorithms and performance guarantees.
  \emph{Operations Research} 72(2):493--513.

\bibitem[{Liang \protect\BIBand{} Yang(2024)}]{liang2024pnod}
Liang L, Yang H (2024) Pnod: An efficient projected newton framework for exact
  optimal experimental designs. \emph{arXiv preprint arXiv:2409.18392} .

\bibitem[{Lika et~al.(2014)Lika, Kolomvatsos, \protect\BIBand{}
  Hadjiefthymiades}]{lika2014facing}
Lika B, Kolomvatsos K, Hadjiefthymiades S (2014) Facing the cold start problem
  in recommender systems. \emph{Expert systems with applications}
  41(4):2065--2073.

\bibitem[{Madan et~al.(2019)Madan, Singh, Tantipongpipat, \protect\BIBand{}
  Xie}]{madan2019combinatorial}
Madan V, Singh M, Tantipongpipat U, Xie W (2019) Combinatorial algorithms for
  optimal design. \emph{Conference on Learning Theory}, 2210--2258 (PMLR).

\bibitem[{Marshall(1979)}]{marshall1979inequalities}
Marshall A (1979) Inequalities: Theory of majorization and its applications.

\bibitem[{Mason et~al.(2003)Mason, Gunst, \protect\BIBand{}
  Hess}]{mason2003statistical}
Mason RL, Gunst RF, Hess JL (2003) \emph{Statistical design and analysis of
  experiments: with applications to engineering and science} (John Wiley \&
  Sons).

\bibitem[{Nikolov(2015)}]{nikolov2015randomized}
Nikolov A (2015) Randomized rounding for the largest simplex problem.
  \emph{Proceedings of the forty-seventh annual ACM symposium on Theory of
  computing}, 861--870.

\bibitem[{Nikolov et~al.(2022)Nikolov, Singh, \protect\BIBand{}
  Tantipongpipat}]{nikolov2022proportional}
Nikolov A, Singh M, Tantipongpipat U (2022) Proportional volume sampling and
  approximation algorithms for a-optimal design. \emph{Mathematics of
  Operations Research} 47(2):847--877.

\bibitem[{Pukelsheim(2006)}]{pukelsheim2006optimal}
Pukelsheim F (2006) \emph{Optimal design of experiments} (SIAM).

\bibitem[{Quesada \protect\BIBand{} Grossmann(1992)}]{quesada1992lp}
Quesada I, Grossmann IE (1992) An lp/nlp based branch and bound algorithm for
  convex minlp optimization problems. \emph{Computers \& chemical engineering}
  16(10-11):937--947.

\bibitem[{Quinn \protect\BIBand{} Keough(2002)}]{quinn2002experimental}
Quinn GP, Keough MJ (2002) \emph{Experimental design and data analysis for
  biologists} (Cambridge university press).

\bibitem[{Rainforth et~al.(2024)Rainforth, Foster, Ivanova, \protect\BIBand{}
  Bickford~Smith}]{rainforth2024modern}
Rainforth T, Foster A, Ivanova DR, Bickford~Smith F (2024) Modern bayesian
  experimental design. \emph{Statistical Science} 39(1):100--114.

\bibitem[{Rockafellar(1997)}]{rockafellar1997convex}
Rockafellar RT (1997) \emph{Convex analysis}, volume~28 (Princeton university
  press).

\bibitem[{Rubens et~al.(2009)Rubens, Tomioka, \protect\BIBand{}
  Sugiyama}]{rubens2009output}
Rubens N, Tomioka R, Sugiyama M (2009) Output divergence criterion for active
  learning in collaborative settings. \emph{IPSJ Online Transactions}
  2:240--249.

\bibitem[{Sagnol(2011)}]{sagnol2011computing}
Sagnol G (2011) Computing optimal designs of multiresponse experiments reduces
  to second-order cone programming. \emph{Journal of Statistical Planning and
  Inference} 141(5):1684--1708.

\bibitem[{Sagnol \protect\BIBand{} Harman(2015)}]{sagnol2015computing}
Sagnol G, Harman R (2015) Computing exact d-optimal designs by mixed integer
  second-order cone programming. \emph{The Annals of Statistics}
  43(5):2198--2224.

\bibitem[{Sedrakyan \protect\BIBand{} Sedrakyan(2018)}]{sedrakyan2018algebraic}
Sedrakyan H, Sedrakyan N (2018) \emph{Algebraic inequalities} (Springer).

\bibitem[{Stallrich et~al.(2023)Stallrich, Allen-Moyer, \protect\BIBand{}
  Jones}]{stallrich2023d}
Stallrich J, Allen-Moyer K, Jones B (2023) D-and a-optimal screening designs.
  \emph{Technometrics} 65(4):492--501.

\bibitem[{Tantipongpipat(2020)}]{tantipongpipat2020lambda}
Tantipongpipat U (2020) $\lambda$-regularized a-optimal design and its
  approximation by $\lambda$-regularized proportional volume sampling.
  \emph{arXiv preprint arXiv:2006.11182} .

\bibitem[{Volkovs et~al.(2017)Volkovs, Yu, \protect\BIBand{}
  Poutanen}]{volkovs2017content}
Volkovs M, Yu GW, Poutanen T (2017) Content-based neighbor models for cold
  start in recommender systems. \emph{Proceedings of the Recommender Systems
  Challenge 2017}, 1--6.

\bibitem[{Wang et~al.(2017)Wang, Yu, \protect\BIBand{}
  Singh}]{wang2017computationally}
Wang Y, Yu AW, Singh A (2017) On computationally tractable selection of
  experiments in measurement-constrained regression models. \emph{Journal of
  Machine Learning Research} 18(143):1--41.

\bibitem[{Welch(1982)}]{welch1982algorithmic}
Welch WJ (1982) Algorithmic complexity: three np-hard problems in computational
  statistics. \emph{Journal of Statistical Computation and Simulation}
  15(1):17--25.

\bibitem[{Winer et~al.(1971)Winer, Brown, Michels
  et~al.}]{winer1971statistical}
Winer BJ, Brown DR, Michels KM, et~al. (1971) \emph{Statistical principles in
  experimental design}, volume~2 (Mcgraw-hill New York).

\bibitem[{Zhao \protect\BIBand{} Wang(2015)}]{zhao2015theoretical}
Zhao X, Wang J (2015) A theoretical analysis of two-stage recommendation for
  cold-start collaborative filtering. \emph{Proceedings of the 2015
  International Conference on The Theory of Information Retrieval}, 71--80.

\end{thebibliography}
\newpage
\begin{APPENDICES}
\newpage
\renewcommand{\theHsection}{A\arabic{section}}

\section{Additional proof}
\subsection{Proof of \Cref{prop:amesp}}
\label{proof:amesp}
\begin{proof}{Proof.}
We prove the result by showing that the objective functions of \ref{aed} and  \ref{amesp} are equivalent.
    For any  solution $S$ of \ref{aed}, let $\sigma_1\ge \cdots\ge \sigma_d\ge 0$  and $\beta_1\ge \cdots\ge \beta_k \ge 0$ denote the eigenvalues of $\bm A_S \bm A_S^{\top}\in \S_+^d$ and $\bm A_S^{\top} \bm A_S\in \S_+^k$, respectively. Then, our analysis is split into two parts depending on whether $k\le d$ holds or not.
    \begin{enumerate}[(i)]
        \item $k\in [d]$. Because $\bm A_S \bm A_S^{\top}\in \S_+^d$ and $\bm A_S^{\top} \bm A_S\in \S_+^k$ admit the same nonzero eigenvalues,  their eigenvalues satisfy $\sigma_i=\beta_i$ for any $i\in [k]$ and $\sigma_i=0$ for all $i\in [k+1,d]$. Given this, we have that
    \begin{align*}
  f(S) &=  \tr \left( \left(\bm A_S \bm A_S^{\top} + \lambda \bm I_d\right)^{-1} \right) = \sum_{i\in [d]} \frac{1}{\sigma_i+\lambda} = \sum_{i\in [k]} \frac{1}{\beta_i+\lambda} + \frac{d-k}{\lambda}\\
  &= \tr \left( \left( \bm A_S^{\top} \bm A_S + \lambda \bm I_k\right)^{-1} \right) + \frac{d-k}{\lambda} =\tr \left( \left( \bm C_{S,S}\right)^{-1} \right) + \frac{d-k}{\lambda},
    \end{align*}
     where the first and the last equations follow from the definitions of $\bm A$ and $\bm C$, respectively.

     \item $k\in [d+1,n]$. Analogously, according to the relation of the eigenvalues between $\bm A_S \bm A_S^{\top}\in \S_+^d$ and $\bm A_S^{\top} \bm A_S\in \S_+^k$,  we get $\sigma_i=\beta_i$ for any $i\in [d]$ and $\beta_i=0$ for all $i\in [d+1,k]$. Then, we have that
    \begin{align*}
  f(S) &=  \tr \left( \left(\bm A_S \bm A_S^{\top} + \lambda \bm I_d\right)^{-1} \right) = \sum_{i\in [d]} \frac{1}{\sigma_i+\lambda} = \sum_{i\in [k]} \frac{1}{\beta_i+\lambda} - \frac{k-d}{\lambda}\\
  &= \tr \left( \left( \bm A_S^{\top} \bm A_S + \lambda \bm I_k\right)^{-1} \right) + \frac{d-k}{\lambda} =\tr \left( \left( \bm C_{S,S}\right)^{-1} \right) + \frac{d-k}{\lambda}.
    \end{align*}
We thus complete the proof. \qed
    \end{enumerate}
\end{proof}

\subsection{Proof of \Cref{them:hard}}\label{proof:hard}
\begin{proof}{Proof.}
In order to prove the result, we  reduce the NP-hard independent set decision problem to  \ref{amesp}. Given a simple undirected graph $G([n], E)$, we define a symmetric matrix $\bm C$ as
\[ C_{ij} =  \begin{cases}
   n & \text{if } i=j;\\
    1 & \text{if } (i,j)\in E;\\
    0 & \text{Otherwise}.
\end{cases} \]
For all $i\in [n]$, it is easy to verify that $C_{ii}=n>n-1\ge \sum_{j\neq i}|C_{ij}|$, implying that
 $\bm C$ is a strictly diagonally dominant matrix. Hence, the matrix $\bm C$ is positive definite, and so is any principal submatrix of $\bm C$. Besides, we can express $\bm C=\bm A\bm A^{\top} +\lambda \bm I_n$ in the form of the objective matrix in \ref{amesp}, where $\lambda$ equals the smallest eigenvalue of $\bm C$, and $\bm A$ denotes the Cholesky factor of $\bm C-\lambda\bm I_n$. Therefore, the resulting pair $(\bm A, \lambda)$ can be used as an input instance of \ref{aed} such that \ref{aed} is equivalent to \ref{amesp}. 
 
 For any subset $S\subseteq [n]$, the subgragh $G_S$ comprises the vertices $S$ and edges in the set $E\cap (S\times S)$, and then, the subgragh $G_S$ can be represented by the principal submatrix $\bm C_{S,S}$. 
 
 Next, for a subset $S$ of size $k$, we compute the trace of the inverse of $\bm C_{S,S}$ in two cases.
\begin{enumerate}
    \item[(i)] If $S$ is a sized-$k$ independent set for the graph $G$, according to the property of independent sets, the resulting subgraph $G_S$ has no edges, that is, $(i,j)\notin E$ for any two vertices $i,j$ in $S$. Therefore, the corresponding submatrix $\bm C_{S,S}$ to $G_S$ must be diagonal, and we can easily compute $\tr((\bm C_{S,S})^{-1}) =\sum_{i\in S} 1/C_{ii} = k/n$.
\item[(ii)] If $S$ is not a sized-$k$ independent set, then, unlike Part (i), there is at least one nonzero off-diagonal entry in $\bm C_{S,S}$. Suppose that the vector $\bm \sigma\in \Re_{++}^k$ contains the eigenvalues of $\bm C_{S,S}$. It follows that 
$$\|\bm C_{S,S}\|_F^2 = \sum_{i\in [k]} \sigma_i^2=\sum_{i\in S}\sum_{j\in S}C_{ij}^2 \ge kn^2+2.$$ In addition, we have that $\sum_{i\in [k]}\sigma_i=\tr(\bm C_{S,S})=kn$. As a result,
not all entries of $\bm \sigma$  can be equal. Otherwise, we would find that $\sum_{i\in [k]}\sigma_i^2=kn^2$, a contradiction with the inequality above.

According to \Cref{ineq:amhm}, we get
\[ \frac{k}{\tr((\bm C_{S,S})^{-1}) } = \frac{k}{\sum_{i\in [k]} \frac{1}{\sigma_i} } < \frac{\sum_{i\in [k]} \sigma_i}{k}   = n \ \  \Longrightarrow \ \ \tr((\bm C_{S,S})^{-1}) > \frac{k}{n}, \]
where the first inequality is strict because the entries of $\bm \sigma$ are not all equal.
\end{enumerate}

By combining the results in Parts (i) and (ii), we conclude that  if  \ref{amesp} yields an optimal value of $k/n+(d-k)/\lambda$, then there exists an independent set of  size $k$ in the graph $G$. On the other hand, if the optimal value is larger than $k/n+(d-k)/\lambda$,  the independent set decision problem is false. We thus complete the proof.
 \qed
\end{proof}

\subsection{Proof of \Cref{lem:techniques}}\label{proof:lem_tech}

\begin{proof}{Proof.}
For each $i\in S$, let us consider the following probability to sample $i$ out of $S$, proposed by \cite{derezinski2017subsampling}:
\[\P(i|S)=\frac{1-\bm a_i^{\top}\left(\bm A_{S} \bm A_{S}^{\top} + \lambda\bm I_d\right)^{-1}\bm a_i}{s-\tr\left(\bm A_S^{\top}\left(\bm A_{S} \bm A_{S}^{\top} + \lambda\bm I_d\right)^{-1} \bm A_S\right)}.\]
It is easy to verify that $\sum_{i\in S}\P(i|S)=1$. Then, the expected function value of removing an element from $S$ is equal to
\begin{equation}\label{eq:expect}
\begin{aligned}
\E\left[f(S\setminus\{i\})\right]& = \sum_{i\in S} \P(i|S)f(S\setminus\{i\})=   \sum_{i\in S} \P(i|S) \tr \left(\left(\bm A_{S} \bm A_{S}^{\top} -\bm a_i\bm a_i^{\top} + \lambda \bm I_d \right)^{-1}\right)\\
  &=\sum_{i\in S} \P(i|S) \left( f(S) + \frac{ \bm a_i^{\top}   \left(\bm A_{S} \bm A_{S}^{\top} + \lambda \bm I_d\right)^{-2} \bm a_i }{1-\bm a_i^{\top} \left(\bm A_{S} \bm A_{S}^{\top}+\lambda \bm I_d\right)^{-1} \bm a_i}\right)  \\
  &= f(S)  + \frac{\tr\left( \bm A_S^{\top}   \left(\bm A_{S} \bm A_{S}^{\top} + \lambda \bm I_d\right)^{-2} \bm A_S\right)}{s- \tr\left(\bm A_S^{\top}\left(\bm A_{S} \bm A_{S}^{\top}+\lambda\bm I_d\right)^{-1} \bm A_S\right)} \\
  &=f(S) \left( 1 + \frac{\tr\left(   \left(\bm A_{S} \bm A_{S}^{\top} + \lambda \bm I_d \right)^{-2} \bm A_S \bm A_S^{\top}\right)}{\tr \left( \left(\bm A_{S} \bm A_{S}^{\top} + \lambda \bm I_d \right)^{-1}  \right) \left(s- \tr\left(\left(\bm A_{S} \bm A_{S}^{\top}+\lambda\bm I_d\right)^{-1} \bm A_S \bm A_S^{\top}\right)\right)} \right),
\end{aligned}
\end{equation}
where the second equation is based on the Sherman–Morrison formula,  the third equation is from the expression of $\P(i|S)$, and the last one results from the cyclic property of  trace operator. 

Observe that the matrix $\bm A_{S} \bm A_{S}^{\top}\in \S_+^d$ has a rank of at most $\min\{s,d\}$; hence, it has at least $d-\min\{s,d\}$ zero eigenvalues.
Let  $\bm A_{S} \bm A_{S}^{\top} = \bm Q \bm \Lambda \bm Q^{\top}$ be the eigen-decomposition of $\bm A_{S} \bm A_{S}^{\top}$, where $\bm \Lambda\in \S_{+}^{d}$ is a diagonal matrix comprising $\min\{s,d\}$ nonnegative eigenvalues $\sigma_1, \cdots, \sigma_{\min\{s,d\}}$ and zero eigenvalues $\sigma_{\min\{s,d\}+1}=\cdots=\sigma_d=0$ of $\bm A_{S} \bm A_{S}^{\top}$. Given the eigen-decomposition of $\bm A_S\bm A_S^{\top}$, we can show that
\begin{equation}\label{eq:matrices}
\begin{aligned}
 &M_1:=\tr\left( \left(\bm A_{S} \bm A_{S}^{\top} + \lambda \bm I_d\right)^{-2} \bm A_S \bm A_S^{\top}\right)= \tr\left( \left(\bm \Lambda + \lambda \bm I_d\right)^{-2} \bm \Lambda \right)= \sum_{i\in[\min\{s,d\} ]} \frac{\sigma_i}{(\lambda+\sigma_i)^2},   \\ 
 & M_2:=\tr \left( \left(\bm A_{S} \bm A_{S}^{\top} + \lambda \bm I_d \right)^{-1}  \right) =\tr\left(\left(\bm \Lambda+\lambda\bm I_d\right)^{-1} \right)=\sum_{i\in [\min\{s,d\} ]} \frac{1}{\lambda+\sigma_i} + (d-\min\{s,d\})\frac{1}{\lambda}, \text{ and}\\
 &M_3:= \tr\left(\left(\bm A_{S} \bm A_{S}^{\top}+\lambda\bm I_d \right)^{-1} \bm A_S \bm A_S^{\top}\right) =\tr\left( \left(\bm \Lambda + \lambda \bm I_d\right)^{-1} \bm \Lambda \right) = \sum_{i\in [\min\{s,d\}]} \frac{\sigma_i}{\lambda+\sigma_i}.
\end{aligned}
\end{equation}

Replacing the expressions on the right-hand side of \eqref{eq:expect} with $M_1, M_2, M_3$ leads to
\begin{align}\label{eq:bound}
    \min_{i\in S} f(S\setminus\{i\}) \le \E\left[f(S\setminus\{i\})\right] = f(S)\left(1+\frac{M_1}{M_2\left(s- M_3\right)} \right). 
\end{align}
To prove the result, we will bound the value of  $\frac{M_1}{M_2\left(s- M_3\right)}$ from above in two cases. 

\noindent\textbf{Case I:} $s\le d-1$. 
To begin, let us establish a lower bound: 
\begin{align}\label{ineq:lb}
\sum_{i\in [s]} \frac{\lambda}{\lambda+\sigma_i} \ge \frac{s^2\lambda}{\sum_{i\in [s]} (\sigma_i+ \lambda)} = \frac{s^2\lambda}{\sum_{i\in S}\|\bm a_i\|_2^2+s\lambda}\ge \frac{s^2\lambda}{s\max_{i\in [n]}\|\bm a_i\|_2^2+ s\lambda}=s\rho, 
\end{align}
where the first inequality is from \Cref{ineq:amhm}, and the second equality is  from the fact that $\sigma_{s+1}=\cdots=\sigma_d=0$. Next,
based on \eqref{eq:matrices} and \eqref{ineq:lb}, we have that
\begin{equation}\label{eq:sless}
    \begin{aligned}
\frac{M_1}{M_2\left(s- M_3\right)} = 
 \frac{\sum_{i\in [s]} \frac{\sigma_i}{(\lambda+\sigma_i)^2}}{\left(\sum_{i\in [s]} \frac{1}{\lambda+\sigma_i} + (d-s)\frac{1}{\lambda}\right) \left(\lambda\sum_{i\in [s]} \frac{1}{\lambda+\sigma_i}\right)} \le  \frac{1}{\sum_{i\in [s]} \frac{\lambda}{\lambda+\sigma_i} + d-s} \le \frac{1}{s\rho+d-s} ,  
\end{aligned}
    \end{equation}
where the first equation is obtained from plugging the expressions of $M_1,M_2$, and $M_3$ in \eqref{eq:matrices}, the second inequality is because $\frac{\sigma_i}{(\lambda+\sigma_i)^2}\le\frac{1}{\lambda+\sigma_i}$ for each $i\in [s]$, and the last one is from \eqref{ineq:lb}.
  
\noindent\textbf{Case II:} $s\ge d$. Analogous to \eqref{ineq:lb} and \eqref{eq:sless}, we can show that
\begin{align*}
\sum_{i\in [d]} \frac{\lambda}{\lambda+\sigma_i} \ge \frac{d^2\lambda}{\sum_{i\in [d]} (\sigma_i+ \lambda)} = \frac{d^2\lambda}{\sum_{i\in S}\|\bm a_i\|_2^2+d\lambda}\ge \frac{d^2\lambda}{s\max_{i\in [n]}\|\bm a_i\|_2^2+ d\lambda}=d\rho, 
\end{align*} 
and
\begin{align}\label{eq:greater}
\frac{M_1}{M_2\left(s- M_3\right)} = 
\frac{\sum_{i\in [d]} \frac{\sigma_i}{(\lambda+\sigma_i)^2}}{\left(\sum_{i\in [d]} \frac{1}{\lambda+\sigma_i}\right) \left(s-d+d-\sum_{i\in [d]} \frac{\sigma_i}{\lambda+\sigma_i}\right)} \le \frac{1}{s-d+\sum_{i\in [d]} \frac{\lambda}{\lambda+\sigma_i}} \le \frac{1}{s-d+ d\rho}.
\end{align}

By plugging the results of  \eqref{eq:sless} and \eqref{eq:greater} into \eqref{eq:bound}, we  complete the proof.
\qed
\end{proof}

\subsection{Proof of \Cref{cor:gap}}
\label{proof:gap}
\begin{proof}{Proof.}
{Part (i).} Let $S$ be an optimal solution of \ref{aed} with the selection size $s\in [d,n]$. Then, $|S|=s$ and $f(S)=z_s$. Based on \Cref{lem:techniques}, we have 
\begin{align}\label{eq:seq}
    z_{s-1} \le \min_{i\in S} f(S\setminus\{i\}) \le
\left(1+\frac{1}{d\rho+s-d}\right) f(S)=\frac{d\rho+s-d+1}{d\rho+s-d} z_{s},
\end{align}
 where the first inequality is because $S\setminus\{i\}$ is a feasible solution to \ref{aed} at $s-1$.

It is clear that ${\nu}_k^{\text{C}}$ must be at least as large as the objective value of selecting all data points, that is, $\nu_k^{\text{C}} \ge z_n$. Then, sequentially applying \eqref{eq:seq} with $s=k+1,\cdots,n$ leads to
    \begin{align*}
 z_{k} \le \prod_{s\in [k+1,n]} \frac{s-d+1+d\rho}{s-d+d\rho} z_n = \frac{n-d+1+d\rho}{k+1-d+d\rho} z_n \le \frac{n-d+1+d\rho}{k-d+1+d\rho} \nu_k^{\text{C}} \le \frac{n-d+1}{k-d+1} \nu_k^{\text{C}},
\end{align*} 
where the last inequality follows immediately because of $\rho>0$. 

Part (ii). We construct a worst-case example of \ref{aed-R1} to demonstrate how an unbounded optimality gap occurs for $k\in [d-1]$.
\begin{example}\label{eg:r1}
Fix $d=n$. Let $\bm a_i = \bm e_i$ for all 
$i\in [n]$, where $\bm e_i$ represents the $i$-th column  of $\bm I_d$.\qedA
\end{example}
In \Cref{eg:r1}, for a given $k\in [d-1]$,  any feasible solution is optimal for  \ref{aed}, since 
\[\tr\bigg(\bigg(\sum_{i\in S} \bm a_i\bm a_i^{\top}+ \lambda \bm I_d\bigg)^{-1}\bigg) = \tr\bigg(\bigg(\sum_{i\in S} \bm e_i\bm e_i^{\top}+ \lambda \bm I_d\bigg)^{-1}\bigg) = \frac{k}{1+\lambda} + \frac{d-k}{\lambda},  \]
holds for any subset $S\subseteq [n]$ with $|S|=k$.
Hence, the optimal value of \ref{aed} is $z_k={k}/{(1+\lambda)}+{(d-k)}/{\lambda}$. On the other hand, the vector $\hat{\bm x}$ is feasible for \ref{aed-R1}, where let $\hat{x}_i=k/n$ for any $i\in [n]$. Plugging $\hat{\bm x}$ into \ref{aed-R1}, we get $$\nu_k^{\text{C}}\le \tr\bigg(\bigg(\sum_{i\in [n]} \frac{k}{n} \bm a_i\bm a_i^{\top}+ \lambda \bm I_d\bigg)^{-1}\bigg) = \tr\left( \frac{k}{n} \bm I_d+ \lambda\bm I_d \right) = \frac{dn}{k+n\lambda} \le \frac{dn}{k} .$$
It follows that $$\frac{z_k}{\nu_k^{\text{C}}}\ge \frac{\frac{k}{1+\lambda}+\frac{d-k}{\lambda}}{\frac{dn}{k}}\ge \frac{\frac{d-k}{\lambda}}{\frac{dn}{k}} = \frac{k(d-k)}{dn\lambda} \to \infty, \text{ if } \lambda\to 0.$$
We thus complete the proof.
\qed
\end{proof}

\subsection{Proof of \Cref{cor:amesp_gap}}
\label{proof:amesp_gap}
\begin{proof}{Proof.}
The proof includes two parts.
\begin{enumerate}[(i)]
    \item We first establish a-${(d-1)}/{(d-k)}$ optimality gap for any $k\in [d-1]$. The proof includes two steps- deriving a lower bound of ${\nu}^{\text{M}}_k$ and an upper bound of $z_k$.

\noindent \textbf{Step I.} Let $\hat{\bm x}$ be an optimal solution to \ref{amesp-R}. Suppose that $\beta_1\ge \cdots \ge \beta_n\ge 0$ denote the eigenvalues of $\sum_{i\in [n]}\hat x_i \bm h_i\bm h_i^{\top}$. 
Then, we have that
\begin{equation}\label{eq:trace}
\begin{aligned}
&\sum_{i\in [n]} \beta_i =\tr\bigg(\sum_{i\in [n]} \hat x_i \bm h_i\bm h_i^{\top}\bigg)= \tr\left(\bm H\Diag(\bm x) \bm H^{\top}\right)= \tr\left(\Diag(\sqrt{\hat{\bm x}})\bm H^{\top}\bm H\Diag(\sqrt{\hat{\bm x}})\right)\\
&=\tr\left(\Diag(\sqrt{\hat{\bm x}})(\bm A^{\top}\bm A+\lambda\bm I_n) \Diag(\sqrt{\hat{\bm x}})\right)=\tr\bigg(\sum_{i\in [n]}\hat x_i \bm a_i\bm a_i^{\top} \bigg) + \lambda \tr\left(\Diag({\hat{\bm x}})\right)\\
&=\sum_{i\in[n]}\hat x_i\|\bm a_i\|_2^2 + k\lambda,
\end{aligned}
\end{equation}
where the third and the fifth equations are from the cyclic property of the trace operator, and the fourth equation is from the definition of $\bm H^{\top}\bm H$.

Next, we show that
\begin{equation*}
\begin{aligned}
 {\nu}^{\text{M}}_k -\frac{d-k}{\lambda} &= \Phi \bigg(\sum_{i\in [n]} \hat x_i \bm h_i\bm h_i^{\top} \bigg) = \sum_{i\in [\iota]} \frac{1}{\beta_i} + \frac{k-\iota}{\frac{1}{k-\iota}\sum_{i\in [\eta+1, n]} \beta_i} \ge  \frac{k^2}{\sum_{i\in [n]} \beta_i} = \frac{k^2}{\sum_{i\in[n]}\hat x_i\|\bm a_i\|_2^2 + k\lambda} \\
   &\ge  \frac{k^2}{k\max_{i\in [n]} \|\bm a_i\|_2^2 + k\lambda}  = \frac{k}{\max_{i\in [n]} \|\bm a_i\|_2^2 + \lambda},
\end{aligned}
\end{equation*}
where the second equation is from \Cref{def:phi} of $\Phi$, the first inequality is from \Cref{ineq:amhm},  the third equation is because of  \eqref{eq:trace}, and the second inequality stems from the fact that $\sum_{i\in [n]} \hat{x}_i=k$.

\noindent \textbf{Step II.} As noted in \Cref{remark:mono}, $z_k\le z_1$ holds. We thus get
\begin{equation*}
    \begin{aligned}
        z_k &\le    z_1 = \min_{i\in [n]} \tr \bigg( \bigg( \bm a_i\bm a_i^{\top}+ \lambda \bm I_d\bigg)^{-1} \bigg) = \min_{i\in [n]} \frac{1}{\|\bm a_i\|_2^2+\lambda} + \frac{d-1}{\lambda} = \frac{1}{\max_{i\in [n]}\|\bm a_i\|_2^2+\lambda} + \frac{d-1}{\lambda}. 
    \end{aligned}
\end{equation*}
Combing Steps I and II, we obtain that
\[ \frac{z_k}{{\nu}^{\text{M}}_k} \le \frac{\frac{1}{\max_{i\in [n]}\|\bm a_i\|_2^2+\lambda} + \frac{d-1}{\lambda} }{\frac{k}{\max_{i\in [n]}\|\bm a_i\|_2^2 + \lambda} + \frac{d-k}{\lambda}} = \frac{\lambda+ (d-1)(\max_{i\in [n]}\|\bm a_i\|_2^2 + \lambda) }{k\lambda+ (d-k)(\max_{i\in [n]}\|\bm a_i\|_2^2+ \lambda)} \le \frac{d-1}{d-k},\]
where the last step is because for any $c_1\ge c_2$, it is easy to verify that $\frac{c_1+\lambda}{c_2+k\lambda}
\le \frac{c_1}{c_2}$.

 Furthermore,  according to \citet[Theorem 11]{li2024best}, we have that 
$$1\le \frac{z_k - \frac{d-k}{\lambda}}{{\nu}^{\text{M}}_k-\frac{d-k}{\lambda}}\le \min\{k, n-k+1\}.$$
For any $k\in [d]$, since $d-k\ge 0$, it is evident that the inequality still holds after removing the term $-\frac{d-k}{\lambda}$. We thus obtain a $\min\{k, n-k+1\}$ optimality gap for \ref{amesp-R} when $k\in [d]$.

Combining the optimality gaps for any $k\in [d-1]$, we obtain that
\[ \frac{z_k}{{\nu}^{\text{M}}_k} \le \min\left\{\frac{d-1}{d-k}, k, n-k+1\right\}.\]
Notably,  the ratio of $k$ can be omitted because 
\[k-\frac{d-1}{d-k} = \frac{1}{d-k} (kd-k^2-d+1) = \frac{1}{d-k} (k-1)(d-k-1)  \ge 0. \]

\item To prove the result, we construct an example where ${\nu}^{\text{M}}_k<0$, as shown below.
\begin{example}\label{eg:amesp}
Suppose $d=1$, $k=2$,  $n=3$, $\lambda=1$, $\bm a_1=1$,   $\bm a_2=-1$, and $\bm a_3=2$. 
\qedA
\end{example}
In \Cref{eg:amesp}, we observe that $\bm A\bm A^{\top}=\sum_{i\in [n]}\bm a_i\bm a_i^{\top}=6$. Consequently,  the matrix $\bm A^{\top}\bm A\in \S_+^n$ has eigenvalues: $6, 0, 0$, and $\bm C = \bm A^{\top}\bm A + \lambda\bm I_n$ has eigenvalues: $7, 1, 1$.

We construct a feasible solution $\hat{\bm x}$ to \ref{amesp-R} that yields a negative objective value for the particular instance in \Cref{eg:amesp}. Specifically, we set $\hat x_i=k/n=2/3$ for any $i\in [n]$. Thus,  $\sum_{i\in [n]}\hat x_i \bm h_i\bm h_i^{\top}=2/3\bm H\bm H^{\top}$.  
For any matrix $\bm X$, it is established that $\bm X\bm X^{\top}$ and $\bm X^{\top}\bm X$ have the same nonzero eigenvalues. According to this result,  given $\bm H^{\top}\bm H = \bm C$, it is easy to verify that  $2/3\bm H\bm H^{\top}$ has eigenvalues: $14/3, 2/3, 2/3$. Clearly, $14/3>4/3$ holds.
From \Cref{def:phi}, we get $\iota=1$ in this case, and we thus have that
\begin{align*}  \Phi\bigg(\sum_{i\in [n]}\hat x_i \bm h_i\bm h_i^{\top} \bigg)+\frac{d-k}{\lambda}= \frac{3}{14} + \frac{3}{4} -1=-\frac{1}{28}  <0.
\end{align*}
At optimality of \ref{amesp-R}, we must have  ${\nu}^{\text{M}}_k<0$ in \Cref{eg:amesp}.
\end{enumerate}
Thus, we complete the proof.
\qed
\end{proof}

\subsection{Proof of \Cref{lem:socp}}
\label{proof:socp}
\begin{proof}{Proof.}
First, define $\bm B = \bm A \Diag(\sqrt{\bm x})\in \Re^{d\times n}$. The minimization problem in \Cref{lem:socp} can be written as 
\begin{align*}
	\min_{\bm X \in \Re^{d\times n}} \frac{1}{\lambda}\tr\left[\bm X \left(\bm B^{\top}\bm B  + \lambda\bm I_n \right)  \bm X^{\top}\right] - \frac{2}{\lambda} \tr\left(\bm X \bm B^{\top}\right)  + \frac{d}{\lambda},
\end{align*}
which is 
a strongly convex and unconstrained problem over $\bm X$. Its first-order optimality condition gives an optimal solution $\bm X^*$. We get
\[\frac{2}{\lambda}\bm X^* \left( \bm B^{\top}\bm B  + \lambda \bm I_n \right)  = \frac{2}{\lambda} \bm B,\]
which leads to $\bm X^* = \bm B \left( \bm B^{\top}\bm B  + \lambda \bm I_n \right)^{-1}$.
Plugging $\bm X^*$ into the minimization problem above,  the optimal value is equal to
\begin{equation*}
	\begin{aligned}
		-\frac{1}{\lambda} \tr \left[ \bm B \left(  \bm B^{\top} \bm B  + \lambda \bm I_n \right)^{-1} \bm B^{\top}  \right] + \frac{d}{\lambda} = \tr \left[ \left(\bm B\bm B^{\top}  + \lambda \bm I_d \right)^{-1}  \right]	=  \tr \bigg[\bigg(\sum_{i\in [n]} x_i \bm a_i\bm a_i^{\top}+ \lambda \bm I_d\bigg)^{-1} \bigg],
	\end{aligned}
\end{equation*}
where the first equation is due to  the Woodbury matrix identity, i.e., 
\[\left(\bm B\bm B^{\top}  + \lambda \bm I_d \right)^{-1}  = \frac{1}{\lambda} \bm I_d - \frac{1}{\lambda}  \bm B (\bm B^{\top}\bm B + \lambda \bm I_n )^{-1}\bm B^{\top},\]
and the second equation is  because of  the definition of $\bm B$. This completes the proof.
\qed
\end{proof}

\subsection{Proof of \Cref{prop:misocp}}
\label{proof:misocp}
\begin{proof}{Proof.}
According to the identity in \Cref{lem:socp}, we can reformulate \eqref{aed1} as 
	\begin{align}\label{aed_soc1}
	z_k=\min_{\begin{subarray}{c}
		\bm x \in \{0,1\}^n, 
        \bm X \in \Re^{d\times n}
		\end{subarray}
}  \bigg\{\frac{1}{\lambda} \left\|\bm X \Diag\left(\sqrt{\bm x}\right) \bm A^{\top} - \bm I_{d}\right\|_F^2 +\|\bm X\|^2_F: \sum_{i\in  [n]} x_i=k \bigg\},
	\end{align}
    Next, we  show that  \eqref{aed_soc1} and \ref{aed_miscop} can be transformed into each other in an equivalent manner.

    \begin{enumerate}[(i)]
        \item Let $(\bm X^*, \bm x^*, \bm \mu^*)$ be an optimal solution of \ref{aed_miscop}. Then, $\|\bm X^*_i\|_2^2=\mu^*_ix^*_i$ must hold for all $i\in [n]$.  
        
Next, let us construct a matrix $\hat{\bm X}$.        For each $i, j\in [n]$, if $x^*_i=0$, let $\hat{X}_{ij} =0$, and if $x^*_i\neq 0$, let $\hat{X}_{ij} = {X_{ij}^*}/{\sqrt{x}_i}$. Then it is easy to check that
\begin{align*}
    &\|\hat{\bm X}_i\|_2^2=\frac{\|\bm X^*_i\|_2^2}{x_i^*}=\mu^*_i, \forall i\in [n], \quad \text{and}\\
    &\frac{1}{\lambda} \left\|\hat{\bm X} \Diag\left(\sqrt{\bm x^*}\right) \bm A^{\top} - \bm I_{d}\right\|_F^2 +\|\hat{\bm X}\|^2_F = \frac{1}{\lambda} \|\bm X^* \bm A^{\top} - \bm I_{d}\|_F^2 +\sum_{i\in [n]}
	\mu^*_i, 
\end{align*}
which implies that $(\hat{\bm X}, \bm x^*) $ is a feasible solution to \eqref{aed_soc1} with the same objective value as \ref{aed_miscop}.
\item Let $(\hat{\bm X}, \hat{\bm x})$ be an optimal solution of \eqref{aed_soc1}. If $\hat x_i=0$ for any $i\in [n]$, then $\hat{\bm X}_i=\bm 0$, because in this context, $\hat{\bm X}_i$ only affects the Frobenius norm of $\hat{\bm X}$ in the objective function of \eqref{aed_soc1}. Then, we define $\bm X^*$ and $\bm \mu^*$ as $\bm X^*= \hat{\bm X} \Diag(\sqrt{\hat {\bm x}})$ and
$\mu^*_i=  \|\hat{\bm X}_i\|_2^2$ for all $i\in [n]$. We get $\|\bm X^*_i\|_2^2 = \|\hat{\bm X}_i\|_2^2 \hat{x}_i = \mu^*_i  \hat{x}_i$ for each $i \in [n]$, implying that $(\bm X^*, \hat{\bm x}, \bm \mu^*)$ is feasible for \ref{aed_miscop}. Analogous to Part (i),  we can demonstrate that $(\bm X^*, \hat{\bm x}, \bm \mu^*)$ gives the same objective value for \ref{aed_miscop} as \eqref{aed_soc1}. 
    \end{enumerate}
Hence, \ref{aed_miscop} and \eqref{aed_soc1} are equivalent. The analysis above is independent of the binary property of $\bm x$, and thus can be readily extended to show the equivalence of \ref{aed-R1} and \eqref{aed_scop}. \qed 
\end{proof}

\subsection{Proof of \Cref{lem:obj}}
\label{proof:lem_obj}
\begin{proof}{Proof.}
Let $\sigma_1\ge \cdots \ge \sigma_d\ge 0$ denote the eigenvalues of $\sum_{i\in [n]} x_i \bm a_i\bm a_i^{\top}$. Since the rank of $\sum_{i\in [n]} x_i \bm a_i\bm a_i^{\top}\in \S_+^d$ must be bounded by $\tilde k$, we have that $\sigma_{\tilde k+1}=\cdots =\sigma_d=0$. Accordingly, the objective matrix $(\sum_{i\in [n]} x_i \bm a_i\bm a_i^{\top} + \lambda\bm I_d )$ in \eqref{aed1} has eigenvalues $\sigma_1+\lambda, \cdots \sigma_{\tilde k}+\lambda, \lambda, \cdots, \lambda$. Then, we have that
\[ \tr\bigg (\bigg(\sum_{i\in [n]} x_i \bm a_i\bm a_i^{\top} + \lambda\bm I_d \bigg)^{-1}\bigg) = \sum_{i\in [\tilde k]} \frac{1}{\sigma_i+\lambda} + \frac{d-\tilde k}{\lambda} = \Gamma \bigg(\sum_{i\in [n]} x_i \bm a_i\bm a_i^{\top};\lambda \bigg) + \frac{d-\tilde k}{\lambda}. \]
We thus complete the proof.
\qed
\end{proof}

\subsection{Proof of \Cref{remark:nonconvex}}
\label{proof:nonconvex}
\begin{proof}{Proof.}
    Suppose $d=2$, $k=1$, and $\lambda=1$. Given $\bm X_1=\begin{pmatrix}
        1&0\\
        0&0
    \end{pmatrix}$ and $\bm X_2=\begin{pmatrix}
        0&0\\
        0&1
    \end{pmatrix}$, we have that
    \[ \Gamma\left(\frac{1}{2}\bm X_1+ \frac{1}{2}\bm X_2;\lambda\right) = \frac{1}{\frac{1}{2}+1} = \frac{2}{3} > \frac{1}{2} \Gamma(\bm X_1;\lambda) + \frac{1}{2} \Gamma(\bm X_2;\lambda) = \frac{1}{2},  \]
    where the first equation is because $1/2$ is the largest eigenvalue of the matrix ${1}/{2}\bm X_1+ {1}/{2}\bm X_2$. The strict inequality violates the function convexity. Thus, $\Gamma$ is nonconvex. \qed
\end{proof}

\subsection{Proof of \Cref{lem:inv}}
\label{proof:inv}
\begin{proof}{Proof.}
According to \Cref{lem:eta}, the conditions that define $\eta$ are $\sigma_{\eta} > \frac{1}{\tilde k-\eta} \sum_{i\in [\eta+1,d]} \sigma_i \ge \sigma_{\eta+1}.$ 
By adding  $\lambda>0$ to all terms, the inequalities still hold. It follows that $\sigma^+_{\eta} > \frac{1}{\tilde k-\eta} \sum_{i\in [\eta+1,d]} \sigma^+_i \ge \sigma^+_{\eta+1}$. Given $\sigma^+_i=0$ for all $i\in [d+1,s]$, the result follows immediately from replacing $d$ with $s$.
\qed
\end{proof}

\subsection{Proof of \Cref{prop:conv}}
\label{proof:conv}
\begin{proof}{Proof.}
First, we define a new vector $\bm \sigma^+\in \Re_+^d$ where $\sigma^+_i=\sigma_i+\lambda>0$ for any $i\in [\tilde k]$ and $\sigma^+_i=\sigma_i$ for any $i\in [\tilde k+1, d]$. Given $\sigma^+_1\ge \cdots \ge \sigma^+_d\ge 0$, we define a function $\gamma(\bm \sigma^+):\Re_+^d\to \Re_+$ as
\[\gamma(\bm \sigma^+)= \sum_{i\in [\tilde k]} \frac{1}{\sigma_i^{+}}=\sum_{i\in [\tilde k]} \frac{1}{\sigma_i+\lambda}.\]
By \Cref{def:gamma}, the function $\Gamma(\bm X; \lambda)$ depends only on the eigenvalues of $\bm X$ and can be characterized by $\gamma$, that is, $\Gamma(\bm X; \lambda)=\gamma(\bm \sigma^+)$. Then,
 \citet[Theorem 8]{kim2022convexification} implies that
 $$\conv\Gamma(\bm X; \lambda)=\conv\gamma(\bm \sigma^+).$$
This shifts our attention to deriving  the convex envelope of $\conv \gamma$.
 
According to \citet[Theorems 8 and 9]{li2024best}, the convex envelope of $\gamma$ is defined as
\begin{align*}
  \conv \gamma(\bm \sigma^+) =  \sum_{i\in [\ell]} \frac{1}{\sigma_i^+} + \frac{(\tilde k-\ell)^2}{\sum_{i\in [\ell+1, d]} \sigma_i^+},   
\end{align*}
where $0\le \ell \le \tilde k-1$ denotes the unique integer  satisfying $\sigma^+_{\ell}>\frac{1}{\tilde k-\ell}\sum_{i\in[\ell+1,d]} \sigma^+_i\ge \sigma^+_{\ell+1} $ with the convention $\sigma^+_0=\infty$. Applying \Cref{lem:inv} with $s=d$, we get $\ell=\eta$.
 Hence, we can rewrite $  \conv\gamma$ as
\begin{align*}
  \conv \gamma(\bm \sigma^+)   =  \sum_{i\in [\eta]} \frac{1}{\sigma_i^+} + \frac{(\tilde k-\eta)^2}{\sum_{i\in [\eta+1, d]} \sigma_i^+} =  \sum_{i\in [\eta]} \frac{1}{\sigma_i+\lambda} + \frac{(\tilde k-\eta)^2}{\sum_{i\in [\eta+1, d]} \sigma_i + (\tilde k-\eta)\lambda},  
\end{align*}
where the second equation is because of the definition of $\bm \sigma^+$.

Observe that $\conv \Gamma (\bm X; \lambda)$ is a spectral function, since it can  be characterized by $\conv \gamma(\bm \sigma^+)$, a function defined on the eigenvalues of $\bm X$. According to standard results on spectral functions \citep{drusvyatskiy2015variational,lewis1995convex},  any subgradient of $\conv\Gamma(\bm X)$ at $\bm X$ takes the form $\bm Q \Diag(\bm \varsigma)\bm Q^{\top}$, where $\bm \varsigma$ is a subgradient of underlying function (i.e., $\conv \gamma(\bm \sigma^+)$)  on eigenvalues $\bm \sigma$. The fact that $\bm \varsigma\in \partial \conv \gamma(\bm \sigma^+)$ directly follows from \citet[Proposition 8]{li2024best}.  \qed
\end{proof}

\subsection{Proof of \Cref{lem:gamma}}
\label{proof:gamma}
\begin{proof}{Proof.}
Our proof is split into two parts depending on whether the rank of $\bm X$ attains $\tilde k$ or not.  To begin, we let $\sigma_1\ge \cdots \ge \sigma_d \ge 0$ denote the eigenvalues  of $\bm X$. 
\begin{enumerate}[(i)]
        \item The rank of $\bm X$ is equal to $\tilde k$. Without the loss of generality, suppose that  $0\le \eta \le \tilde k-1$ is the largest integer satisfying $\sigma_{\eta}>\sigma_{\eta+1}=\cdots =\sigma_{\tilde k}$, where $\sigma_0=\infty$. Then, given $\sigma_{\tilde k}>\sigma_{\tilde k+1}=\cdots=\sigma_d=0$, it is easy to verify that
        \[\sigma_{\eta}>\frac{1}{\tilde k-\eta}\sum_{i\in[\eta+1,d]} \sigma_i = \frac{1}{\tilde k-\eta}\sum_{i\in[\eta+1,\tilde k]} \sigma_i= \sigma_{\eta+1},\]
        which implies that $\eta$ satisfies the conditions in \Cref{lem:eta}. Then, according to \Cref{prop:conv}, $\conv \Gamma(\bm X; \lambda)$ is equal to
        \begin{align*}
  \conv\Gamma(\bm X; \lambda)  & = \sum_{i\in [\eta]} \frac{1}{\sigma_i+\lambda} + \frac{(\tilde k-\eta)^2}{\sum_{i\in[\eta+1,d]}\sigma_{i}+ (\tilde k-\eta) \lambda} = \sum_{i\in [\eta]} \frac{1}{\sigma_i+\lambda} + \frac{\tilde k -\eta}{\sigma_{\eta+1}+\lambda}  =\sum_{i\in [\tilde k]} \frac{1}{\sigma_i+\lambda}\\  &= \sum_{i\in [d]} \frac{1}{\sigma_i+\lambda} - \frac{d-\tilde k}{\lambda}=\tr((\bm X+\lambda\bm I_d)^{-1}) - \frac{d-\tilde k}{\lambda},      
        \end{align*}
    where the second and third equations are because $\sigma_{\eta+1}=\cdots=\sigma_{\tilde k}>0=\sigma_{\tilde k+1}=\cdots =\sigma_d$.
    
    \item The rank of $\bm X$ is less than $\tilde k$. Let $\eta$ denote its rank, and thus $\eta\le \tilde k-1$. It is clear that $\sigma_{\eta}>0=\sigma_{\eta+1}=\cdots= \sigma_d=0$. Given this, 
 we can readily check that
$\sigma_{\eta}>\frac{1}{\tilde k-\eta}\sum_{i\in[\eta+1,d]} \sigma_i = 0 = \sigma_{\eta+1},$
    indicating that the integer $ \eta$ exactly satisfies the conditions in \Cref{lem:eta}. Then, following the analysis of Part (i), we conclude that
    \begin{align*}
  \conv\Gamma(\bm X;\lambda)  &= \sum_{i\in [\eta]} \frac{1}{\sigma_i+\lambda} + \frac{(\tilde k-\eta)^2}{\sum_{i\in[\eta+1,d]}\sigma_{i}+ (\tilde k-\eta) \lambda}= \sum_{i\in [\eta]} \frac{1}{\sigma_i+\lambda} + \frac{\tilde k-\eta}{\lambda} = \sum_{i\in [\tilde k]} \frac{1}{\sigma_i+\lambda} \\
  & =\tr((\bm X+\lambda\bm I_d)^{-1}) - \frac{d-\tilde k}{\lambda}.       
    \end{align*}
    \end{enumerate}
We thus complete the proof.\qed
\end{proof}

\subsection{Proof of \Cref{them:compare}}
\label{proof:compare}
We begin by introducing  \textit{Schur-convex}, which are critical to proving our results.
\begin{definition}[\citealt{constantine1983schur}]\label{def:schur} A function $g: \Re^n\to \Re$ is Schur-convex if for all $\bm \nu, \bm \mu \in \Re^n$ such that $\bm \mu $ majorizes  $\bm \nu$,
one has that $g(\bm \mu) \ge g(\bm \nu)$. 
\end{definition}
Now we are ready to prove \Cref{them:compare}.
\begin{proof}{Proof.}
To begin, for a vector $\hat{\bm x}\in[0,1]^n$ with $\sum_{i\in [n]}\hat x_i=k$, we let $\bm X=\sum_{i\in [n]} \hat x_i \bm a_i\bm a_i^{\top}$, and let $\sigma_1\ge\cdots\ge \sigma_d\ge0$ denote the eigenvalues of $\bm X$. Let  $0\le \eta\le \tilde k-1$ denote the unique integer from \Cref{lem:eta}, that is, $\sigma_{\eta} > \frac{1}{\tilde k-\eta} \sum_{i\in [\eta+1,d]} \sigma_i \ge \sigma_{\eta+1}$.
The following of the proof includes three parts.  
\begin{enumerate}[(i)]
    \item For any $k\in [d,n]$, the rank of $\bm X$ must be bounded by $\tilde k$. Therefore, the identity $\conv\Gamma (\bm X;\lambda)+{(d-\tilde k)}/{\lambda}=\tr((\bm X+\lambda\bm I_d)^{-1})$ always holds based on \Cref{lem:gamma}, implying that the optimal values of \ref{aed-R2} and \ref{aed-R1} are the same.
    
\item For any $k\in [d-1]$, we begin by establishing a technical result on Schur-convexity. 
    \begin{claim}\label{claim}
     The function $\xi(\bm y)=\sum_{i\in [d]}\frac{1}{y_i+\lambda}:\Re^d_{+}\to\Re$ is Schur-convex.   
    \end{claim}
    \begin{proof}{Proof.}
  It is recognized that every convex and symmetric function is Schur-convex (see \citealt{marshall1979inequalities}).    We observe that the function $\xi$ is convex and permutation-invariant with the entries of $\bm y$. Hence, it is Schur-convex.    \qedA
    \end{proof}

Applying Part (i) of \Cref{prop:conv} with $\tilde k=k$, we can transform $\conv\Gamma (\bm X;\lambda)$ into the function $\xi$:
   \begin{align*}
\conv\Gamma (\bm X;\lambda)  = \sum_{i\in [\eta]} \frac{1}{\sigma_i+\lambda} + \frac{k-\eta}{\frac{1}{k-\eta}\sum_{i\in [\eta+1, d]} \sigma_i + \lambda} = \sum_{i\in [d]}\frac{1}{\hat{\sigma}_i+\lambda} - \frac{d-k}{\lambda}= \xi(\hat{\bm \sigma}) - \frac{d-k}{\lambda},      
   \end{align*}
where  $\hat{\bm \sigma}\in \Re_+^d$ is defined as
    \[\hat \sigma_i = \sigma_i,  \forall i\in[\eta], \quad  \hat \sigma_i = \frac{1}{k-\eta} \sum_{i\in [\eta+1,d]} \sigma_i,   \forall i\in[\eta+1,k],  \quad \hat \sigma_i=0,  \forall i\in[k+1,d].\]
    
Given $\sigma_{\eta}>\frac{1}{k-\eta}\sum_{i\in[\eta+1,d]} \sigma_i\ge \sigma_{\eta+1} $,  we can verify that
\begin{align*}
&\sum_{i\in [\ell]} \hat{\sigma}_i =  \sum_{i\in [\ell]} {\sigma}_i, \forall \ell\in [\eta],  \quad \sum_{i \in [\ell]}  \hat{\sigma}_i \ge \sum_{i\in [\eta]}  \hat{\sigma}_i +(\ell-\eta)\sigma_{\eta+1}     \ge  \sum_{i\in [\ell]} {\sigma}_i, \forall \ell\in [\eta+1,k-1], \\  &\sum_{i \in [\ell]}  \hat{\sigma}_i = \sum_{i\in [d]}  \sigma_{i}     \ge  \sum_{i\in [\ell]} {\sigma}_i, \forall \ell\in [k, d-1], \quad \text{and} \quad  \sum_{i\in [d]}\hat{\sigma}_i = \sum_{i\in [d]}   \sigma_i, 
\end{align*}
which means that $\hat{\bm \sigma}$ majorizes $\bm \sigma$.  By  the Schur-convexity of $\xi$ in \Cref{claim}, we get
\begin{align*}
 \tr((\bm X+\lambda\bm I_d)^{-1})=  \sum_{i\in [d]}\frac{1}{\sigma_i+\lambda} =\xi(\bm \sigma) &\le \xi(\hat{\bm \sigma}) =\conv \Gamma(\bm X;\lambda) + \frac{d-k}{\lambda}. 
\end{align*}
We thus conclude that the objective value of \ref{aed-R2} is no less than that of \ref{aed-R1}.
\item Finally, we prove the dominance of \ref{aed-R2} over \ref{amesp-R}. Given the vector $\hat{\bm x}$, let
$\beta_1\ge \cdots \ge \beta_n$ be the eigenvalues of $\sum_{i\in [n]}\hat x_i\bm h_i\bm h_i$. By \Cref{def:phi},  the objective function of \ref{amesp-R} can be represented as
\[ \Phi\bigg (\sum_{i\in [n]}\hat x_i\bm h_i\bm h_i\bigg) + \frac{d-k}{\lambda} = \phi(\bm \beta) + \frac{d-k}{\lambda}.\]

We construct an alternative variant, $\bm \beta^+\in \Re_+^n$, of the vector $\bm \sigma\in \Re_+^d$ as shown below.
\[\beta^+_i = \sigma_i+\lambda, \forall i\in[k], \ \ \beta^+_i=\sigma_i, \forall i\in[k+1,d], \beta_i^+=0, \forall i\in [d+1, n].\]

In fact, $\bm \beta^+$ majorizes $\bm \beta$, as demonstrated by \citet[Lemma 1]{li2024augmented}. In addition, following \Cref{claim}, we can show that the function $\phi$ from \Cref{def:phi} is also Schur-convex, implying that
\[ \phi(\bm \beta)  \le  \phi(\bm \beta^+)=
 \sum_{i\in [\iota]} \frac{1}{\beta^+_i} + \frac{(k-\iota)^2}{\sum_{i\in [\iota+1, n]} \beta^+_i},\]
where $0\le \iota \le k-1$ is a unique integer such that  $\beta^+_{\iota}>\frac{1}{k-\iota}\sum_{i\in[\iota+1,n]} \beta^+_i\ge \beta^+_{\iota+1} $ with $\beta_0=\infty$.
Next,
to prove the result, it suffices to show that the objective function of \ref{aed-R2} (i.e., $\conv \Gamma(\bm X;\lambda)+(d-\tilde k)/\lambda$) is exactly equal to $\phi(\bm \beta^+)+(d-k)/\lambda$.
\begin{enumerate}
\item $k\in [d]$. As $\tilde k=k$ holds in this context, we can verify that the construction $\bm \beta^+$ follows from the variant of $\bm \sigma$ in \Cref{lem:inv} with $s=n$, implying that $\eta=\iota$. Then, it follows that
\[\phi(\bm \beta^+)=\sum_{i\in [\eta]} \frac{1}{\beta^+_i} + \frac{(k-\eta)^2}{\sum_{i\in [\eta+1, n]} \beta^+_i}=\sum_{i\in [\eta]} \frac{1}{\sigma_i+\lambda} + \frac{(k-\eta)^2}{\sum_{i\in [\eta+1, d]} \sigma_i + (k-\eta)\lambda}=\conv\Gamma(\bm X;\lambda),\]
where the second equation is obtained by substituting the expression of $\bm \beta^+$ and omitting the zero-sum term, i.e., $\sum_{i\in [d+1,n]}\beta_i^+=0$.
\item $k\in [d+1, n]$. Let $r\in [d]$ denote the rank of $\bm X$. In this case, by the construction of $\bm \beta^+$, it is easy to verify that $\beta_r^+> \frac{1}{k-r} \sum_{i\in [r+1,n]} \beta_i^+ = \frac{1}{k-r}(k-r)\lambda= \beta_{r+1}^+$. Since $\iota$ is the unique index satisfying $\beta^+_{\iota}>\frac{1}{k-\iota}\sum_{i\in[\iota+1,n]} \beta^+_i\ge \beta^+_{\iota+1}$, $\iota=r$ must hold. Thus, the function value $\phi(\bm \beta^+)+(d-k)/\lambda $ reduces to
\begin{align*}
\phi(\bm \beta^+)+\frac{d-k}{\lambda} &=  \sum_{i\in [r]} \frac{1}{\beta^+_i} + \frac{(k-r)^2}{(k-r)\lambda} +  \frac{d-k}{\lambda}=  \sum_{i\in [r]} \frac{1}{\sigma_i+\lambda} + \frac{d-r}{\lambda}=\tr((\bm X+\lambda\bm I_d)^{-1})\\
&=\conv\Gamma(\bm X;\lambda),
\end{align*}
where the last equation is from the result of Part (i). 
\end{enumerate} 
 This concludes the proof.
    \qed
\end{enumerate}
\end{proof}

\subsection{Proof of \Cref{them:approx}}
\label{proof:approx}
\begin{proof}{Proof.}
We will prove the approximation ratios and time complexities of \Cref{algo:forward} and \Cref{algo:backward}, respectively.
\begin{enumerate}[(i)]
    \item $k\in [d-1]$. For any sized-$s\in [d-1]$ subset $S$, according to \Cref{lem:techniques}, we have that
    \begin{align}\label{eq:sequen}
  \min_{i\in S} f(S\setminus \{i\}) \le \frac{d-s+s\rho+1}{d-s+s\rho} f(S) \le \frac{d-s+1}{d-s} f(S),      
    \end{align}
where the second inequality is from  $\rho \in (0,1)$.
Suppose that $S^*$ is an optimal solution of \ref{aed}. Here, $|S^*|=k\le d-1$. By
applying \eqref{eq:sequen} with $S=S^*$ and $s=k$,  we have that
\[\min_{i\in S^*} f(S^*\setminus\{i\}) \le  \frac{d-k+1}{d-k} f(S^*) = \frac{d-k+1}{d-k} z_k.\]
Let $i^*=\argmin_{i\in S^*} f(S^*\setminus\{i\})$. Then, we can 
apply \eqref{eq:sequen} again with $S=S^*\setminus\{i^*\}$ and $s=k-1$. By doing this with $s=k, k-1, \cdots, 2$ in a sequential manner, we get 
\[\min_{i\in S^*} f(\{i\}) \le \prod_{s\in [2, k]} \frac{d-s+1}{d-s} f(S^*) = \frac{d-1}{d-k} z_k.\]
Clearly,  the first iteration of \Cref{algo:forward} selects $i^*\in \argmin_{i\in [n]} f(\{i\})$. Thus, the approximation ratio is immediate because $f(S_F)\le \min_{i\in S^*} f(\{i\})$ must hold.  

In \Cref{algo:forward}, there are $k$ iterations in total, and at each iteration,  finding the next point $i^*$ and updating $\bm \Lambda$ take $\O(nd^2)$ and $\O(d^2)$ operations, respectively. Hence, the overall time complexity of \Cref{algo:forward} is $\O(knd^2)$. 

\item $k\in [d]$. For the output $S_B$ of \Cref{algo:backward}, using \Cref{lem:techniques} with $[n]$ and $s=n,n-1, \cdots, k+1$, we have that
\[f(S_B) \le \prod_{s\in [k+1, n]} \frac{s-d+d\rho+1}{s-d+d\rho} f([n]) =  \frac{n-d+1+d\rho}{k-d+1+d\rho} z_n \le \frac{n-d+1}{k-d+1} z_k,\]
where the last inequality is because $\rho>0$ and $z_n\le z_k$.

\Cref{algo:backward} begins with computing the sum of $n$ rank-one matrices of size $d\times d$ and inverting this matrix, which requires $\O(nd^2+ d^3)$ operations. 
Then, \Cref{algo:backward} performs $(n-k)$ iterations in its \textit{for} loop, where each iteration takes $\O(nd^2)$ and $\O(d^2)$ operations, similar to \Cref{algo:forward}. Hence, the time complexity of \Cref{algo:backward}
 is $\O((n-k)nd^2)$. \qed
\end{enumerate}
\end{proof}
\end{APPENDICES}

\end{document}